\documentclass{compositio}

\usepackage{amsmath,amssymb,amsthm}
\usepackage{hyperref}
\usepackage{enumerate}
\usepackage{enumitem}
\usepackage{xfrac}
\usepackage{nicefrac}
\usepackage[all]{xy}
\usepackage{graphicx}
\usepackage{float}
\usepackage{subfigure}
\usepackage{bbm}
\usepackage{tikz}
\usetikzlibrary{decorations.markings}

\usepackage{exscale}
\usepackage{relsize}

\numberwithin{equation}{section}
\numberwithin{figure}{section}

\theoremstyle{definition}
\newtheorem{definition}{Definition}[section]

\newtheorem{remark}[definition]{Remark}

\theoremstyle{plain}
\newtheorem{theorem}[definition]{Theorem}
\newtheorem{proposition}[definition]{Proposition}
\newtheorem{corollary}[definition]{Corollary}
\newtheorem{lemma}[definition]{Lemma}

\title{Siegel--Veech Measures of Convex Flat Cone Spheres}

\author{Kai Fu}
\email{kai.fu@mis.mpg.de}
\address{Max Planck Institute for Mathematics in the Sciences, Leipzig, Germany}

\thanks{}

\begin{document}
\begin{abstract}
A classical theorem of Siegel gives the average number of lattice points in bounded subsets of Euclidean space. Motivated by this, Veech introduced the Siegel--Veech formula for translation surfaces, describing the average number of saddle connections of bounded length. However, no such formula is known for flat surfaces with cone angles that are irrational multiples of $\pi$.

A convex flat cone sphere is the Riemann sphere with a flat cone metric whose cone angles lie in $(0,2\pi)$. In this paper, we extend the Siegel--Veech theory to convex flat cone spheres by introducing generalized transforms that are uniformly bounded on moduli space and give rise to Siegel--Veech measures on the positive real line. We prove that these measures are absolutely continuous and piecewise real analytic. In contrast to translation surfaces, which benefit from $GL(2,\mathbb{R})$ dynamics, our approach instead uses tools from o-minimality.

Finally, we analyze the asymptotic behavior of the Siegel--Veech measure near zero, obtaining results that parallel the computation of Siegel--Veech constants in the translation surface case.
\end{abstract}

\maketitle

\setcounter{tocdepth}{1}
\tableofcontents

\section{Introduction}\label{bigsec:intro}

\subsection{Flat Cone Spheres and Moduli Spaces}\label{sec:introflatspheresandmodulispace}
Let $\underline{x} := (x_1, \ldots, x_n)$ be $n$ distinct labeled points on the Riemann sphere $\mathbb{C}P^1$. A \textbf{flat cone sphere} $X$ is a triple $(\mathbb{C}P^1, \underline{x}, g)$, where $g$ is a conformal flat metric with conical singularities at the points $x_1, \ldots, x_n$. Let $a_i$ denote the cone angle at the singularity $x_i$, and define the \textbf{curvature} at $x_i$ as $k_i := \dfrac{2\pi - a_i}{2\pi}$. The Gauss--Bonnet formula implies that $\sum_{i=1}^n k_i = 2$. 

A flat cone sphere is said to be \textbf{convex} if all curvatures lie in the interval $(0,1)$. Convex polyhedra in Euclidean space $\mathbb{R}^3$ provide examples of convex flat cone spheres. A precise and general formulation is provided in Definition~\ref{def:flatspheres}.

A vector $\underline{k} := (k_1,\ldots,k_n)$ is called a \textbf{curvature vector} if $\sum_{i=1}^n k_i = 2$. Let $\mathbb{P}\Omega(\underline{k})$ denote the moduli space of convex flat cone spheres with fixed curvature vector $\underline{k}$, where two flat spheres are identified if they are related by an isometry and scaling that preserve the labeled singularities. 

Troyanov proved in~\cite{Tro1986} that the following forgetful map
$$
\mathbb{P}\Omega(\underline{k}) \to \mathcal{M}_{0,n}, \quad X = (\mathbb{C}P^1, \underline{x}, g) \mapsto (\mathbb{C}P^1, \underline{x}),
$$
is a bijection, where $\mathcal{M}_{0,n}$ denotes the moduli space of Riemman spheres with $n$ distinct labeled points. In particular, we endow $\mathbb{P}\Omega(\underline{k})$ with the complex structure of $\mathcal{M}_{0,n}$.

The moduli space $\mathbb{P}\Omega(\underline{k})$ has been studied in depth from various perspectives; see, for instance, \cite{DM, thu, McMullen, KN, SmillieTriangulations, sauvaget2024volumesmodulispacesflat, Ngu24}.

In~\cite{thu}, Thurston construct a complex hyperbolic metric on $\mathbb{P}\Omega(\underline{k})$, denoted by $h_{Thu}$, and studied its metric completion. Roughly speaking, the metric completion is stratified, with each boundary stratum isometric to a moduli space of convex flat cone spheres of lower dimension. The boundary strata correspond bijectively to the \textbf{$\underline{k}$-admissible partitions} of $\{1,\ldots,n\}$. We recall the details of this stratification structure in Section~\ref{sec:complexhyper}.

Denote by $\mu_{Thu}$ the smooth measure on $\mathbb{P}\Omega(\underline{k})$ induced by the complex hyperbolic metric $h_{Thu}$. In~\cite{thu}, Thuston showed that $\mu_{Thu}$ is a finite measure. The total volume $\mu_{Thu}\big(\mathbb{P}\Omega(\underline{k})\big)$ was first computed by McMullen~\cite{McMullen}. Later, Koziarz and Nguyen~\cite{KN} introduced an alternative approach to computing these volumes by studying the intersection theory of boundary divisors in the moduli spaces. Furthermore, Sauvaget~\cite{sauvaget2024volumesmodulispacesflat} derived the first explicit volume formulas for moduli spaces of flat cone surfaces in any genus with rational (not necessarily positive) curvatures. Nguyen~\cite{Ngu24} subsequently extended these results to arbitrary real curvatures in the case of genus zero.

\subsection{Counting Geodesics on Flat Cone Spheres}

A fundamental problem in the study of flat cone spheres is to understand the distribution of geodesics. In this article, we are interested in counting the number of geodesics of bounded length.

\smallskip

\noindent{\textit{\textbf{Counting functions.}}} 

We begin by introducing the families of geodesics that are of interest in this study.

\begin{itemize}
    \item A \textbf{trajectory} on a flat cone sphere $X$ is a compact geodesic segment that does not contain any singularity in the interior.
    
    \item A trajectory on $X$ is called a \textbf{saddle connection} if it joins two conical singularities. 
    
    \item A trajectory is called a \textbf{regular closed geodesic} if it is a closed geodesic on $X\setminus\underline{x}$.
\end{itemize}

Two regular closed geodesics are said to be \textbf{parallel} if they bound an immersed flat cylinder containing no singularity. In particular, two parallel closed geodesics have the same metric length. This definition of parallelism defines an equivalence relation on the set of regular closed geodesics.

We define the \textbf{normalized length} of a trajectory $\gamma$ as 
$$\ell_\gamma(X) := \frac{|\gamma|}{\sqrt{\operatorname{Area}(X)}},$$
where $|\gamma|$ is the metric length of $\gamma$ in $X$, and $\operatorname{Area}(X)$ is the area of $X$.

For each positive real number $R$, we define $N^{sc}(X,R)$ and $N^{cg}(X,R)$
as the number of saddle connections and equivalence classes of regular closed geodesics on $X$ of normalized length less than $R$, respectively.

\smallskip

\noindent{\textit{\textbf{{Irrational polygonal billiards.}}}

The study of these counting functions is closely related to the field of polygonal billiards. Let $G$ be a polygon in the Euclidean plane. It is called \textbf{rational} if all of its angles are rational multiples of~$\pi$; otherwise, it is called \textbf{irrational}. A \textbf{billiard path} in $G$ is a finite sequence of segments $(s_1, \ldots, s_m)$ such that each vertex $s_i \cap s_{i+1}$ lies in the interior of an edge of~$G$, and the angles that $s_i$ and $s_{i+1}$ make with this edge are complementary. 

A billiard path can be unfolded into a geodesic on a flat cone sphere as follows. Given a polygon $G$ in the Euclidean plane, consider $G$ together with a reflected copy $G'$, and identify their boundaries. This produces a flat cone sphere, denoted by $X_G$, which we refer to as the \textbf{pillowcase of the polygon $G$}. Further details can be found in~\cite{FoxKer, KZ}. Given a billiard path $(s_1, \ldots, s_n)$ in $G$, we lift $s_1$ to $G$ in $X_G$, then lift $s_2$ to $G'$, and continue in this alternating manner. The resulting path is a geodesic on $X_G$. Therefore, to study billiard dynamics in $G$, one can instead study the geodesics on the associated pillowcase $X_G$.

It is proved in~\cite{RSch06, RSch08, tokarsky2018point} that for a triangle $G$ with angles at most $112.3^\circ$, the pillowcase $X_G$ admits a regular closed geodesic. However, for a general irrational polygon $G$, the existence of a regular closed geodesic remains an open problem.

For saddle connections, Katok proved in~\cite{Kat} that the growth of $N^{sc}(X_G, R)$ is sub-exponential in $R$. Scheglov~\cite{Sch} later provided an explicit sub-exponential upper bound for almost every triangular billiard, but no general explicit upper bound is known for arbitrary polygons. For the lower bound, Hooper constructed in~\cite{HooperLowerBounds} that for any $k \in \mathbb{N}$, there is an irrational polygon $G$ for which $N^{cg}(X_G,R)$ grows at least like $R \log^k R$. For general polygons, however, no general lower bounds are known.

\smallskip

\noindent{\textit{\textbf{Translation surfaces and $d$-differentials.}}}

Many of these problems have been extensively studied in the setting of translation surfaces. 

A $d$-\textbf{differential} $\xi$ on a Riemann surface $M$ is a meromorphic section of the canonical line bundle on $M$. It is of the form $f(z)(dz)^d$ in a conformal chart $z$ on $M$. Note that a $d$-differential induces a flat cone metric on $M$.

When $d=1$, a holomorphic $1$-differential is a holomorphic $1$-form on $M$, usually denoted by $\omega$. We call the pair $(M,\omega)$ a \textbf{translation surface}. We refer the reader to the book by Athreya and Masur~\cite{AthreyaMasur2024} for further background.

Let $\mathcal{H}^{d}(\underline{\mu})$ denote the moduli space of $d$-differentials whose zeroes and poles are prescribed by the partition $\underline{\mu} = (\mu_1, \ldots, \mu_n) \in \mathbb{Z}^n$. 

When the curvatures $\underline{k}$ of a flat cone sphere $X$ are rational multiples of $2\pi$, the flat cone metric is induced by a $d$-differential for some $d\in\mathbb{N}$. It follows that $\mathbb{P}\Omega(\underline{k}) = \mathbb{P}\mathcal{H}^{d}(\underline{\mu})$. Moreover, such a flat cone sphere $X$ admits a canonical finite-degree branched cover by a translation surface. Hence, in the rational case, one could use the theory of trasnlation surfaces to study trajectories.

The moduli space $\mathcal{H}^1(\underline{\mu})$ of translation surfaces admits a $GL^+(2,\mathbb{R})$-action. Many results on saddle connections and regular closed geodesics on translation surfaces are derived by studying this action. Masur proved in~\cite{Mas86} that every translation surface contains a periodic geodesic, and showed in~\cite{Ma1, Ma2} that the number of periodic geodesics of bounded length satisfies quadratic upper and lower bounds. This quadratic growth rate is much sharper than the known subexponential upper bounds for flat cone spheres.

In~\cite{Veech}, Veech established a formula for the average number of saddle connections over the moduli space of translation surfaces, now known as the \textbf{Siegel--Veech formula}. We will revisit this formula in detail in Section~\ref{sec:comparisonsiegelveech}. With the help of this formula, Eskin and Masur~\cite{EM} showed that the number of saddle connections and closed geodesics of length at most $R$ grows asymptotically like $cR^2$ for almost every translation surface in the moduli space. Later, a weak version of the quadratic growth was established in~\cite{EMirMoh} for every translation surface in $\mathcal{H}^1(\underline{\mu})$. In~\cite{aygun2025countinggeodesicsprimeorderkdifferentials}, Aygun determines a similar weak quadratic growth for generic $d$-differentials in $\mathcal{H}^d(\underline{\mu})$, when $d$ is prime and the genus is greater than $2$.

Athreya, Cheung, and Masur~\cite{athreya2019siegel} developed a generalized Siegel--Veech formula for counting pairs of saddle connections on translation surfaces. Later, Athreya, Fairchild, and Masur~\cite{Athreya2023} used the formula to study the asymptotic growth of the number of such pairs.\newline

While these results provide a comprehensive understanding of counting functions for translation surfaces, studying counting functions for individual convex flat cone spheres remains challenging. Motivated by the success of Siegel--Veech formulas in the translation surface setting, we aim to establish an analogous formula for convex flat cone spheres. This allows us to study the statistical properties of counting functions in the moduli space.

\subsection{Curvature Gap}
To state the main result, we restrict our attention to curvature vectors with positive curvature gap. The notion of \textbf{curvature gap} for a curvature vector $\underline{k}$ was introduced in~\cite{fu2023boundssaddleconnectionsflat} as
\begin{equation}\label{equ:curvaturegap}
    \delta(\underline{k}) = \min_{I \subset \{1, \ldots, n\}} \left| 1 - \sum_{i \in I} k_i \right|.
\end{equation}

The positivity of the curvature gap of a flat cone sphere $X$ forbids the existence of \textbf{simple} regular closed geodesics on $X$, that is, regular closed geodesics without self-intersections. If $X$ admits a simple regular closed geodesic $\gamma$, then $\gamma$ divides $X$ into two domains, and the Gauss--Bonnet formula implies that the singularities in each domain have total curvature equal to one. Hence, the curvature gap of $X$ is zero. Furthermore, according to~\cite{thu}, the compactification $\overline{\mathbb{P}\Omega}(\underline{k})$ is compact if and only if the curvature gap $\delta(\underline{k})$ is positive. In this case, the metric completion yields a compactification of $\mathbb{P}\Omega(\underline{k})$.

\begin{remark}
    Note that zero curvature gap imposes affine subspaces 
$$
Z_I := \Big\{ \underline{k} \in (0,1)^n \,\mid\, \sum_{i \in I} k_i = 1 \Big\}, \quad I \subset \{1,\ldots,n\},
$$
on the space of curvature vectors $\left\{ \underline{k} \in (0,1)^n \,\middle|\, \sum k_i = 2 \right\}$. For any curvature vector $\underline{k}$ lying in the same connected component of the complement $\bigcup_{I \subset \{1,\ldots,n\}} Z_I$, Thurston's compactification $\overline{\mathbb{P}\Omega}(\underline{k})$ is structurally similar, in the sense that the set of $\underline{k}$-admissible partitions are the same. Moreover, the total volume of $\mathbb{P}\Omega(\underline{k})$ is continuous as a function of the curvature vector, but not continuously differentiable precisely along the subspaces $Z_I$; see~\cite{McMullen,KN}.
\end{remark}

\subsection{Siegel--Veech Transforms}
The Siegel--Veech transform was introduced by Veech in~\cite{Veech} in the context of translation surfaces. We adapt this notion to convex flat cone spheres as follows. We discuss the differences between Veech's definitions and the following definitions in~Section~\ref{sec:comparisonsiegelveech}, Remark~\ref{rmk:comparetransforms}.

\begin{definition}[Siegel--Veech transforms]\label{def:introsiegelveechtransform}
Let $\underline{k}\in(0,1)^n$ be a curvature vector. Given a compactly supported continuous function $f \in C_c(\mathbb{R}_{>0})$, we define the \textbf{Siegel--Veech transform for saddle connections} by
$$
\hat{f}^{sc}(X) := \sum_{\gamma} f(\ell_{\gamma}(X)),
$$
where the sum is taken over all saddle connections $\gamma$ on $X \in \mathbb{P}\Omega(\underline{k})$.

Similarly, the \textbf{Siegel--Veech transform for regular closed geodesics} is given by
$$
\hat{f}^{cg}(X) := \sum_{\gamma} f(\ell_{\gamma}(X)),
$$
where the sum is over all representatives of equivalence classes of regular closed geodesics on $X \in \mathbb{P}\Omega(\underline{k})$.
\end{definition}

Note that the above definition of transforms can be extended to larger function spaces; for instance, to the space of compactly supported bounded measurable functions. For example, if $f$ is the indicator function $\mathbbm{1}_{(0,R)}$ on the interval $(0, R)$, then
$$
\hat{\mathbbm{1}}^{sc}_{(0,R)}(X) = N^{sc}(X, R) \quad \text{and} \quad \hat{\mathbbm{1}}^{cg}_{(0,R)}(X) = N^{cg}(X, R).
$$

We prove in Theorem~\ref{thm:integrabilityandfinitedecomp} that if the curvature gap $\delta(\underline{k})$ is positive, then the Siegel--Veech transforms satisfy that
$$
\hat{f}^{sc}, \hat{f}^{cg} \in L^{\infty}\big(\mathbb{P}\Omega(\underline{k}), \mu_{Thu}\big).
$$

An interesting question concerns the case where the curvature gap is zero: what is the integrability of the Siegel--Veech transforms in this setting? More precisely, do they belong to $L^1$, or even to $L^\infty$, over $\mathbb{P}\Omega(\underline{k})$? The current method relies on the uniform length estimates established in~\cite{fu2023boundssaddleconnectionsflat, fu2024uniformlengthestimatestrajectories}, but these estimates fail when the curvature gap is zero. To treat this case, one would need new estimates for the counting functions.

\begin{remark}
In the setting of hyperbolic surfaces, Mirzakhani introduces in~\cite{Mirzakhani2007} geometric functions on $\mathcal{M}_{g,n}$. More precisely, for $X\in \mathcal{M}_{g,n}$, a simple closed curve $\gamma$ and a function $f:\mathbb{R}_{>0}\to \mathbb{R}$, she considers the transform
$$
f_\gamma(X)=\sum_{\alpha} f\bigl(L_X(\alpha)\bigr),
$$
where $L_X(\alpha)$ denotes the hyperbolic length of $\alpha$ on $X$, and the sum is over all simple closed geodesics on $X$ of the same topological type as $\gamma$. One result of~\cite{Mirzakhani2007} is an integration formula for such geometric functions on $\mathcal{M}_{g,n}$ with respect to the Weil--Petersson volume form; see~\cite[Theorem~7.1]{Mirzakhani2007}. In this paper, we study a similar problem for Siegel--Veech transforms on moduli spaces of convex flat cone spheres. See \cite{Wang2025CHWP} for a relation between the Weil--Petersson form and Thurston's complex hyperbolic form on the moduli space of flat cone spheres.
\end{remark}

\subsection{Siegel--Veech Measures}

The integrability of Siegel--Veech transforms allows us to define Radon measures on $\mathbb{R}_{>0}$ via the Riesz–Markov–Kakutani representation theorem. We refer to the Radon measures arising from our construction as \textbf{Siegel--Veech measures}, in analogy with the terminology introduced in~\cite{athreya2019siegel}.

\begin{definition}[Siegel--Veech measures]
    Let $\underline{k}\in (0,1)^n$ be a curvature vector with positive curvature gap. The \textbf{Siegel--Veech measure for saddle connections} associated with $\mathbb{P}\Omega(\underline{k})$ is the unique Radon measure on $\mathbb{R}_{>0}$ corresponding to the continuous linear functional on $C_c(\mathbb{R}_{>0})$:
    $$
        f \mapsto \int_{\mathbb{P}\Omega(\underline k)}\hat{f}^{sc}(X)d\mu_{Thu},
    $$
    denoted by $\mu^{sc}_{\underline{k}}$.
    
    Similarly, the \textbf{Siegel--Veech measure for regular closed geodesics} is defined as the unique Radon measure on $\mathbb{R}_{>0}$ corresponding to the continuous linear functional on $C_c(\mathbb{R}_{>0})$:
    $$
    f \mapsto \int_{\mathbb{P}\Omega(\underline k)}\hat{f}^{cg}(X)d\mu_{Thu},
    $$
    denoted by $\mu^{cg}_{\underline{k}}$.
\end{definition}

The continuity of the above linear functionals are shown in Theorem~\ref{thm:integrabilityandfinitedecomp}.

\begin{remark}
    The Siegel--Veech measures introduced here should not be confused with the \textbf{Siegel measures} discussed in Veech's article~\cite{Veech}.
\end{remark}

With these definitions in place, we can now state the main theorem of this paper.

\begin{theorem}\label{thm:main3}
    Let $\underline{k}\in(0,1)^n$ be a curvature vector with $n \ge 4$ and positive curvature gap $\delta(\underline{k})$. Then, the following statements hold:
    \begin{enumerate}
        \item \textbf{Absolute Continuity:}
        The Siegel–Veech measures $\mu^{sc}_{\underline{k}}$ and $\mu^{cg}_{\underline{k}}$ are absolutely continuous with respect to the Lebesgue measure $\mu_{Leb}$ on $\mathbb{R}_{>0}$. That is, there exist Lebesgue measurable functions $\rho^{sc}_{\underline{k}}$ and $\rho^{cg}_{\underline{k}}$ on $\mathbb{R}_{>0}$ such that
$$
            \int_{\mathbb{P}\Omega(\underline{k})}\hat{f}^{sc}(X)d\mu_{Thu} = \int_{\mathbb{R}_{>0}}f(t)\rho^{sc}_{\underline{k}}(t)dt,
$$
        and
$$
            \int_{\mathbb{P}\Omega(\underline{k})}\hat{f}^{cg}(X)d\mu_{Thu} = \int_{\mathbb{R}_{>0}}f(t)\rho^{cg}_{\underline{k}}(t)dt,
$$
        for any $f \in C_c(\mathbb{R}_{>0})$.
        
        \item \textbf{Analyticity:}  
        There exists an infinite partition $0 = t_0 < t_1 < \ldots < t_m < \ldots$ with $\lim\limits_{m \to \infty} t_m = \infty$ such that the functions
        $R \mapsto \mu^{sc}_{\underline{k}}(0, R)$ and  $R \mapsto \mu^{cg}_{\underline{k}}(0, R)$
        are real analytic in each $(t_i, t_{i+1})$.

        \item \textbf{Support:}
        The support of $\mu^{cg}_{\underline k}$ is contained in $\big[\sqrt{\pi\delta(\underline{k})},\infty\big)$.

        \item \textbf{Asymptotics:}  
        As $\varepsilon\to0$, we have
        \begin{equation}\label{equ:mainasym}
            \mu^{sc}_{\underline{k}}(0,\varepsilon)=\sum\limits_{\mathrm{codim}(B)=1} L(P)\cdot\mu_{Thu}(B)\cdot\varepsilon^2 + O(\varepsilon^4),
        \end{equation}
        where the sum is taken over all codimension-$1$ boundary strata $B$ of $\overline{\mathbb{P}\Omega}(\underline{k})$, with each $B$ corresponding to a $\underline{k}$-admissible partition $P$ of $\{1,\ldots,n\}$.
    \end{enumerate}
\end{theorem}

The leading coefficient of $\varepsilon^2$ in~\eqref{equ:mainasym} is explicit. The term $\mu_{Thu}(B)$ in the formula represents the complex hyperbolic volume of the moduli space of convex flat cone spheres associated with the boundary stratum $B$. As mentioned in Section~\ref{sec:introflatspheresandmodulispace}, the total volumes have been explicitly computed in~\cite{McMullen,KN,sauvaget2024volumesmodulispacesflat,Ngu24}.

The number $L(P)$ is defined as follows. A $\underline{k}$-admissible partition $P$ corresponding to a codimension-one boundary stratum is characterized by having a unique non-singleton part $p = \{i_1, i_2\}$ such that $k_{i_1} + k_{i_2} < 1$. The criteria for $\underline{k}$-admissibility and the computation of codimension are given in Section~\ref{sec:complexhyper}. Let $P$ be such a partition, and assume without loss of generality that $k_{i_1} \le k_{i_2}$. Then,
$$
L(P)= \pi(1-k_{i_1}-k_{i_2})\cdot\frac{\sin{(k_{i_1}\pi)}\sin{(k_{i_2}\pi)}}{\sin{(k_{i_1}+k_{i_2})\pi}}\cdot\left[ 1 + \sum_{1\le j \le \lceil 1/(2 - 2k_{i_2}) \rceil - 1}    \left( \frac{1}{2\sin(jk_{i_2}\pi)} \right)^2\right].
$$

The expression~\eqref{equ:mainasym} can be interpreted as describing the asymptotic volume (with multiplicity) of the subset of $\mathbb{P}\Omega(\underline{k})$ consisting of spheres that contain at least one saddle connection of normalized length less than~$\varepsilon$.

More generally, Theorem~\ref{thm:main1} provides asymptotic formulas for the volumes of subsets consisting of spheres with multiple short saddle connections. Our computation parallels the asymptotic volume computation for translation surfaces with at least one short saddle connection in~\cite{EMZ}.

We compare these results with known asymptotics for translation surfaces in Section~\ref{sec:comparisonsiegelveech}.

\begin{remark}
    When the number of singularities is three, we normalize the volume to be $1$ for boundary strata $B$ of dimension zero in the formula~\eqref{equ:mainasym}.
\end{remark}

\begin{remark}
    In the case that the number of singularities is three, the moduli space $\mathbb{P}\Omega(\underline k)$ consists of a single sphere $X$. Note that $X$ is the pillowcase of a Euclidean triangle. In this case, the Siegel--Veech measures $\mu^{sc}_{\underline{k}}$ and $\mu^{cg}_{\underline{k}}$ are atomic. More precisely, they can be expressed as sums of Dirac measures on $\mathbb{R}_{>0}$:
    $$
    \mu^{sc}_{\underline{k}} = \sum_{\gamma} \delta_{\ell_{\gamma}(X)} 
    \quad \text{and} \quad 
    \mu^{cg}_{\underline{k}} = \sum_{\gamma} \delta_{\ell_{\gamma}(X)},
    $$
    where the first sum is taken over all saddle connections on $X$, and the second sum is taken over all representatives of equivalence classes of regular closed geodesics on $X$. In particular, (i) and (ii) do not hold for $n=3$, but (iii) and (iv) still hold.
\end{remark}

\begin{remark}
    When $\underline{k}$ consists of rational numbers, Theorem~\ref{thm:main3} applies to the corresponding moduli space of $d$-differentials $\mathbb{P}\mathcal{H}^{d}(\underline{\mu}) = \mathbb{P}\Omega(\underline{k})$. We refer to~\cite{NguyenVolForm} for a comparison of the complex hyperbolic metric and the canonical volume form on $\mathbb{P}\mathcal{H}^{d}(\underline{\mu})$.
\end{remark}

\subsection{Comparison with the Classical Siegel--Veech Formula}\label{sec:comparisonsiegelveech}
We include this subsection for readers with background in translation surfaces and the classical Siegel--Veech formula.

We recall the Siegel--Veech formula for translation surfaces introduced in~\cite{Veech}. We denote a translation surface by $(M,\omega)$, where $M$ is a Riemann surface and $\omega$ is a holomorphic $1$-form. We denote by $\mathcal{H}^1_{1}(\underline{\mu})$ the locus of area-one translation surfaces in $\mathcal{H}^1(\underline{\mu})$. The moduli space $\mathcal{H}^1(\underline{\mu})$ admits a $GL^+(2,\mathbb{R})$-action, and the locus $\mathcal{H}^1_1(\underline{\mu})$ carries an $SL(2,\mathbb{R})$-action. Throughout the following, let $\mathcal{H}$ denote a connected component of $\mathcal{H}^1_1(\underline{\mu})$.

A key result from~\cite{EMir,EMirMoh,Filip2016} asserts that any ergodic $SL(2,\mathbb{R})$-invariant finite measure on $\mathcal{H}$ is supported on an algebraic subvariety locally defined by linear equations with real coefficients. A canonical example of such a measure is the \textbf{Masur--Veech measure} $\mu_{\mathrm{MV}}$, which is supported on the full stratum $\mathcal{H}$; see~\cite{Masur1982,Veech1982,Masur1991}.

In our setting, the moduli space $(\mathbb{P}\Omega(\underline{k}),\mu_{Thu})$ serves as an analogue of $(\mathcal{H},\mu_{MV})$ in the context of translation surfaces, with $\mu_{Thu}$ playing a role comparable to the Masur-Veech measure.

In Veech's paper~\cite{Veech}, the \textbf{Siegel--Veech transform} is defined for $f\in C_c(\mathbb{R}^2)$ as
$$
\hat{f}(M,\omega) := \sum_{\gamma} f(z_\gamma),
$$
where the sum is taken over all saddle connections on $(M, \omega)$, and $z_\gamma$ denotes the \textbf{period} of $\gamma$, that is, $z_\gamma := \int_\gamma \omega \in \mathbb{C}$.

\begin{remark}[Comparison of Siegel--Veech transforms]\label{rmk:comparetransforms}
    The Siegel--Veech transforms in Definition~\ref{def:introsiegelveechtransform} are defined for functions on $\mathbb{R}_{>0}$ rather than $\mathbb{R}^2$. This distinction arises because parallel transport is not trivial on a convex flat cone sphere, and there is no canonical notion of periods.
\end{remark}

\begin{remark}[Comparison of Integrability]
    Veech proved in~\cite{Veech} that the Siegel--Veech transform $\hat{f}$ belongs to $L^1(\mathcal{H},\mu_{MV})$. Later, Athreya, Cheung, and Masur~\cite{athreya2019siegel} established the stronger result that $\hat{f} \in L^2(\mathcal{H},\mu_{MV})$. However, it was also observed in~\cite{athreya2019siegel} that $\hat{f} \notin L^3(\mathcal{H},\mu_{MV})$.

    In our setting, the Siegel--Veech transforms belong to $L^\infty(\mathbb{P}\Omega(\underline{k}),\mu_{Thu})$ by Theorem~\ref{thm:integrabilityandfinitedecomp}. The proof relies crucially on prior results on uniform length estimate of geodesics in~\cite{fu2023boundssaddleconnectionsflat, fu2024uniformlengthestimatestrajectories}.
\end{remark}

Let $\nu$ be any ergodic $SL(2,\mathbb{R})$-invariant finite measure on $\mathcal{H}$. Consider the linear functional on $C_c(\mathbb{R}^2)$ defined by
$$
f\mapsto \int_{\mathcal{H}}\hat{f} \, d\nu.
$$
Let $\mu_{SV}$ denote the measure on $\mathbb{R}^2$ corresponding to the above linear functional.
The \textbf{Siegel--Veech formula} established in~\cite{Veech} says that
\begin{equation}\label{equ:siegelveechtranslationsurfaces}
    \mu_{SV} = c(\mathcal{H},\nu)\cdot\mu_{Leb}
\end{equation}
where $c(\mathcal{H},\nu)>0$ is the \textbf{Siegel--Veech constant}, and $\mu_{Leb}$ denotes the standard Lebesgue measure on $\mathbb{R}^2$.

In~\cite{Veech}, Veech used the $SL(2,\mathbb{R})$-action on $\mathcal{H}$ to deduce that $\mu_{SV}$ is $SL(2,\mathbb{R})$-invariant on the plane and further establish~\eqref{equ:siegelveechtranslationsurfaces}. However, in our setting, there is no such a group action available on the moduli space of convex flat cone spheres. Instead, we introduce new tools from o-minimal theory to establish our generalized Siegel--Veech formulas; see Section~\ref{bigsec:regularity}.

Since $\mu_{SV}(0,\varepsilon) = c(\mathcal{H},\nu) \cdot \varepsilon^2$, one approach to compute $c(\mathcal{H},\nu)$ is to analyze the asymptotic behavior of the measure $\mu_{SV}(0,\varepsilon)$ as $\varepsilon\to0$.
By studying this asymptotic behavior, Eskin, Masur, and Zorich established formulas for Siegel--Veech constants in~\cite{EMZ}. Furthermore, building on~\cite{EMZ}, Goujard computed in~\cite{Gou} the Siegel--Veech constants for strata of \textbf{half-translation surfaces}, which are Riemann surfaces equipped with holomorphic quadratic differentials.

In our setting, the formula~\eqref{equ:mainasym} describing the asymptotics of $\mu^{sc}_{\underline{k}}(0,\varepsilon)$ for convex flat cone spheres can be viewed as an analogue of the Siegel--Veech constant computation. Indeed, the leading coefficient in formula~\eqref{equ:mainasym} closely resembles the formula given in~\cite{EMZ}, which relates to the volume of the boundary stratum of the moduli space.

\subsection{Regularity of Thurston Measure}

We establish a more general result, Theorem~\ref{thm:main1}, which computes the asymptotic volumes of subsets of spheres with multiple short saddle connections.

This theorem parallels many known results about translation surfaces. Recall that $\mathcal{H}$ denotes a connected component of $\mathcal{H}^1_1(\underline{\mu})$.
For $K > 0$ and $\varepsilon > 0$, let $\mathcal{H}_K^\varepsilon$ denote the subset of translation surfaces in $\mathcal{H}$ that have at least two non-parallel cylinders, each with modulus greater than $K$ and circumference less than~$\varepsilon$. An ergodic $SL(2,\mathbb{R})$-invariant measure $\nu$ on $\mathcal{H}$ is called \textbf{regular} if there exists $K > 0$ such that  
$$
\nu\big(\mathcal{H}_K^\varepsilon\big) = o(\varepsilon^2) \quad \text{as } \varepsilon \to 0.
$$

It is shown in~\cite{EMZ} and~\cite{Masur1991} that the Masur--Veech measure is regular. 
This regularity is needed in the work of
Eskin, Kontsevich, and Zorich~\cite{EKZ}. They further conjectured in~\cite{EKZ} that all ergodic $SL(2,\mathbb{R})$-invariant finite measures $\nu$ on $\mathcal{H}$ are regular. Avila, Matheus, and Yoccoz~\cite{Avila2013} proved this conjecture by establishing a stronger result: letting $\mathcal{H}_{(2)}^\varepsilon$ denote the subset of surfaces in $\mathcal{H}$ with two non-parallel saddle connections of length less than~$\varepsilon$, they showed that
$$
\nu\big(\mathcal{H}_{(2)}^\varepsilon\big) = o(\varepsilon^2).
$$
Nguyen~\cite{Nguyen2012} provided stronger estimates for certain invariant measures~$\nu$. More recently, Dozier~\cite{Dozier2023} established a stronger form of regularity for arbitrary such~$\nu$, extending the result to subsets of surfaces with multiple short saddle connections.

Our result, Theorem~\ref{thm:main1}, can be viewed as a strengthened form of regularity for the Thurston measure on $\mathbb{P}\Omega(\underline{k})$. We explicitly compute the asymptotic volume of the subset consisting of spheres with multiple short saddle connections and provide an estimate on the error term.\newline

An interesting direction for future work is to study the asymptotic geometry of convex flat cone spheres as the number of singularities tends to infinity. The asymptotic formula~\eqref{equ:mainasym} might be useful for studying the local geometry of convex flat cone spheres with a large number of singularities. This is inspired by recent studies on large-genus translation surfaces; see~\cite{masur2024lengthssaddleconnectionsrandom} and~\cite{bowen2025benjaminischrammlimitshighgenus}.

\subsection{Proof Strategy}

We outline here the main ideas underlying the proofs of our results.

We first establish the integrability of the Siegel--Veech transforms in Theorem~\ref{thm:integrabilityandfinitedecomp}. The key input comes from earlier results~\cite{fu2023boundssaddleconnectionsflat,fu2024uniformlengthestimatestrajectories}, which provide uniform upper bounds on the number of saddle connections of bounded length.

Theorem~\ref{thm:integrabilityandfinitedecomp} also shows that the Siegel--Veech measure decomposes into a sum of Radon measures associated with \textbf{Delaunay regions}. The notion of Delaunay triangulations on flat cone spheres was introduced in~\cite{Delaunay, Masur1991, thu}. A Delaunay region is a subset of the moduli space consisting of flat cone spheres sharing the same Delaunay triangulation (see Section~\ref{sec:delaunay}).

To prove that the Siegel--Veech measures are piecewise real analytic and absolutely continuous, a crucial step is to show that these Delaunay regions are semialgebraic in local coordinate charts of the moduli space, as established in Theorem~\ref{thm:semialgdelaunayregion}. This semialgebraic structure allows us to apply tools from o-minimal geometry to deduce the desired regularity properties.

To analyze the asymptotic behavior of Siegel--Veech measures near zero, we estimate the volumes of neighborhoods of boundary strata in the metric completion of the moduli space; see Theorem~\ref{thm:main1}. A main difficulty is the construction of explicit coordinate charts near these boundary strata. We recall Thurston’s surgery from~\cite{thu} to parameterize neighborhoods of codimension-one boundary strata. Generalized Thurston surgeries, introduced in~\cite{fu2024uniformlengthestimatestrajectories}, are further used to construct coordinate charts near higher-codimension boundary strata.

Finally, we approximate the complex hyperbolic metric by a product metric (Proposition~\ref{prop:metricapproximation}) in the constructed coordinate charts. The proof of Theorem~\ref{thm:main1} is completed by induction on the codimension of the boundary strata.

\subsection{Organization}

We conclude this introduction with an outline of the paper.

In Section~\ref{bigsec:moduli}, we review the necessary background on moduli spaces of flat cone spheres.  

In Section~\ref{bigsec:siegelveechmeasure}, we introduce generalized Siegel--Veech transforms for convex flat cone spheres, prove their $L^\infty$ integrability, and define the associated Siegel--Veech measures.  

In Section~\ref{bigsec:regularity}, we prove that Siegel--Veech measures are absolutely continuous and piecewise real analytic.  

In Section~\ref{bigsec:convexinfinite}, we discuss infinite-area flat spheres, which arise naturally in the study of the boundary of the moduli space.  

In Section~\ref{bigsec:coor}, we construct coordinate charts near the boundary strata of the moduli space, which play a crucial role in the volume asymptotics.  

In Section~\ref{bigsec:measurebounds}, we prove Theorem~\ref{thm:main1}, which establishes the volume asymptotics for subsets of moduli space with multiple short saddle connections.  

\begin{acknowledgements}
I am grateful to my advisors Vincent Delecroix and Elise Goujard for introducing me to this problem and for many insightful discussions and invaluable guidance. I thank Yohan Brunebarbe for help with o-minimal tools, Duc-Manh Nguyen for many helpful discussions, and Guillaume Tahar for helpful comments on a previous version of this manuscript. I also thank the anonymous referee for useful comments and suggestions.
\end{acknowledgements}

\section{Notation and Background}\label{bigsec:moduli}

In this section, we present the necessary background, following~\cite{Veech, thu, fu2024uniformlengthestimatestrajectories}. 

We begin by defining some fundamental concepts related to curves on a surface $M$. A \textbf{curve} on $M$ is a continuous map from $[0,1]$ to $M$, while a \textbf{loop} is a continuous map from $\mathbb{S}^1$ to $M$. A curve is called \textbf{simple} if it is injective on the open interval $(0,1)$, and a loop is \textbf{simple} if it is injective on $\mathbb{S}^1$. Let $\underline{x} = (x_1, \ldots, x_n)$ denote a set of distinct labeled points on $M$. An \textbf{arc} on $(M, \underline{x})$ is defined as a curve that begins and ends at labeled points, with no labeled points in its interior, and a \textbf{loop} on $(M, \underline{x})$ is defined as a loop on $M$ with no labeled points in its interior.

\subsection{Flat Spheres}
\begin{definition}\label{def:flatspheres}
Let $M$ be a conformal sphere. Let $\underline{k}=(k_1, \ldots, k_n) \in \mathbb{R}^n$ be a real vector and let $\underline{x}=(x_1, \ldots, x_n)$ denote a set of $n$ distinct labeled points in $M$. A \textbf{flat metric} $g$ on $M$ with curvature $k_i$ at $x_i$ for $1\le i\le n$ is a conformal metric on $M\setminus\{x_1,\ldots,x_n\}$ that satisfies the following:
\begin{itemize}
    \item There is a conformal chart $z$ in a neighborhood of $x\in M\setminus\{x_1,\ldots,x_n\}$ such that $z(x)=0$ and $g$ is equal to $|dz|^2$. We refer to such a chart as a \textbf{flat chart at $x$}.
    \item For each $i$, there is a conformal chart $z$ in a neighborhood of $x_i$ such that $z(x_i) = 0$ and $g$ is equal to $|z|^{-2k_i}|dz|^2$.
\end{itemize}
We refer to the triple $(M, \underline{x}, g)$ as a \textbf{flat sphere}. The points $x_1, \ldots, x_n$ are called \textbf{singularities} of $X$ and $\underline{k}$ is called the \textbf{curvature vector} of $X$.
\end{definition}
A singularity is called \textbf{conical} if its curvature is less than $1$, and \textbf{polar} otherwise. For a conical singularity $x_i$, its neighborhood is isometric to a region around the apex of an infinite cone with cone angle $2\pi(1 - k_i)$. In contrast, if $x_i$ is a pole, its neighborhood is isometric to an open end of an infinite cone with cone angle $2\pi(k_i - 1)$. Note that the Gauss-Bonnet formula implies that $\sum_{i=1}^{n} k_i = 2$.

A \textbf{flat cone sphere} is a flat sphere where all singularities are conical. A flat sphere with at least one pole is called an \textbf{infinite flat sphere}. Note that the total area of a flat sphere is finite if and only if all singularities are conical. A flat sphere is called \textbf{convex} if all its curvatures are positive, and \textbf{non-negative} if all its curvatures are non-negative.

We generalize the notion of the curvature gap~\eqref{equ:curvaturegap} to a flat cone sphere $X = (M, \underline{x}, g)$ (possibly with poles) by defining
$$
    \delta(\underline{k}) = \min_{I \subset \{1, \ldots, n\}} \left| 1 - \sum_{i \in I, k_i < 1} k_i \right|.
$$

\begin{lemma}[{\cite[Lemma~2.2]{fu2023boundssaddleconnectionsflat}}]\label{lem:upperboundcurvaturegap}
    The curvature gap $\delta$ of any flat cone sphere satisfies $\delta \leq \frac{1}{3}$.
\end{lemma}

Note that the curvature gap of an infinite flat sphere can be greater than~$\frac{1}{3}$.

Given a flat surface $X$, a \textbf{developing map} associated with $X$ is an orientation-preserving and locally isometric map from $\widetilde{X}$ to the Euclidean plane, where $\widetilde{X}$ is the metric completion of the universal cover of $X \setminus \{x_1, \ldots, x_n\}$ with the pull-back metric of $X$. We usually denote a developing map by $\mathrm{Dev}: \widetilde{X} \to \mathbb{C}$. In addition, given a developing map $\mathrm{Dev}$, the \textbf{holonomy} associated with $X$ and $\mathrm{Dev}$ is the unique morphism $\rho$ from $\pi_1\big(X \setminus \{x_1, \ldots, x_n\}\big)$ to the orientation-preserving isometry group of the plane, such that $\mathrm{Dev}(\gamma \cdot p) = \rho(\gamma) \mathrm{Dev}(p)$, with $\gamma\in \pi_1\big(X \setminus \{x_1, \ldots, x_n\}\big)$ and $p\in \widetilde{X}$. For more information about developing maps and holonomy, consult Section~3.4 of~\cite{Thurston1997}.

Let $X_1$ and $X_2$ be two flat spheres. A homeomorphism $f: X_1 \to X_2$ is called a \textbf{homothety} if it is of the form $az + b$ in the flat charts in $X_1$ and $X_2$, where $a \in \mathbb{C}^*$ and $b \in \mathbb{C}$. Note that the absolute value $|a|$ is independent of the choice of flat charts. The map $f$ is an isometry if and only if $|a| = 1$.

\subsection{Moduli Spaces and Teichm\"uller Spaces}\label{sec:basicdef}

Given a vector $\underline{k} = (k_1, \ldots, k_n) \in \mathbb{R}^n$ such that $\sum_{i=1}^{n} k_i = 2$, we denote by $\mathbb{P}\Omega(\underline{k})$ the \textbf{moduli space} of flat surfaces $X= (M, \underline{x}, g)$ of fixed curvature vector $\underline{k}$, considered up to homothety, where each singularity $x_i$ has curvature $k_i$. Note that we do not require the singularity to be conical.

Assuming that all curvatures are less than one, Troyanov proved in~\cite{Tro1986} that there exists a unique conformal flat metric $g$ on a Riemann surface $(M, \underline{x})$, up to scaling, such that the curvature at $x_i$ is $k_i$. Later, Veech extended this result in~\cite{Veech1993} (Theorem~1.13), showing that the existence and uniqueness also hold for curvatures no less than one.

Therefore, the moduli space $\mathbb{P}\Omega(\underline{k})$ is canonically identified with the moduli space of Riemann spheres with $n$ distinct labeled points $\mathcal{M}_{0,n}$. The following map
$$
\left\{(z_1,\ldots,z_{n-3})\in \mathbb{C}^{n-3} \mid z_i\ne 0, 1,\ \text{and } z_i\ne z_j\ \text{for } i\ne j \right\} \overset{\sim}{\rightarrow} \mathcal{M}_{0,n}
$$
which sends $(z_1,\ldots,z_{n-3})$ to $\big(\mathbb{C}P^1,(0,1,\infty,z_1,\ldots,z_{n-3})\big)$, is a biholomorphism.
We endow $\mathbb{P}\Omega(\underline{k})$ with the complex structure of $\mathcal{M}_{0,n}$.

Let $S$ be a topological sphere and $\underline{s} = (s_0, \ldots, s_n)$ be $n$ distinct labeled points in $S$. We denote by $\mathbb{P}\mathcal{T}_{S, \underline{s}}(\underline{k})$ the \textbf{Teichm\"uller space} of $(S, \underline{s})$. The space consists of the equivalence classes of pairs $(X,f)$, where $X = (M, \underline{x}, g)$ is from the moduli space $\mathbb{P}\Omega(\underline{k})$, and the marking $f:(S,\underline{s}) \to (M, \underline{x})$ is a homeomorphism such that $f(s_i) = x_i$ for all $i$. Two pairs $(X,f)$ and $(X', f')$ are equivalent if and only if there is a homothety $h:X\to X'$ such that $h\circ f$ is isotopic to $f'$ relative to $\underline{s}$.

We denote by $\pi: \mathbb{P}\mathcal{T}_{S, \underline{s}}(\underline{k})\to \mathbb{P}\Omega(\underline{k})$ the canonical projection that forgets markings.\newline

\noindent\textit{\textbf{Convention.}} Throughout this paper, statements of lemmas, propositions, and theorems involving moduli spaces or Teichm\"uller spaces are made under the following convention, unless explicitly stated otherwise. We fix a curvature vector $\underline{k} = (k_1, \ldots, k_n) \in (0,1)^n$, and a topological sphere $(S, \underline{s})$ with $n$ distinct labeled points. Accordingly, we refer to the associated moduli space $\mathbb{P}\Omega(\underline{k})$ of convex flat cone spheres and the Teichm\"uller space $\mathbb{P}\mathcal{T}_{S, \underline{s}}(\underline{k})$.

\subsection{Spanning Tree Coordinate Charts}\label{sec:spanningtree}
We recall the coordinate charts on the moduli space and Teichm\"uller spaces constructed by Thurston in~\cite{thu}.

Let $(S, \underline{s})$ be a sphere with distinct labeled points $\underline{s}$. An \textbf{embedded graph} in $(S, \underline{s})$ is an ordered pair $(V,E)$ consisting of:
\begin{itemize}
    \item A subset $V$ of the singularities $\underline{s}$, whose elements are called \textbf{vertices}.
    \item A set $E$ of \textbf{edges}, where the edges are simple arcs joining vertices with disjoint interiors in $(S, \underline{s})$.
\end{itemize}
An embedded graph $F$ is called an \textbf{embedded forest} in $(S, \underline{s})$ if there is no cycle in $F$. We refer to an embedded forest simply as a forest for conciseness.

A forest $F$ is called a \textbf{tree} if it has exactly one connected component. A tree in $(S, \underline{s})$ is called \textbf{spanning} if it connects all the labeled points.

Two forests $F$ and $F'$ in the sphere $(S, \underline{s})$ are called \textbf{topologically equivalent} if there exists an orientation-preserving homeomorphism $f: (S, \underline{s}) \to (S, \underline{s})$ preserving the labels such that $f(F) = F'$. The \textbf{topological type} of a forest $F$ is defined as the set of all forests on $(S, \underline{s})$ that are topologically equivalent to $F$, denoted by $[F]$.

These notions of forests extend directly to flat spheres $X = (M, \underline{x}, g)$. Furthermore, a forest $F$ in $X$ is called \textbf{geometric} if all its edges are saddle connections.

\smallskip

\noindent\textit{\textbf{Convention}} We always assume that $F$ is oriented and labeled, meaning each edge of $F$ is oriented, and the edges $e_i$ of $F$ are labeled by $i = 0, 1,\ldots, m$.

\smallskip

Given a spanning tree $F$ on $(S, \underline{s})$, we define a subset $D_{S,\underline{s}}(\underline{k};F)$ of $\mathbb{P}\mathcal{T}_{S,\underline{s}}(\underline{k})$ as:
\begin{align*}
    D_{S,\underline{s}}(\underline{k};F) := &\left\{ (X,f)\in \mathbb{P}\mathcal{T}_{S,\underline{s}}(\underline{k}) \mid \text{the equivalence class of } (X,f) \text{ contains a representative} \right.\\
    & \left. \text{such that } f(F)\text{ is a geometric spanning tree.}  \right\}
\end{align*}
For simplicity, we will also use the notation $F$ to refer to the image $f(F)$ on $X$.

A local coordinate chart on $D_{S,\underline{s}}(\underline{k};F)$ is constructed as follows. For $(X,f) \in D_{S,\underline{s}}(\underline{k};F)$, note that $X \setminus F$ is simply connected. Thus, a developing map $Dev$ of $X$ can be induced on $X \setminus F$. Under this map $Dev$, $X \setminus F$ is mapped to a polygon (possibly immersed) in $\mathbb{C}$. Each edge of $f(F)$ corresponds to two vectors on the polygon’s boundary, differing by a rotation. Since $F$ is oriented, we select the vector such that $X \setminus F$ is mapped to the left side of the vector. Denote these vectors corresponding to edges $e_i$ by $z_i$ for $i = 0, \ldots, n-2$. Note that $z_{n-2}$ is a complex linear combination of $z_0, \ldots, z_{n-3}$. 

We call the vector $(z_0, \ldots, z_{n-3}) \in \mathbb{C}^{n-2}$ an \textbf{unprojectivized spanning tree parameter} of $X$ in $D_{S,\underline{s}}(\underline{k};F)$. 
Considering the projection to $\mathbb{CP}^{n-3}$, we define the map
\begin{equation}\label{equ:coordinatemap}
    \phi_F: D_{S,\underline{s}}(\underline{k};F) \to \mathbb{C}P^{n-3}, \quad (X, f) \mapsto [z_0, z_1, \ldots, z_{n-3}].
\end{equation}
Thurston~\cite{thu} demonstrated that this map is a local biholomorphism. Therefore, it defines a local analytic coordinate chart on $D_{S,\underline{s}}(\underline{k};F)$, called a \textbf{spanning tree coordinate chart} associated with the (oriented and labeled) spanning tree $F$. 

\smallskip

\noindent\textit{\textbf{Convention}} When using a spanning tree coordinate chart, we typically select a standard affine chart of $\mathbb{CP}^n$. In other words, we normalize $z_0$ (or possibly another entry) to $1$. Under this normalization, $D_{S,\underline{s}}(\underline{k};F)$ is locally parameterized by $(z_1, \ldots, z_{n-3}) \in \mathbb{C}^{n-3}$, and the coordinate map $\phi_F$ takes values in $\mathbb{C}^{n-3}$:
\begin{equation}\label{equ:coordinatemap2}
    \phi_F \colon D_{S,\underline{s}}(\underline{k};F) \to \mathbb{C}^{n-3}, \quad (X, f) \mapsto  (z_1, \ldots, z_{n-3}).
\end{equation}

Since the projection $\pi: \mathbb{P}\mathcal{T}_{S, \underline{s}}(\underline{k}) \to \mathbb{P}\Omega(\underline{k})$ is a local homeomorphism, the spanning tree coordinate chart extends to the moduli space $\mathbb{P}\Omega(\underline{k})$.

\subsection{Delaunay Regions}\label{sec:delaunay}
We recall the definition of Delaunay triangulations of flat cone spheres. The concept of Delaunay triangulation was introduced by Boris Delaunay in~\cite{Delaunay}, and was later generalized to translation surfaces~\cite{Masur1991} and flat cone surfaces~\cite{thu}.

A \textbf{cell decomposition} $T$ of a sphere $(S, \underline{s})$ with labeled points is a cell decomposition of $S$ such that vertices are the labeled points $\underline{s} = s_1, \ldots, s_n$, and the edges are simple arcs. A cell decomposition is called a \textbf{triangulation} if each face is a triangle.

Let $X=(M,\underline{x},g)$ be a flat cone sphere. A cell decomposition $T$ of $X$ is called \textbf{geometric} if it is a cell decomposition of $(M, \underline{x})$ and the edges are saddle connections.

A disk in a flat cone sphere $X$ is an immersion $d: B(0,r)\to X$ of an open disk of radius $r$ such that $d$ is a local isometry. In particular, no singularity of $X$ lies in the image of $d$. A disk is \textbf{maximal} if it is not properly contained in any other disk. We denote the continuous extension by $\overline{d}:\overline{B}(0,r)\to X$. Let $f$ be a polygonal face of a cell decomposition on $X$ whose edges are saddle connections. We say that $f$ is \textbf{inscribed in a disk} $\overline{d}:\overline{B}(0, r)\to X$ in $X$ if $f$ is the image under $\overline{d}$ of a polygon inscribed in $\overline{B}(0, r)$, and we call such a face $f$ \textbf{Delaunay}.

\begin{definition}\label{def:delaunay}
    A geometric triangulation $T$ of $X$ is a \textbf{Delaunay triangulation} of $X$ if every face of $T$ is Delaunay.

    A flat cone sphere $X$ is called \textbf{Delaunay-generic} if there exists a Delaunay triangulation $T$ of $X$ such that no two faces are inscribed in the same disk in $X$.
\end{definition}
According to~\cite{thu}, there exists at most finitely many Delaunay triangulations on a flat cone sphere. Note that a Delaunay-generic flat cone sphere has a unique Delaunay triangulation.

\begin{definition}
    Let $T$ be a cell decomposition of $(S, \underline{s})$. We define $D_{S,\underline{s}}(\underline{k};T)$ as the subset of $\mathbb{P}\mathcal{T}_{S, \underline{s}}(\underline{k})$ consisting of pairs $(X,f)$ such that the equivalence class of $(X,f)$ has a representative for which $f(T)$ is geometric.

    Assume that $T$ is a triangulation of $(S, \underline{s})$. We define the \textbf{Delaunay region} $D_{S,\underline{s}}^{Del}(\underline{k}; T)$ as the subset of $D_{S,\underline{s}}(\underline{k};T)$ consisting of $(X,f)$ such that $X$ is Delaunay-generic and the equivalence class of $(X,f)$ has a representative such that $f(T)$ is the Delaunay triangulation of $X$.
\end{definition}
By abuse of notation, for each $(X,f)\in D_{S,\underline{s}}^{\mathrm{Del}}(\underline{k};T)$, we also write $T$ for the Delaunay triangulation of $X$.

Let $F$ be a spanning tree consisting of edges of $T$. The Delaunay region $D_{S,\underline{s}}^{Del}(\underline{k}; T)$ is contained in $D_{S,\underline{s}}(\underline{k};F)$. Moreover, we prove later in Lemma~\ref{lem:delaunaychart} that the map~\eqref{equ:coordinatemap} restricted to $D_{S,\underline{s}}^{Del}(\underline{k};T)$ is a bijection. In particular, we have a global spanning tree coordinate chart associated to $F$ on $D_{S,\underline{s}}^{Del}(\underline{k};T)$.\newline

We explain the meaning of Delaunay-genericity in the following.

\begin{definition}
    A subset $A$ of a real analytic $n$-manifold $M$ is called a \textbf{real analytic subset} if, for every $p\in A$, there exists a real analytic chart $(U, \varphi)$ with $\varphi: U \to \mathbb{R}^n$, such that $\varphi(A \cap U)$ is the zero locus of a real analytic function $f:\varphi(U)\to \mathbb{R}$.
\end{definition}

\begin{lemma}\label{lem:genericdelaunay}
Let $\underline{k}\in (\infty,1)^n$ be a curvature vector. The restriction of $\pi: \mathbb{P}\mathcal{T}_{S, \underline{s}}(\underline{k}) \to \mathbb{P}\Omega(\underline{k})$ to any Delaunay region is injective. Moreover, there exist finitely many Delaunay regions 
$$
D_{S,\underline{s}}^{Del}(\underline{k};T_1), \ldots, D_{S,\underline{s}}^{Del}(\underline{k};T_m)
$$ 
in $\mathbb{P}\mathcal{T}_{S, \underline{s}}(\underline{k})$ such that
\begin{itemize}
    \item $\{\pi(D_{S,\underline{s}}^{Del}(\underline{k};T_i))\}_{i=1}^m$ are disjoint and their union is the subset of the Delaunay-generic flat cone spheres in $\mathbb{P}\Omega(\underline{k})$.
    \item The complement of these regions is a real analytic subset.
\end{itemize}
\end{lemma}

To prove the lemma, we utilize the following criterion from~\cite{GS} to determine if a triangular face is Delaunay:

\begin{lemma}[{\cite[Lemma~8.1]{GS}}]\label{lem:delaunaycondition}
    Let $f = (A_0A_1A_2)$ be a triangle embedded in the complex plane, where the vertices $A_0, A_1, A_2$ are arranged in counterclockwise order. Let $B$ be a point in the plane. Denote the complex vector from $A_0$ to $A_i$ by $\alpha_i$ for $i = 1, 2$, and let $\beta$ denote the complex vector from $A_0$ to $B$. Then, the point $B$ does not lie in the interior of the circumcircle of $A_0, A_1, A_2$ if and only if the following inequality holds:
    \begin{equation}\label{equ:delaunaycondition}
        \left|
        \begin{array}{ccc}
        \mathrm{Re}(\alpha_1)  & \mathrm{Im}(\alpha_1) & |\alpha_1|^2 \\
        \mathrm{Re}(\alpha_2)  & \mathrm{Im}(\alpha_2) & |\alpha_2|^2 \\
        \mathrm{Re}(\beta)  & \mathrm{Im}(\beta) & |\beta|^2
        \end{array}
        \right| \ge 0.
    \end{equation}
    Furthermore, the points $A_0,A_1,A_2$ and $B$ are cocircular if and only if the determinant above is zero.
\end{lemma}

\begin{proof}[Proof of Lemma~\ref{lem:genericdelaunay}]
The injectivity of $\pi$ restricted to a Delaunay region follows from the fact that any orientation-preserving homeomorphism of the sphere $(S, \underline{s})$ fixing a triangulation is isotopic to the identity. In particular, two Delaunay regions maps to the same image or disjoint images in  $\mathbb{P}\Omega(\underline{k})$.

Since there are finitely many triangulations on the sphere $(S, \underline{s})$ up to homeomorphisms, we select $T_1,\ldots,T_m$ to be a sequence of triangulations of $(S, \underline{s})$ such that $D_{S,\underline{s}}^{Del}(\underline{k};T_i)$ is non-empty and any Delaunay-generic convex flat cone sphere is contained in one of $\pi(D_{S,\underline{s}}^{Del}(\underline{k};T_i))$, $1\le i\le m$.

Denote by $C$ the complement of $\pi(D_{S,\underline{s}}^{Del}(\underline{k};T_1)), \ldots, \pi(D_{S,\underline{s}}^{Del}(\underline{k};T_m))$ in $\mathbb{P}\Omega(\underline{k})$. Let $X$ be a convex flat cone sphere in $C$. Although the Delaunay triangulations are not unique on $X$, we can consider a cell decomposition  $T'$ on $X$ satisfying that, each face of $T'$ is inscribed in a disk, and different faces are inscribed in different disks. Note that such a cell decomposition is uniquely determined by the flat metric,  and we call it \textbf{Delaunay cell decomposition} of $X$. Any Delaunay triangulation of $X$ is a subdivision of $T'$.

Let $F$ be a spanning tree contained in $T'$. Consider a spanning tree coordinate chart near $X$ associated to $F$. For a face $f$ of $T'$ which has more than three vertices, 
let $A_0,A_1,A_2$ be the vertices of $f$ arranged in counterclockwise order, and let $B$ be another vertex of $f$. We consider the locus in $U$ defined by
$$
P_{A_0,A_1,A_2,B,f}(\alpha_1,\alpha_2,\beta) = \left|
        \begin{array}{ccc}
        \mathrm{Im}(\alpha_1)  & \mathrm{Re}(\alpha_1) & |\alpha_1|^2 \\
        \mathrm{Im}(\alpha_2)  & \mathrm{Re}(\alpha_2) & |\alpha_2|^2 \\
        \mathrm{Im}(\beta)  & \mathrm{Re}(\beta) & |\beta|^2
        \end{array}
        \right| = 0,
$$
where $\alpha_1,\alpha_2,\beta$ are associated to $A_0,A_1,A_2, B$ as in Lemma~\ref{lem:delaunaycondition}. Note that $\alpha_1,\alpha_2,\beta$ are complex linear combination of the spanning tree coordinates in $U$. It follows that $P_{A_0,A_1,A_2,B,f}$ defines a real analytic function on $U$. 

We define $L\subset U$ to be the zero of the product of $P_{A_0,A_1,A_2,B,f}$, where the product is taken over all non-triangular faces $f$ of $T'$ and all vertices $A_0,A_1,A_2,B$ in $f$ defined as above. We show below that $L = C\cap U$ for sufficiently small $U$.

By shrinking $U$, we may assume that $T'$ induces a geometric cell decomposition on any $X'$ in $U$. In other words, $U$ is the projection of a neighborhood contained in $D_{S,\underline{s}}(\underline{k};T')$ in the Teichm\"uller space. 

For each $X' \in U$, subdivide every polygonal face of $T'$ by a Delaunay triangulation of that face, and denote the resulting triangulation by $T''$. Note that $T''$ depends on $X'$. We claim that, after shrinking $U$ if necessary, $T''$ is Delaunay for every $X' \in U$. As illustrated in Figure~\ref{fig:deformdelaunay}, let $f$ be a face of $T''$, and let $f_1$ be a face adjacent to $f$. Write $A_0,A_1,A_2$ for the vertices of $f$, listed in counterclockwise order, and let $B$ be the remaining vertex of $f_1$. 

If $f$ and $f_1$ are contained in the same polygonal face of $T'$ (see the right side of~Figure~\ref{fig:deformdelaunay}), then, by construction, they are triangular faces in a Delaunay triangulation of the polygonal face. Hence inequality~\eqref{equ:delaunaycondition} for $A_0,A_1,A_2,B$ holds on $X'$. Now assume that $f$ and $f_1$ are contained in two distinct faces of $T'$ (see the left side of~Figure~\ref{fig:deformdelaunay}). Note that the strict inequality~\eqref{equ:delaunaycondition} for $A_0,A_1,A_2,B$ holds on $X$. Since this inequality is strict, after shrinking $U$ if necessary, we may assume that it continues to hold for every $X' \in U$. Thus inequality~\eqref{equ:delaunaycondition} holds for every pair of adjacent faces of $T''$ on every $X' \in U$. Therefore, $T''$ is a Delaunay triangulation of $X'$.

\begin{figure}[!htbp]
	\centering
	\includegraphics[width=0.8\linewidth]{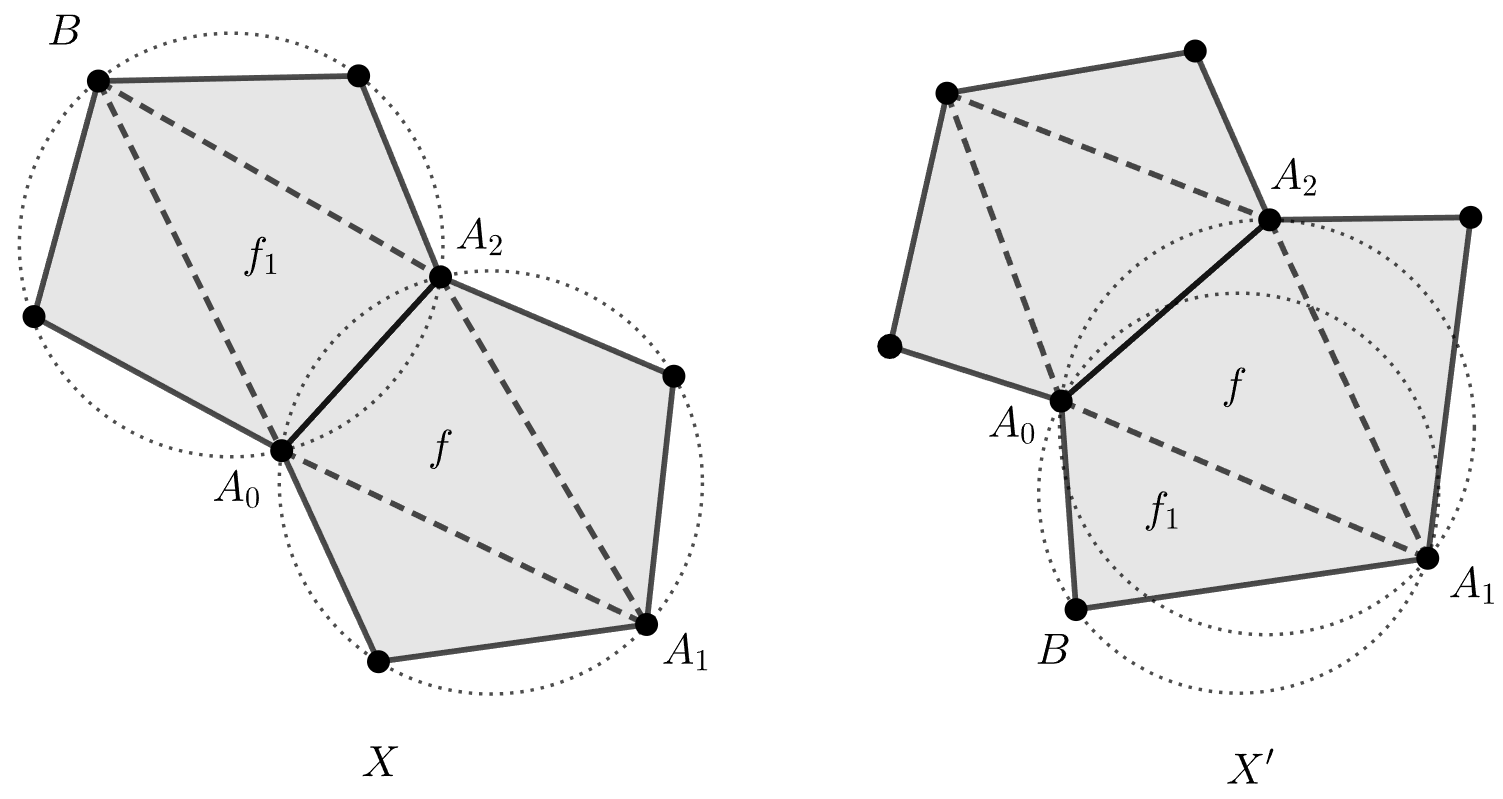}
    \caption{Left: two polygonal faces in the Delaunay cell decomposition of $X$. The two triangle faces $f$ and $f_1$, lying in different polygons, satisfy the strict Delaunay condition~\eqref{equ:delaunaycondition}. Right: the corresponding deformed polygonal faces in a nearby sphere $X'$, where the adjacent triangle faces $f$ and $f_1$ lie in a single polygon and arise from a Delaunay triangulation of the polygon, hence satisfy~\eqref{equ:delaunaycondition}.}\label{fig:deformdelaunay}
\end{figure}

Now, if $X' \in L$, then the equations of $L$ imply that Delaunay triangulation of some polygonal face $f$ of $T'$ is not unique. In particular, we get more than one Delaunay triangulation $T''$ on $X'$ by considering different subdivisions of $T'$. This implies that $X'\in C\cap U$.

On the other hand, for $X'\in C\cap U$, assume that $X'$ is outside $L$. Then each polygon face of $T'$ on $X'$ has a unique Delaunay triangulation. Let $T''$ be the subdivision of $T'$ by Delaunay triangulations of the polygon faces. Then, the strict inequality~\eqref{equ:delaunaycondition} holds for any two adjacent triangles of $T''$. In particular, each face of $T''$ is inscribed in a different immersed disk. Therefore, $X'$ is Delaunay-generic. However, this contradicts the fact that $X'$ is in $C\cap U$. This shows that $C \cap U \subseteq L$.  

Therefore, we prove that the complement $C$ is locally equal to $L$, which is defined as a the zero of a real analytic function. In particular, $C$ is a real analytic subset.
\end{proof}

According to Lemma~\ref{lem:genericdelaunay}, a Delaunay region $D_{S,\underline{s}}^{Del}(\underline{k};T)$ is embedded into the moduli space $\mathbb{P}\Omega(\underline{k})$. We use the same notation for its image in $\mathbb{P}\Omega(\underline{k})$ and call it a \textbf{Delaunay region in $\mathbb{P}\Omega(\underline{k})$}.

\subsection{Complex Hyperbolic Metrics and Metric Completion}\label{sec:complexhyper}
We recall the complex hyperbolic metrics on the moduli spaces of convex flat cone spheres constructed by Thurston in~\cite{thu}. 

We first recall complex hyperbolic spaces. Given a Hermitian form $H$ on $\mathbb{C}^{n+1}$ with signature $(1, n)$, we define
$$
\mathbb{H}_{\mathbb{C}}^n := \{v \in \mathbb{C}^{n+1} \mid H(v,v) > 0\}/\mathbb{C}^{\ast}.
$$
The Hermitian form $-H$ restricted to $\{v \in \mathbb{C}^{n+1} \mid H(v,v) = 1\}$ induces a Hermitian metric $h$ on $\mathbb{H}_{\mathbb{C}}^n$. We call $(\mathbb{H}_{\mathbb{C}}^n, h)$ the \textbf{complex hyperbolic space of dimension $n$}.

In the case of \textbf{convex} flat cone spheres, Thurston constructed in~\cite{thu} a complex hyperbolic metric on the moduli space of flat cone spheres $\mathbb{P}\Omega(\underline{k})$. Let $F$ be a spanning tree on the sphere $(S, \underline{s})$ and let $D_{S,\underline{s}}(\underline{k};F)$ be the corresponding spanning tree coordinate chart. Let $(z_0,\ldots,z_{n-3})\in \mathbb{C}^{n-2}$ be the vectors associated to  $n-2$ edges of $F$. Thurston showed that the area $\mathrm{Area}(X)$ for $(X, f)\in D_{S,\underline{s}}(\underline{k};F)$ is a quadratic form in $(z_0,\ldots,z_{n-3})$ with signature $(1, n-3)$. It follows that $\mathrm{Area}(\cdot)$ induces a complex hyperbolic metric on the spanning tree coordinate chart $D_{S,\underline{s}}(\underline{k};F)$. Thurston showed that this metric does not depend on the choice of spanning tree charts, hence it induces a complex hyperbolic metric on the Teichm\"uller space $\mathbb{P}\mathcal{T}_{S, \underline{s}}(\underline{k})$, denoted by $h_{Thu}$. We denote the induced measure by $\mu_{Thu}$.

Note that the construction of the complex hyperbolic metric does not depend on the marking in $(X, f)$. It follows that it induces a metric on the moduli space $\mathbb{P}\Omega(\underline{k})$. We use the same notation, $h_{Thu}$, for the induced metric, and denote the associated measure on $\mathbb{P}\Omega(\underline{k})$ by $\mu_{Thu}$.

We denote the associated Hermitian form of the quadratic form $\mathrm{Area}(\cdot)$ by $H(\cdot,\cdot)$. It is known that the complex hyperbolic metric $h_{Thu}$ can be expressed concretely by the Hermitian form $H$ as follows.

\begin{lemma}[\cite{Mos}, Section $19$]\label{lem:metricformula}
    Let $\underline{k}\in (0,1)^n$ be a curvature vector.
    In any spanning tree coordinate chart $D_{S,\underline{s}}(\underline{k};F)$, the complex hyperbolic metric $h_{Thu}$ is expressed as
    $$(h_{Thu})_{X}=\frac{1}{\mathrm{Area}(1,Z)^2}\left[|H((0,dZ),(1,z))|^2-\mathrm{Area}(0,dZ)\mathrm{Area}(1,Z)\right]$$
    where $Z = (z_1, \ldots,z_{n-3})$ is the coordinate of $X$ in $D_{S,\underline{s}}(\underline{k};F)$ and $dZ=(dz_1,\ldots,dz_{n-3})$.
\end{lemma}

\begin{remark}
According to Lemma~\ref{lem:genericdelaunay}, we know that almost every convex flat cone sphere in $\mathbb{P}\Omega(\underline{k})$ with respect to the measure $\mu_{Thu}$ is Delaunay generic.
\end{remark}

In general, the moduli space $(\mathbb{P}\Omega(\underline{k}),h_{Thu})$ is not metric complete. In~\cite{thu}, Thurston studied the metric completion of the moduli space of convex flat cone spheres. We denote by $\overline{\mathbb{P}\Omega}(\underline{k})$ the metric completion of $\mathbb{P}\Omega(\underline{k})$ with respect to the metric $h_{Thu}$.

To describe the structure of the metric completion, we introduce the following notions. Let $P$ be a partition of $\{1,\ldots,n\}$. We say that $P$ is \textbf{$\underline{k}$-admissible} if, for any set $p \in P$, we have that
$$
\sum_{i\in p}k_i<1.
$$

\noindent\textit{\textbf{Convention}} We always assume that a partition $P$ is ordered and is written as 
$$P = (p_1, \ldots, p_s, p_{s+1},\ldots, p_r)$$ 
where each elements $p_i$ is written as increasing integers, and $p_1, \ldots, p_s$ are the non-singleton elements in $P$. Moreover, the subsequences $(p_1, \ldots, p_s)$ and $(p_{s+1}, \ldots, p_r)$ are increasingly ordered by the minimal values of $p_i$.\newline

Let $P = (p_1, \ldots, p_r)$ be an ordered partition of $\{1,\ldots,n\}$. We define:
\begin{equation}\label{equ:curvaturenotations0}
    \underline{k}(P) := \Big(\sum_{i \in p_1} k_i, \ldots, \sum_{i \in p_r} k_i \Big).
\end{equation}
In addition, for an element $p_i=(i_1,\ldots,i_m)$ of $P$, we define:
\begin{equation}\label{equ:curvaturenotations1}
    \underline{k}(p_i) := \Big(k_{i_1}, \ldots, k_{i_m}, 2 - \sum_{j \in p_i} k_j \Big)
\end{equation}

\begin{lemma}[{\cite[Theorem~0.2]{thu}}]\label{lem:thurstoncompletion}
    Let $\underline{k}\in (0,1)^n$ be a curvature vector. Then, the metric completion $\overline{\mathbb{P}\Omega}(\underline{k})$ is an $(n-3)$-dimensional complex hyperbolic cone-manifold. Furthermore, the completion can be written as
    \begin{equation}\label{equ:stratification}
        \overline{\mathbb{P}\Omega}(\underline{k}) = \bigsqcup_{P}B(P)
    \end{equation}
    where the union is taken over all $\underline{k}$-admissible partitions of $\{1, \ldots, n\}$ and each $B(P)$ satisfies that
    \begin{itemize}
        \item $B(P)$ is a totally geodesic submanifold.
        \item $B(P)$ is isometric to the moduli space $\mathbb{P}\Omega(\underline{k}(P))$. In particular, the codimension of $B(P)$ in $\mathbb{P}\Omega(\underline{k})$ is $n-|P|$, where $|P|$ denotes the cardinality of the set $P$.
    \end{itemize}
    Furthermore, the ends of $\overline{\mathbb{P}\Omega}(\underline{k})$ correspond bijectively to the partitions of $(k_1, \ldots, k_n)$ into two subsets, each summing to $1$. In particular, the metric completion $\overline{\mathbb{P}\Omega}(\underline{k})$ is compact if and only if the curvature gap $\delta(\underline{k})$ is positive.
\end{lemma}

The notion of cone-manifold is given as follows. Let $X$ be a complete Riemannian manifold and let $G$ be its isometry group (or a subgroup of its isometry group). For any $p\in X$, one denote by $X_p$ the set of geodesic rays emanating from it and $G_p=\mathrm{Stab}_G(p)$ stands for its stabiliser. Then a $(X, G)$-cone-manifold is a space built inductively as follows:
\begin{itemize}
    \item if $X$ is 1-dimensional, a $(X,G)$-cone-manifold is just a $(X,G)$-manifold;
    \item otherwise, a $(X,G)$-cone-manifold is a metric space such that any point in it has a neighborhood isometric to a cone over a $(X_p, G_p)$-cone-manifold.
\end{itemize}
In this paper, we do not rely on the notion of cone-manifold. For a more detailed explanation of this concept, we refer the reader to~\cite{McMullen}.

We call the decomposition~\eqref{equ:stratification} the \textbf{stratification} of $\overline{\mathbb{P}\Omega}(\underline{k})$, and each $B(P)$ is called a \textbf{stratum}. Note that $\mathbb{P}\Omega(\underline{k})$ itself is a stratum corresponding to the \textbf{trivial partition} of $\{1,\ldots,n\}$ into singletons. A \textbf{boundary stratum} of $\overline{\mathbb{P}\Omega}(\underline{k})$ is defined as a stratum corresponding to a non-trivial $\underline{k}$-admissible partition $P$.

\begin{definition}\label{def:singularityassociatedtononsingletonelement}
    Let $B(P)$ be a boundary stratum corresponding to $P$. Let $p_1, \ldots, p_s$ be the non-singleton elements in $P$. For $X$ in the boundary stratum $B(P)$, we denote by $x_{p_i}$ the singularity with curvature $\sum_{j \in p_i} k_j\in \underline{k}(P)$ induced by $p_i$. We say that $x_{p_i}$ is \textbf{associated with the (non-singleton) element $p_i$}.
\end{definition} 

Since we need to consider the Teichmüller space of $B(P)$ in this paper, we introduce the notation $(S,\underline{s}(P))$ to denote the base sphere for the associated Teichmüller space $\mathbb{P}\mathcal{T}(\underline{k}(P))$. Here, $\underline{s}(P)$ represents an ordered tuple of $|P|$ distinct labeled points on~$S$:
\begin{equation}\label{equ:boundarybasesphere}
    \underline{s}(P) := (s_1,\ldots,s_{|P|}).
\end{equation}

\subsection{Thurston Surgeries}\label{sec:thurstonsurgeries}
We recall the Thurston surgery defined in~\cite{thu} and the generalized version developed in~\cite{fu2024uniformlengthestimatestrajectories}.

Let $X$ be a convex flat sphere. Let $\gamma$ be a simple saddle connection joining two distinct conical singularities $x_i$ and $x_j$ with $k_i + k_j < 1$. Thurston defined in~\cite{thu} a surgery to collapse $x_i$ and $x_j$ into a new singularity. More precisely, let $(ABC)$ be a triangle in the plane with angles $\pi(1 - k_i - k_j)$, $\pi k_i$, and $\pi k_j$ at $A$, $B$, and $C$, respectively, and let $|BC| = |\gamma|$. Let $(A'B'C')$ be the mirror image of $(ABC)$ along a line parallel to $AC$. We glue the edge $AC$ to $A'C'$ and glue $AB$ to $A'B'$ by Euclidean isometry, respecting the endpoints. This results in a bounded cone, with the edges $BC$ and $B'C'$ forming the boundary of the cone. By direct computation, the curvature at the apex is $k_i + k_j$. We denote this cone by $C_{\gamma}$.

The \textbf{Thurston surgery of $X$ along $\gamma$} is the operation that slits $\gamma$ open and glues $C_{\gamma}$ in such a way that $x_i$ (resp. $x_j$) is identified with $B$ (resp. $C$). Denote the new flat sphere by $X^{(0)}$. Note that:
\begin{itemize}
    \item $x_i$ and $x_j$ become regular points in $X^{(0)}$.
    \item The apex of $C_{\gamma}$ becomes a new singularity in $X^{(0)}$ with curvature $k_i + k_j$.
\end{itemize}

\begin{figure}[!htbp]
	\centering
	\includegraphics[width=0.9\linewidth]{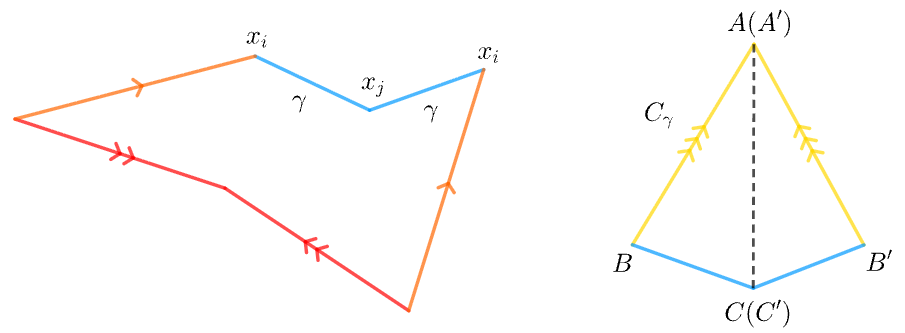}
    \caption{Left: a flat cone sphere with a slit saddle connection $\gamma$; right: a cone $C_\gamma$. The sides are glued by rotations according to the markings on the edges. Thurston surgery glues $C_\gamma$ to the flat cone sphere along $\gamma$.}\label{fig:}
\end{figure}

In~\cite{fu2024uniformlengthestimatestrajectories}, we generalize this surgery to the collision of multiple clusters of conical singularities. 

To describe the surgeries, we first introduce the related notions and results. The \textbf{homotopies} considered on $(S, \underline{s})$ are relative to the set of labeled points $\{s_1, \ldots, s_n\} \subset S$. For an arc or loop $a$ on $(S, \underline{s})$, its homotopy class is denoted by $\alpha$. We extend the class $\alpha$ to include curves on $X$ that are limits of sequences of curves in the original class. In particular, this extension allows representatives of $\alpha$ to pass through singularities. This generalized homotopy class enables us to discuss the existence of shortest representatives in $\alpha$.

\begin{lemma}[{\cite[Lemma~4.5]{fu2024uniformlengthestimatestrajectories}}]
    Let $X$ be a convex flat sphere. Let $\alpha$ be the homotopy class of a simple loop $a$ on $X$. Denote the the singularities in one connected component of $X\setminus a$ by $x_{i_1},\ldots,x_{i_m}$. If a shortest representative of $\alpha$ passes only through the singularities ${x_{i_1},\ldots,x_{i_m}}$, then such a shortest representative is unique.
\end{lemma}

The notion of the core of a meromorphic translation surface was introduced in~\cite{HKK}, and has since been used in several works, including~\cite{tahar2018counting}. We recall here its generalization to flat spheres as developed in~\cite{fu2024uniformlengthestimatestrajectories}.

\begin{definition}[Convex hulls]\label{def:convexhull}
Let $X$ be a convex flat sphere. Let $\alpha$ be the homotopy class of a simple loop $a$. Denote the the singularities in one connected component of $X\setminus a$ by $x_{i_1},\ldots,x_{i_m}$. Assume that $\gamma$ is the unique shortest representative of $\alpha$ which passes only through the singularities ${x_{i_1},\ldots,x_{i_m}}$.

We define the \textbf{convex hull of $x_{i_1}, \ldots, x_{i_m}$ along $\alpha$} to be the closed region enclosed by $\gamma$ that contains the singularities $x_{i_1}, \ldots, x_{i_m}$, and we denote it by $D_{\alpha}$. The convex hull $D_{\alpha}$ is said to be \textbf{degenerate} if $\partial D_{\alpha} = D_{\alpha}$.

Furthermore, if $X$ is a convex flat sphere with one pole, let $\alpha$ be the homotopy class of a loop $a$ around the pole. We define the \textbf{core of $X$} to be the convex hull $D_{\alpha}$ of the conical singularities along $\alpha$.
\end{definition}

Note that the existence of $\gamma$ in Definition~\ref{def:convexhull} implies that the sum of curvatures of $x_{i_1},\ldots,x_{i_m}$ is bounded above by $1$.
 
\begin{lemma}[Lemma~4.10,~\cite{fu2024uniformlengthestimatestrajectories}]\label{lem:uniquereplacementcone}
    Let $X$ be a convex flat sphere and let $D$ be a convex hull of singularities $x_{i_1},\ldots,x_{i_m}$ on $X$. Then there is a unique bounded Euclidean cone $C_D$ with piecewise geodesic boundary such that:
\begin{enumerate}
    \item The curvature of the apex of $C_D$ is $\sum_j k_{i_j}$, where $k_{i_j}$ is the curvature of $x_{i_j}$.

    \item Let $C$ be the infinite cone containing $C_D$. A tubular neighborhood of the boundary of $C\setminus C_D$ is isometric to a tubular neighborhood of $\partial D$ in $X\setminus D$. 
\end{enumerate}     
\end{lemma}

Assume that $D_1, \ldots, D_s$ are disjoint convex hulls on a non-negative flat surface $X$. As in Lemma~\ref{lem:uniquereplacementcone}, let $C_{D_i}\subset C_i$ be the bounded cone in the infinite cone associated to each $D_i$. We define the \textbf{generalized Thurston surgery} of $X$ along $D_1, \ldots, D_s$ as follows:
\begin{itemize}
    \item Cut out each $D_i$ from $X$ and glue $C_{D_i}$ to the surface along $\partial D_i$ for each $i$. Denote the resulting convex flat sphere by $X^{(0)}$.
    \item For each $D_i$, glue the complement $C\setminus C_{D_i}$ to $D_i$ along $\partial D_i$. Denote the resulting infinite flat sphere by $X_i$.
\end{itemize}
We call $X^{(0)}$ the \textbf{top flat sphere} and $X_1, \ldots, X_s$ the \textbf{infinitesimal flat spheres} obtained by the generalized Thurston surgery. We also call $C_{D_1}, \ldots, C_{D_s}$ the \textbf{replacement cones} during the surgery. Note that the conclusion (2) of Lemma~\ref{lem:uniquereplacementcone} implies that $\partial D_i$ is away from singularities on $X^{(0)}$.

\begin{figure}[!htbp]
	\centering
	\includegraphics[width=0.9\linewidth]{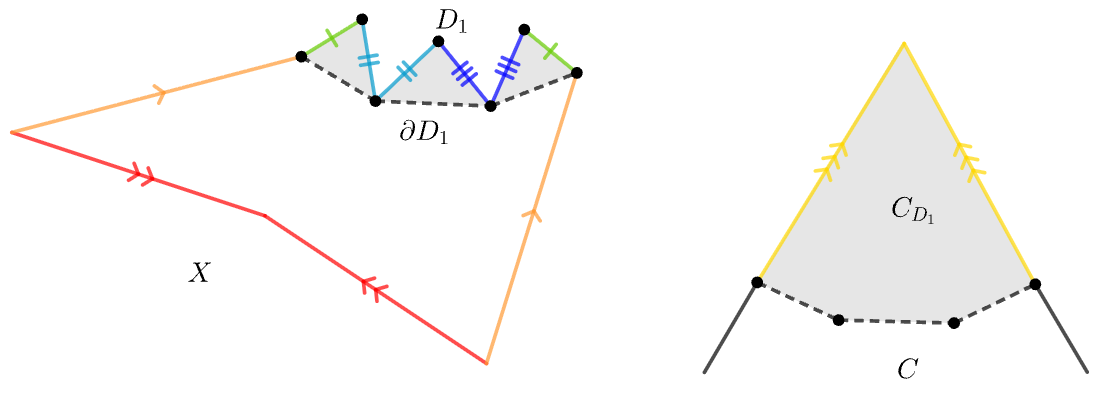}
    \caption{The gray regions on the left form a convex hull $D_1$ in $X$; the generalized Thurston surgery cuts out $D_1$ and glues $C_{D_1}$, shown on the right, to $X\setminus D_1$ along the dotted line $\partial D_1$. The edges in the picture are glued by rotations according to the markings.}\label{fig:1}
\end{figure}

According to the surgery, the core of each infinitesimal flat sphere $X_i$ is metrically isometric to the convex hull $D_i$ in $X$.

\begin{remark}
    Thurston surgery is applied along a saddle connection with distinct endpoints, whereas generalized Thurston surgery is applied along convex hulls. Note that a saddle connection with distinct endpoints forms the convex hull of its endpoints if and only if each endpoint has curvature at most $\frac{1}{2}$.
    
    In this paper, we refer to both the Thurston surgeries and the generalized Thurston surgeries as \textbf{Thurston surgeries}. Whenever a surgery is applied, we will specify whether it is performed along a saddle connection or along convex hulls.
\end{remark}

Let $P$ be a $\underline{k}$-admissible partition of $\{1,\ldots,n\}$, and let $B(P)$ denote the corresponding boundary stratum of $\overline{\mathbb{P}\Omega}(\underline{k})$. Denote the non-singleton elements in the partition $P$ by $p_1, \ldots, p_s$. We say that disjoint convex hulls $D_1, \ldots, D_s$ in $X \in \mathbb{P}\Omega(\underline{k})$ are \textbf{associated with the boundary stratum} $B(P)$ if the labels of the singularities in each $D_i$ are precisely those labels in $p_i$. 

Note that for disjoint convex hulls $D_1, \ldots, D_s$ associated with $B(P)$, the top flat cone sphere $X^{(0)}$ obtained via Thurston surgery is contained in $B(P)$. Furthermore, each infinitesimal flat sphere $X_i$ is contained in the moduli space $\mathbb{P}\Omega(\underline{k}(p_i))$. The curvature vectors $\underline{k}(P)$ and $\underline{k}(p_i)$ are defined in~\eqref{equ:curvaturenotations0} and~\eqref{equ:curvaturenotations1}. 

\begin{remark}\label{rmk:canonicallabels}
To clarify the meaning of the inclusion $X^{(0)} \in B(P)$, we must specify how the singularities of $X^{(0)}$ are labeled. Then the singularity in $X^{(0)}$ corresponding to the apex of $D_i$ is the singularity $x_{p_i}$ associated with a non-singleton element $p_i$. For the remaining singularities, note that $X \setminus \bigsqcup_i D_i$ is embedded into $X^{(0)}$. Then we assign the same labels to these singularities as the corresponding singularities in $X \setminus \bigsqcup_i D_i$. 

Similarly, for infinitesimal flat sphere $X_i$, we assume by default that the conical singularities of $X_i$ are labeled as the corresponding singularities of $D_i$ in $X$. This clarifies the inclusion of $X_i$ in $\mathbb{P}\Omega(\underline{k}(p_i))$.
\end{remark}

\section{Siegel--Veech Measures}\label{bigsec:siegelveechmeasure}

In this section, we focus on the introduction of Siegel--Veech measures.

\subsection{Siegel--Veech Transforms}\label{sec:integrablity}

Let $X$ be a flat cone sphere and let $\gamma$ be a saddle connection or a regular closed geodesic on $X$. Recall that the normalized length of $\gamma$ is defined as $\ell_{\gamma}(X):=\frac{|\gamma|}{\sqrt{\mathrm{Area}(X)}}$, where $|\gamma|$ denotes the metric length of $\gamma$ on $X$.

\begin{definition}[Siegel--Veech transforms]
Let $\underline{k}\in(0,1)^n$ be a curvature vector. Given a compactly supported continuous function $f \in C_c(\mathbb{R}_{>0})$, the \textbf{Siegel--Veech transform for saddle connections} is defined by
\begin{equation}\label{equ:scsvt}
    \hat{f}^{sc}(X) := \sum_{\gamma} f(\ell_{\gamma}(X))
\end{equation}
where the sum is taken over all saddle connections $\gamma$ on $X \in \mathbb{P}\Omega(\underline{k})$.

Similarly, the \textbf{Siegel--Veech transform for regular closed geodesics} is defined by
\begin{equation}\label{equ:cgsvt}
    \hat{f}^{cg}(X) := \sum_{\gamma} f(\ell_{\gamma}(X)),
\end{equation}
where the sum is over all representatives of equivalence classes of regular closed geodesics on $X \in \mathbb{P}\Omega(\underline{k})$.
\end{definition}

Consider the following continuous linear functionals on $C_c(\mathbb{R}_{>0})$:
\begin{equation}\label{equ:lfsc}
    L^{sc}: f \mapsto \int_{\mathbb{P}\Omega(\underline k)}\hat{f}^{sc}(X)d\mu_{Thu}
\end{equation}
and 
\begin{equation}\label{equ:lfcg}
    L^{cg}: f \mapsto \int_{\mathbb{P}\Omega(\underline k)}\hat{f}^{cg}(X)d\mu_{Thu}.
\end{equation}
We defer the proof of the integrability of these Siegel--Veech transforms $\hat{f}^{sc}$ and $\hat{f}^{cg}$, and the continuity of the linear functionals to Theorem~\ref{thm:integrabilityandfinitedecomp}, under the assumption of positive curvature gap.

We define Siegel--Veech measures, based on the Riesz–Markov–Kakutani representation theorem.

\begin{definition}[Siegel--Veech measures]\label{def:svmeasures}
    Let $\underline{k}\in (0,1)^n$ be a curvature vector with positive curvature gap. The \textbf{Siegel--Veech measure for saddle connections} associated to $\mathbb{P}\Omega(\underline k)$ is defined as the unique Radon measure on $\mathbb{R}_{>0}$ corresponding to the linear functional $L^{sc}$, denoted by $\mu^{sc}_{\underline{k}}$.
    
    Similarly, the \textbf{Siegel--Veech measure for regular closed geodesics} associated to $\mathbb{P}\Omega(\underline k)$ is defined as the unique Radon measure on $\mathbb{R}_{>0}$ corresponding to the linear functional $L^{cg}$, denoted by $\mu^{cg}_{\underline k}$.
\end{definition}

Note that it follows directly from the definition that 
\begin{equation}
    \int_{\mathbb{P}\Omega(\underline{k})}\hat{f}^{sc}(X)d\mu_{Thu} = \int_{\mathbb{R}_{>0}} f d\mu^{sc}_{\underline{k}}
\end{equation}
and
\begin{equation}
    \int_{\mathbb{P}\Omega(\underline{k})}\hat{f}^{cg}(X)d\mu_{Thu} = \int_{\mathbb{R}_{>0}} f d\mu^{cg}_{\underline{k}}.
\end{equation}

In the remainder of this section, we prove the integrability of the Siegel--Veech transforms and introduce decompositions of the Siegel--Veech measures, which will be utilized in subsequent discussions.

\subsection{Domains of Geodesics}\label{sec:domain}
Let $\gamma$ be an arc in the sphere $(S, \underline{s})$. Define the following subset of $\mathbb{P}\mathcal{T}_{S, \underline{s}}(\underline{k})$:
\begin{align}\label{equ:domainsc}
    D_{S,\underline{s}}(\underline{k};\gamma) := &\left\{(X, f) \in \mathbb{P}\mathcal{T}_{S, \underline{s}}(\underline{k}) \mid 
    \text{the equivalence class of } (X,f) \text{ contains a representative} \nonumber \right.\\
    &\left. \text{such that } f(\gamma) \text{ is a saddle connection in } X \right\}.
\end{align}
We refer to $D_{S,\underline{s}}(\underline{k};\gamma)$ as \textbf{the domain of} $\gamma$. We define the domain $D_{S,\underline{s}}(\underline{k};\gamma)$ similarly if $\gamma$ is a loop in $(S, \underline{s})$, except that $f(\gamma)$ is realized as a regular closed geodesic in $X$ for $(X, f) \in D_{S,\underline{s}}(\underline{k};\gamma)$.

Let $T$ be a triangulation on $(S,\underline{s})$. Define 
$$
D_{S,\underline{s}}^{Del}(\underline{k};T, \gamma) := D_{S,\underline{s}}^{Del}(\underline{k};T) \cap D_{S,\underline{s}}(\underline{k};\gamma).
$$

By abuse of notation, we also use $\gamma$ to denote the saddle connection or regular closed geodesic $f(\gamma)$ on $X$. 

Note the the normalized length $\ell_{\gamma}(X)=\frac{|\gamma|}{\sqrt{\mathrm{Area}(X)}}$ can be extended to the function on the domain of $\gamma$ as follows:
\begin{equation}\label{equ:normalizedlengthondomainofgeodesics}
    \ell_{\gamma}: D_{S,\underline{s}}(\underline{k};\gamma)\to \mathbb{R}_{>0},\quad (X,f)\mapsto \ell_{f(\gamma)}(X).
\end{equation}

We study in the following the intersections of $D_{S,\underline{s}}^{Del}(\underline{k};T, \gamma)$ in a fixed $D_{S,\underline{s}}^{Del}(\underline{k};T)$. For that, we recall the following results from~\cite{fu2023boundssaddleconnectionsflat, fu2024uniformlengthestimatestrajectories}.

\begin{definition}
    Let $\gamma$ be a trajectory on a flat cone sphere $X$. For a point $p$ on $\gamma$, we say that two tangent vectors to $\gamma$ at $p$ form a \textbf{transverse pair} if they span the tangent space at $p$. The \textbf{self-intersection number} of $\gamma$ is defined as the sum of transverse pairs over all points in the interior of $\gamma$, and we denote it by $\iota(\gamma, \gamma)$.
\end{definition}

\begin{lemma}[{\cite[Proposition 5.8]{fu2023boundssaddleconnectionsflat}}]\label{lem:boundcombinatoriallength}
    Let $X$ be a flat cone sphere with $n$ conical singularities and positive curvature gap $\delta$, and let $T$ be a Delaunay triangulation of $X$. Then, for any saddle connection or a regular closed geodesic $\gamma$ on $X$, the number of triangular faces of $T$ traversed by $\gamma$, counted with multiplicity, is at most $a_1(n,\delta)\sqrt{\iota(\gamma, \gamma)}+a_2(n,\delta)$, where $a_i(n,\delta)$ is an explicit positive number.
\end{lemma}

\begin{lemma}[{\cite[Theorem~1.5]{fu2023boundssaddleconnectionsflat} and~\cite[Theorem~1.6]{fu2024uniformlengthestimatestrajectories}}]\label{lem:lengthestimate}
    Let $n\ge 3$ and $\delta$ be a positive real number. There exists positive constants $c_1,c_2$ such that for any unit area convex flat cone sphere $X$ with $n$ singularities and curvature gap $\delta$, and any trajectory $\gamma$ on $X$, we have that
    \begin{equation}\label{equ:compare}
           \frac{1}{c_1}\sqrt{\iota(\gamma,\gamma)} - c_2 \le |\gamma|\le c_1\sqrt{\iota(\gamma,\gamma)} + c_2.
    \end{equation}
\end{lemma}

Using the above results, we deduce the following finiteness property.
\begin{proposition}\label{prop:finiteboundedlengthgeodesics}
    Let $\underline{k}\in (0, 1)^n$ be a curvature vector with positive curvature gap. Let $T$ be a triangulation on $(S,\underline{s})$. Then, for any given real number $R > 0$, there exist finitely many arcs (resp. loops) $\gamma_1, \ldots, \gamma_m$, up to isotopy relative to $\underline{s}$, on the sphere $(S, \underline{s})$ such that if $\gamma$ is a saddle connection (resp. regular closed geodesic) with normalized length less than $R$ on $X$ for $(X, f) \in D_{S,\underline{s}}^{Del}(\underline{k};T)$, then $(X, f)$ is contained in $D_{S,\underline{s}}^{Del}(\underline{k};T, \gamma_i)$ for some $i$, and $f(\gamma_i)$ is realized as the saddle connection $\gamma$.
\end{proposition}

\begin{proof}
Since $(X, f)$ is contained in $D_{S,\underline{s}}^{Del}(\underline{k};T)$, we know that $T$ is the Delaunay triangulation of $X$. According to Lemma~\ref{lem:lengthestimate}, the self-intersection number of $\gamma$ is bounded above by a constant $c = c(n, \delta, R)$, which depends only on the number of singularities $n$, the curvature gap $\delta > 0$, and the given length bound $R$. Furthermore, by Lemma~\ref{lem:boundcombinatoriallength}, the number of triangular faces traversed by $\gamma$, counted with multiplicity, is also bounded above by a constant $c' =c'(n, \delta, R)$.

Note that there are only finitely many arcs (resp. loops) on $(S, \underline{s})$ with a bounded number of triangular faces of $T$ traversed. Denote these arcs (resp. loops) on $(S, \underline{s})$ by $\gamma_1, \ldots, \gamma_m$. It follows that $\gamma$ coincides with one of $f(\gamma_i)$.
\end{proof}

\subsection{Decomposition of Siegel--Veech Measures}\label{sec:decompofSV}

In this sub-section, we prove the integrability of Siegel--Veech transforms and introduce a decomposition of the Siegel--Veech measures.

We first introduce a family of new measures as follows.
Let $T$ be a triangulation on $(S,\underline{s})$, and let $\gamma$ be an arc (resp. a loop) on $(S,\underline{s})$. 
Define a linear functional on $C_c(\mathbb{R}_{>0})$ as
\begin{equation}\label{equ:lfpath}
    g \mapsto \int_{D_{S,\underline{s}}^{Del}(\underline{k};T, \gamma)}g\circ\ell _{\gamma}(X,f)d\mu_{Thu},
\end{equation}
where $\ell_{\gamma}$ is a function of normalized length~\eqref{equ:normalizedlengthondomainofgeodesics} on $D_{S,\underline{s}}^{Del}(\underline{k};T, \gamma)$. Note that for $(X,f)\in D_{S,\underline{s}}^{Del}(\underline{k};T,\gamma)$, the image $f(\gamma)$ is realized as a saddle connection on $X$. By Lemma~\ref{lem:lengthestimate},
$$
\ell_\gamma(X,f)=\ell_{f(\gamma)}(X)\le c_1\sqrt{\iota\left(f(\gamma),f(\gamma)\right)}+c_2 = c_1\sqrt{\iota(\gamma,\gamma)}+c_2.
$$
It follows that $\ell_\gamma$ is uniformly bounded on $D_{S,\underline{s}}^{Del}(\underline{k};T,\gamma)$.

By Riesz–Markov–Kakutani representation theorem, we denote by $\mu_{T,\gamma}$ the unique Radon measure on $\mathbb{R}_{>0}$ corresponding to the linear functional~\eqref{equ:lfpath}.

Note that the integral in~\eqref{equ:lfpath} is well-defined. This follows from the fact that $D_{S,\underline{s}}^{Del}(\underline{k};T)$ is embedded in the moduli space $\mathbb{P}\Omega(\underline{k})$, and thus the volume of $D_{S,\underline{s}}^{Del}(\underline{k};T)$ is finite.

We now prove the main result of this section.

\begin{theorem}\label{thm:integrabilityandfinitedecomp}
Let $\underline{k}\in (0,1)^n$ be a curvature vector with positive curvature gap. Then, for any compactly supported bounded measurable function $g$, we have $\hat{g}^{sc}, \hat{g}^{cg} \in L^{\infty}(\mathbb{P}\Omega(\underline{k}), \mu_{Thu})$. In particular, the linear functionals $L^{sc}, L^{cg}$ are continuous.

Furthermore, there exist finitely many Delaunay regions $D_{S,\underline{s}}^{Del}(\underline{k};T_1), \ldots, D_{S,\underline{s}}^{Del}(T_a)$ such that, for any given $R > 0$, there exist finitely many arcs $\gamma_1, \ldots, \gamma_m$ on $(S, \underline{s})$ such that
\begin{equation}
    \mu^{sc}_{\underline{k}}\Big|_{(0, R)}= \sum_{i=1}^{a} \sum_{j=1}^{m} \mu_{T_i, \gamma_j}\Big|_{(0, R)}.
\end{equation}
A similar decomposition holds for $\mu^{sc}_{\underline{k}}\Big|_{(0, R)}$ if we replace $\gamma_1, \ldots, \gamma_m$ with loops.
\end{theorem}

To prove the theorem, we recall the following uniform upper bound on the counting from~\cite{fu2024uniformlengthestimatestrajectories}.
\begin{lemma}[{\cite[Corollary~1.10]{fu2024uniformlengthestimatestrajectories}}]\label{lem:uniformupperbound}
    Let $X$ be an area-one convex flat cone sphere with $n$ singularities and positive curvature gap $\delta$. Then, $N^{sc}(X,R)$ and $N^{cg}(X,R)$ are bounded above by $ae^{bR}$, where $a$ and $b$ depend only on $n$ and $\delta$.
\end{lemma}

\begin{proof}[Proof of Theorem~\ref{thm:integrabilityandfinitedecomp}]
    According to Lemma~\ref{lem:genericdelaunay}, we denote by $D_{S,\underline{s}}^{Del}(\underline{k};T_1),\ldots,D_{S,\underline{s}}^{Del}(T_a)$ the Delaunay regions that cover the moduli space $\mathbb{P}\Omega(\underline{k})$ up to measure zero under $\mu_{Thu}$. We also view these Delaunay regions as being contained in the Teichm\"uller space. For any compactly supported bounded measurable function $g$, assume that $\overline{\mathrm{supp}(g)}\subseteq (0, R)$ for some $R>0$. Let $\gamma_1,\ldots,\gamma_m$ be the finitely many arcs constructed in Proposition~\ref{prop:finiteboundedlengthgeodesics}. Then, by a direct computation, for every $X\in D_{S,\underline{s}}^{Del}(\underline{k};T_i)$, we have that
    \begin{equation}\label{equ:measurable}
         \hat{g}^{sc}(X) = \sum_{j=1}^{m}\mathbbm{1}_{D_{S,\underline{s}}^{Del}(T_i,\gamma_j)}(X, f)g\big(\ell_{\gamma}(X,f)\big),
    \end{equation} 
    where $\mathbbm{1}_{D_{S,\underline{s}}^{Del}(T_i,\gamma_j)}$ is the indicator function of $D_{S,\underline{s}}^{Del}(T_i,\gamma_j)$ in $\mathbb{P}\mathcal{T}_{S, \underline{s}}(\underline{k})$, and $(X,f)\in D_{S,\underline{s}}^{Del}(\underline{k};T_i)$. Hence, 
    this implies that $\hat{g}^{sc}$ is measurable. Furthermore, by Lemma~\ref{lem:uniformupperbound}, we have that 
    $$|\hat{g}^{sc}(X)|\le \sum\limits_{\gamma} |g(\ell_{\gamma}(X)|\le||g||_{\infty}N^{sc}(X,R)\le ||g||_{\infty}ae^{bR}$$
    where $N^{sc}(X,\cdot)$ counts the number of saddle connections on $X$ with a given upper bound on the normalized length. It follows that $\hat{g}^{sc}$ is in $L^{\infty}(\mathbb{P}\Omega(\underline k),\mu_{Thu})$. In addition, we have that 
    $$
    |L^{sc}(g)|\le \mu_{Thu}\big(\mathbb{P}\Omega(\underline{k})\big)ae^{bR}||g||_{\infty}
    $$
    for $g\in C_c(\mathbb{R}_{>0})$. Hence, the linear functional $L^{sc}$ is continuous.

    By the equation~\eqref{equ:measurable}, we have that  
    \begin{align*}
        \int_{\mathbb{P}\Omega(\underline{k})}\hat{g}^{sc}(X)d\mu_{Thu}
        & = \sum_{i=1}^a\int_{D_{S,\underline{s}}^{Del}(\underline{k};T_i)}\hat{g}^{sc}\big(X,f)d\mu_{Thu}\\
        & = \sum_{i=1}^a\sum\limits_{i=1}^{m}
        \int_{D_{S,\underline{s}}^{Del}(\underline{k};T_i)}\mathbbm{1}_{D_{S,\underline{s}}^{Del}(T_i,\gamma_j)}(X, f)g\big(\ell_{\gamma}(X,f)\big)d\mu_{Thu}\\
        & = \sum_{i=1}^a\sum\limits_{i=1}^{m}
        \int_{D(T,\gamma_j)}
         g\big(\ell_{\gamma}(X,f)\big)d\mu_{Thu}
    \end{align*}
    It follows that 
    $$
    \mu^{sc}_{\underline{k}}\Big|_{(0, R)} = \sum_{i=1}^{a} \sum_{j=1}^{m} \mu_{T_i, \gamma_j}\Big|_{(0, R)}.  
    $$
    The proof for $\mu^{cg}_{\underline{k}}$ is analogous and is omitted here.
\end{proof}

\section{Absolute Continuity and Piecewise Analyticity}\label{bigsec:regularity}

In this section, we prove that the Siegel–Veech measures are absolutely continuous and locally analytic.

As a key step in the argument, we establish the following semialgebraicity result for certain regions in Teichm\"uller space.

\begin{theorem}\label{thm:semialgdelaunayregion}
Let $\underline{k} \in (0,1)^n$ be a curvature vector. Let $T$ be a triangulation on $(S,\underline{s})$, and let $\gamma$ be an arc or a loop on $(S,\underline{s})$. Then, for any spanning tree $F$ in $T$, the coordinate map $\phi_F$, defined in~\eqref{equ:coordinatemap2}, maps the Delaunay region $D_{S,\underline{s}}^{\mathrm{Del}}(\underline{k};T)$ and the region $D_{S,\underline{s}}^{\mathrm{Del}}(\underline{k};T,\gamma)$ to semialgebraic subsets of $\mathbb{R}^{2(n-3)}$.
\end{theorem}

\subsection{Tools from O-minimal Theory}\label{sec:ominimal}

In this subsection, we introduce the o-minimal tools utilized in this paper. For more details on o-minimal theory, we refer the reader to~\cite{om}.

\begin{definition}[Semialgebraicity]\label{def:semialge}
    Let $E$ be a subset of $\mathbb{R}^m$. We say that $E$ is \textbf{algebraic} if it is a finite union or intersection of subsets of the form $\{x \in \mathbb{R}^m \mid g(x) = 0\}$, where $g$ is a polynomial. 

    Similarly, $E$ is called \textbf{semialgebraic} if it is a finite union or intersection of subsets of the form $\{x \in \mathbb{R}^m \mid g(x) > 0\}$ or $\{x \in \mathbb{R}^m \mid g(x) = 0\},$
    where $g$ is a polynomial in $m$ variables.
\end{definition}

\begin{definition}[Semianalyticity]\label{def:semianalytic}
    A subset $E \subset \mathbb{R}^m$ is called \textbf{(real) semianalytic} if, for any point $x \in \mathbb{R}^m$, there exists a neighborhood $U$ of $x$ such that $E \cap U$ is a finite union or intersection of subsets of the form 
    $\{x \in \mathbb{R}^m \mid g(x) > 0\}$ or $\{x \in \mathbb{R}^m \mid g(x) = 0\},$
    where $g$ is an analytic function on $U$.    
\end{definition}

\begin{definition}[Global semianalyticity]
    A subset $E \subset \mathbb{R}^m$ is called \textbf{globally semianalytic} if the image of $E$ under the map
    $$
    (x_1, \ldots, x_m) \mapsto \bigg(\frac{x_1}{\sqrt{1+x_1^2}}, \ldots, \frac{x_m}{\sqrt{1+x_m^2}}\bigg)
    $$
    is a semianalytic subset of $\mathbb{R}^m$.
\end{definition}

A map $f: E \subset \mathbb{R}^m \to \mathbb{R}^n$ is called \textbf{semialgebraic} (or \textbf{semianalytic}, or \textbf{globally semianalytic}) if its graph,
$
\{(x, f(x)) \mid x \in \mathbb{R}^m\},
$
is a semialgebraic (or semianalytic, or globally semianalytic) subset of $\mathbb{R}^{m+n}$.

In the following, we shall refer to a globally semianalytic subset or map as \textbf{subanalytic} for conciseness. This terminology is adopted in accordance with the conventions used in the work of R. Cluckers and D. J. Miller~\cite{intsta}.

\begin{lemma}[\cite{om}]\label{lem:semalgeimplysubanaly}
    A semialgebraic subset of $\mathbb{R}^m$ is subanalytic.
\end{lemma}

We introduce a result from the paper~\cite{intsta} in the following. For a subanalytic set $D\subset\mathbb R^m$ and a Lebesgue measurable function $f:D\times\mathbb{R}^m\to\mathbb{R}$, we define $I_D(f):D\to\mathbb{R}$ as
\begin{equation}\label{equ:defininte}
    x \mapsto
    \begin{cases}    
        \displaystyle\int_{\mathbb{R}^m} f(x, y) \, dy & \text{if $f(x, \cdot)$ is integrable for each $x \in D$,} \\[8pt]
        0 & \text{otherwise.}
    \end{cases}
\end{equation}

\begin{definition}[{\cite[Definition~1.2]{intsta}}]
    For a subanalytic subset $D\subseteq \mathbb{R}^m$, let $\mathcal{C}(D)$ be the $\mathbb{R}$-algebra of real valued functions on $D$ generated by all subanalytic functions $D$ and all functions $x\mapsto \log f(x)$ where $f:D\to(0,+\infty)$ is subanalytic. We call $f\in \mathcal{C}(D)$ a \textbf{constructible function} on $D$.    
\end{definition}

A main result in~\cite{intsta} states as follows:
\begin{lemma}[\cite{intsta}, Theorem 1.3]\label{lem:constructible}
    Let $f$ be in $\mathcal{C}(D\times\mathbb{R}^m)$ for some subanalytic set $D$. Then $I_D(f)$ is in $\mathcal{C}(D)$.
\end{lemma}

In addition, we note the following basic property of constructible functions:
\begin{lemma}[{\cite[Section 4]{DMM}}]\label{lem:piecewiseanaly}
    If $f$ is a constructible function on $\mathbb{R}$, then there exists a finite partition $0 = t_1 < \ldots < t_{m-1} < t_m = \infty$ such that $f|_{(t_{i-1}, t_i)}$ is analytic for $i = 1,\ldots,m$.
\end{lemma}

\subsection{Semialgebraicity of Delaunay Regions}

\begin{lemma}\label{lem:delaunaychart}
Let $T$ be a triangulation on $(S,\underline{s})$.
Then, for any spanning tree $F$ in $T$, the coordinate map $\phi_F$ maps the Delaunay region $D_{S,\underline{s}}^{Del}(\underline{k};T)$ to a semialgebraic subset of $\mathbb{R}^{2(n-3)}$. 
\end{lemma}

\begin{proof}
    To prove the lemma, our strategy is to define a semialgebraic subset in the spanning tree coordinate chart $D_{S,\underline{s}}(\underline{k};F)$, and prove that it corresponds to the Delaunay region $D_{S,\underline{s}}^{Del}(\underline{k};T)$.

    We denote the edges of $F$ by $e_0,\ldots,e_{n-2}$. Recall that in the construction of spanning tree coordinate chart, we immerse the complement $X\setminus F$ into the plane by a developing map $Dev$, and associate the vector $(z_0,\ldots,z_{n-2})$ to the edge of $F$. We normalize the number $z_0$ to be 1, so the vector $(z_1,\ldots,z_{n-3})$ forms a coordinate on $D_{S,\underline{s}}(\underline{k};F)$. For $X\in D_{S,\underline{s}}^{Del}(\underline{k};T)$, we have the induced triangulation on $X\setminus T$, denoted by $T$ also. In particular, each face of $T$ maps to a Euclidean triangle on the plane. Under the developing of $X\setminus F$, each oriented edge $e$ of $T$ is associated with a complex number, denoted by $z(e)$. We write the oriented edge with different orientation of $e$ as $-e$, so $z(-e) = -z(e)$.  Note that each $z(e)$ is a complex linear combination of $(1, z_1,\ldots,z_{n-3})$.

    Consider a subset $D \subset \mathbb{C}^{n-3}$ defined by the following constraints:
    \begin{enumerate} 
        \item For each triangular face of $T$ on $X\setminus F$, let $e, e'$ be two edges in counterclockwise order at a vertex $v$ of the triangle. We orient these edges so that their starting endpoint is $v$.
        We require that 
        $$
        \operatorname{Im}\left(\overline{z(e)}z(e')\right) > 0.
        $$
        \item For any quadrilateral of $T$ on $X$, denote the two triangular faces in the quadrilateral by $f_1$ and $f_2$, and let $e, e', e''$ be the three edges in the quadrilateral at the same vertex $v$ in counterclockwise order. We also orient them so that their starting endpoint is $v$.
    
        We develop $f_1$ and $f_2$ to a quadrilateral on the plane via the developing map $Dev$, such that $f_1$ maps to the triangle $Dev|_{X \setminus F}(f_1)$ in $Dev(X \setminus F)$. Note that $f_2$ maps to the triangle $Dev|_{X \setminus F}(f_2)$ if $e'$ is not in $F$; otherwise, $f_2$ maps to a triangle obtained by rotating $Dev|_{X \setminus F}(f_2)$, where the rotation is determined by holonomy. 
    
        Let $w(e), w(e'), w(e'')$ be the associated complex vectors under this developing. In either case, these vectors are complex linear combinations of  $(1, z_1, \ldots, z_{n-3})$. We require that they satisfy the Delaunay condition~\eqref{equ:delaunaycondition} in Lemma~\ref{lem:delaunaycondition}, with $\alpha_1 = w(e)$, $\alpha_2 = w(e')$, and $\beta = w(e'')$.
    \end{enumerate}
    Note that the above constraints define a semialgebraic subset $D$ in $\mathbb{C}^{n-3}$.

    We prove that $\phi_F(D_{S,\underline{s}}^{Del}(\underline{k};T)) = D$ as follows.
    
    Since $T$ is the Delaunay triangulation on $X$ for $(X, f) \in D_{S,\underline{s}}^{Del}(\underline{k};T)$ and $X$ is Delaunay-generic, we have $\phi_F(X) \subset D$. Conversely, for any vector $(z_1, \ldots, z_{n-3}) \in D$, constraints~(1) ensures that we can construct a Euclidean triangle for each face of $T$ using the given vector. By gluing these Euclidean triangles according to the triangulation $T$ on $(S, \underline{s})$, we obtain a flat cone sphere $X$ with $T$ as its geometric triangulation. Denote the singularity at the vertex $s_i$ by $x_i$, and let its curvature be $k'_i$. Since the gluing follows the structure of $T$, the curvature $k'_i$ satisfies $k'_i = k_i - n_i$ for some non-negative integer $n_i$. 
    
    Given that $\sum k_i = 2$, and by the Gauss-Bonnet formula, $\sum k'_i = 2$, it follows that $n_i = 0$ for each $i$. Hence, $X$ has the curvature vector $\underline{k}$. Finally, constraint~(3) ensures that $T$ is the Delaunay triangulation of $X$. 
    
    Since $D_{S,\underline{s}}^{Del}(\underline{k};T)$ is embedded in the moduli space, there is a unique way to associate $X$ with a marking $f$ such that $(X, f) \in D_{S,\underline{s}}^{Del}(\underline{k};T)$. Therefore, $\phi_F(X, f) = (z_1, \ldots, z_{n-3})$, and we conclude that $\phi_F(D_{S,\underline{s}}^{Del}(\underline{k};T)) = D$.
\end{proof}

\subsection{Semialgebraicity of Regions of Geodesics}
We begin by introducing the following terminologies, which are adopted from~\cite{Tabachnikov1995, RSch06, RSch08, Tabachnikov2005,RSch22}.

\begin{definition}
    Let $X$ be a flat cone sphere, and let $T$ be a geometric triangulation on $X$. Suppose $\gamma$ is a saddle connection on $X$ that is not an edge of $T$. Define the \textbf{unfolding $P(\gamma, X, T)$ of $\gamma$} as the triangulated polygon immersed in the plane obtained by developing the triangular faces of $T$ that $\gamma$ passes through.
    
    Similarly, let $\gamma$ be a regular closed geodesic on $X$, and let $e_0$ be an edge of $T$ that intersects $\gamma$. Define the {unfolding $P(\gamma, X, T, e_0)$ of $\gamma$} as the triangulated polygon immersed in the plane obtained by developing the triangular faces of $T$ that $\gamma$ passes through, starting from the edge $e_0$ and returning to the edge $e_0$.
\end{definition}

Figures~\ref{fig:unfoldsaddleconnection} and~\ref{fig:unfoldclosedgeodesic} illustrate the unfoldings of a saddle connection and a regular closed geodesic, respectively.

Note that a saddle connection (or a regular closed geodesic) $\gamma$ unfolds into a straight segment in the polygon $P(\gamma, X, T)$ (or $P(\gamma, X, T, e_0)$). For simplicity, we also denote this segment by $\gamma$. In the case of a regular closed geodesic, the edge $e_0$ is unfolded into two parallel edges in the polygon $P(\gamma, X, T, e_0)$, denoted by $e_0'$ and $e_0''$; see Figure~\ref{fig:unfoldclosedgeodesic}.

\begin{definition}
    Let $X$ be a flat cone sphere, and let $T$ be a geometric triangulation on $X$. For a regular closed geodesic $\gamma$ on $X$, we define the maximal strip in $P(\gamma, X, T, e_0)$, foliated by segments that start at $e_0'$ and end at $e_0''$, and are parallel to $\gamma$, as the \textbf{corridor} of the regular closed geodesic $\gamma$.
\end{definition}

Figure~\ref{fig:unfoldclosedgeodesic} illustrates an example of the corridor associated with a regular closed geodesic.

We prove the semialgebraicity of domains of saddle connections and regular closed geodesics in the following.

\begin{lemma}\label{lem:semialgdomain}
    Let $T$ be a traingulation on $(S,\underline{s})$, and let $\gamma$ be an arc (or a loop) in $(S,\underline{s})$. Then, for any spanning tree $F$ in $T$, the coordinate map $\phi_F$ maps $D_{S,\underline{s}}^{Del}(\underline{k};T,\gamma)$ to a semialgebraic subset of $\mathbb{R}^{2(n-3)}$.
\end{lemma}

\begin{proof}
    Let $\phi:D_{S,\underline{s}}(\underline{k};F)\to \mathbb{C}^{n-3}$ be the coordinate map~\eqref{equ:coordinatemap2}. Let $D=\phi(D_{S,\underline{s}}^{Del}(\underline{k};T))$ be the semialgebraic subset constructed in the proof of Lemma~\ref{lem:delaunaychart}. We show that $\phi(D_{S,\underline{s}}^{Del}(\underline{k};T,\gamma))$ is obtained by intersecting $D$ with additional polynomial constraints.

    If $\gamma$ coincides with an edge of $T$, then $D_{S,\underline{s}}^{\mathrm{Del}}(\underline{k};T, \gamma) = D_{S,\underline{s}}^{\mathrm{Del}}(\underline{k};T)$, and hence the region is semialgebraic.

    If $\gamma$ is an arc different from the edges of $T$, for any $(X,f)\in D_{S,\underline{s}}^{Del}(\underline{k};T, \gamma)$, note that $\gamma$ is unfolded to a diagonal of the polygon $P(\gamma, X, T)$. We denote the starting and the ending vertices of the diagonal $\gamma$ by $v_1$ and $v_2$ respectively. Note that such a diagonal divides the rest of the vertices of $P(\gamma, X, T)$ into to two subsets, denoted by $V_l$ and $V_r$. The sets $V_l$ and $V_r$ consist of the vertices on the left and right hand side of the segment $\gamma$ respectively. 

    We place the polygon $P(\gamma, X, T)$ such that the first triangle that $\gamma$ passes through is developed by $Dev$ restricted to $X\setminus F$. 
    As in Lemma~\ref{lem:delaunaychart}, an oriented edge $e$ of the polygon $P(\gamma, X, T)$ is associated to a complex vector, denoted by $w(e)$. Note that $w(e)$ is a complex linear combination of $(1, z_1,\ldots,z_{n-3})$. For each vertex $v$ of $P(\gamma,X,T)$, let $\gamma_v$ be an oriented path on the boundary from $v_1$ to $v$. Denote by $w_{v_1}(v)$ the complex vector obtained by summing the $w(e)$, where the oriented edge $e$ is in contained in $\gamma_v$ with the same orientation. To ensure that the diagonal between $v_1$ and $v_2$ lies inside $P(\gamma, X, T)$, we impose the following constraint on $D$:
    \begin{enumerate}
        \setcounter{enumi}{2}
        \item We require that for $v\in V_l$,
        $$
        \operatorname{Im}\frac{w_{v_1}(v)}{w_{v_1}(v_2)}>0
        $$
        and for $v\in V_r$,
        $$        
        \operatorname{Im}\frac{w_{v_1}(v)}{w_{v_1}(v_2)}<0.
        $$
    \end{enumerate}
    Let $D_1$ be the subset of $D$ with the constraint~(3). Note that $D_1$ is also a semialgebraic subset.

    \begin{figure}[!htbp]
	\centering
	\includegraphics[width=0.8\linewidth]{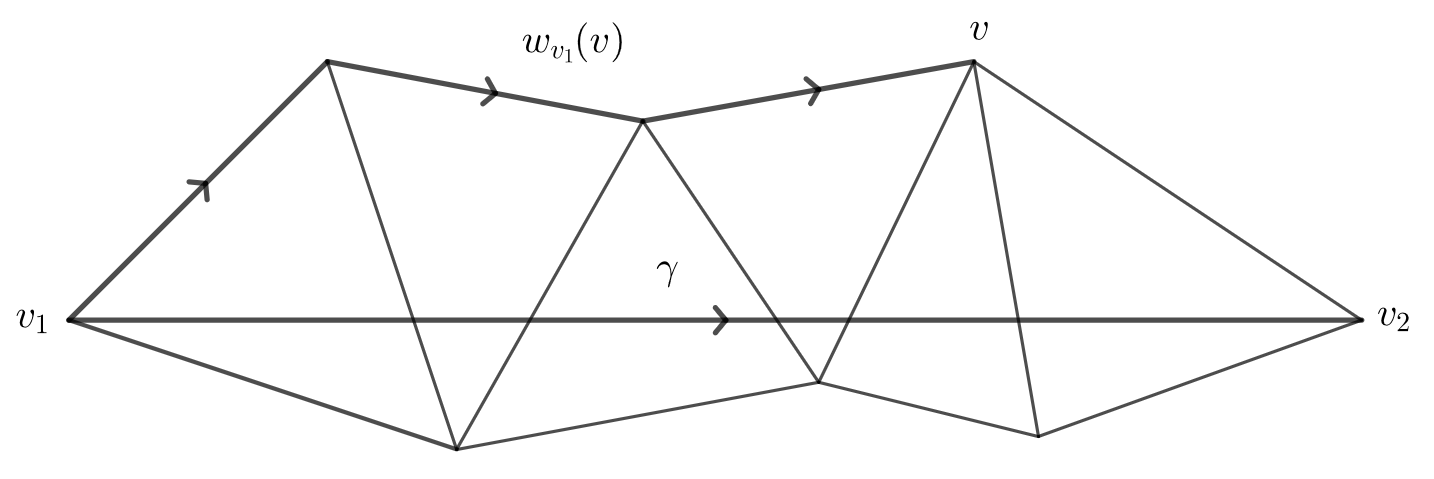}
	\footnotesize
    \caption{$P(\gamma, X,T)$ and the complex vector $w_{v_1}(v)$ associated to the vertices $v$ and $v_1$.}\label{fig:unfoldsaddleconnection}
    \end{figure}
    
    We show that $\phi(D_{S,\underline{s}}^{Del}(\underline{k};T, \gamma)) = D_1$ as follows. Note that if the coordinate $(z_1,\ldots,z_{n-3})$ of $(X,f)\in D_{S,\underline{s}}^{Del}(\underline{k};T)$ is contained in $D_1$, the constraint~(3) implies that the diagonal between $v_1$ and $v_2$ exists in the interior of $P(\gamma, X, T)$. Then this diagonal corresponds to a saddle connection on $X$ realizing $f(\gamma)$. Conversely, if $(X,f)$ is contained in $D_{S,\underline{s}}^{Del}(\underline{k};T, \gamma)$, then $\gamma$ is unfolded to the diagonal between $v_1$ and $v_2$. Thus the constraint~(3) holds. Therefore this proves that $\phi(D_{S,\underline{s}}^{Del}(\underline{k};T, \gamma)) = D_1$. 

    We prove similarly when $\gamma$ is a regular closed geodesic. Let $P(\gamma, X, T,e_0)$ be the polygon associated to $\gamma$ for $(X,f)\in D_{S,\underline{s}}^{Del}(\underline{k};T, \gamma)$. Denote by $L$ the corridor of $\gamma$ in $P(\gamma, X, T,e_0)$.  Similarly, the corridor divides the vertices of $P(\gamma, X, T,e_0)$ into two subsets $V'_l$ and $V'_r$ such that the vertices in $V'_l$ and $V'_r$ are on the left and right hand side of the segment $\gamma$ respectively.
    
    Define the complex number $w(e)$ associated to oriented edge $e$ of $P(\gamma, X, T,e_0)$ as in the case of ars. Again, for a vertex $v$ of $P(\gamma, X, T,e_0)$, define a complex vector $w_v(v')$ associated to an oriented path in the boundary of $P(\gamma, X, T,e_0)$ from $v$ to $v'$ as above. Let $w(e_0)$ be the vector that translates the edge $e_0'$ to $e_0''$, which corresponds to the vector representing the segment $\gamma$ in $P(\gamma, X, T,e_0)$. 

    We impose the following additional constraint on $D$:
    \begin{enumerate}
        \setcounter{enumi}{2}
        \renewcommand{\labelenumi}{(\arabic{enumi}')}
        \item For any vertex $v$ in $V'_r$ and any vertex $v'$ in $V'_l$, we require that
        $$
        \operatorname{Im}\frac{w_{v}(v')}{w(e_0)}>0.
        $$
    \end{enumerate}
    Let $D_2$ be the subset of $D$ with the constraint~(3'). Then $D_2$ is also a semialgebraic subset.

    \begin{figure}[!htbp]
	\centering
	\includegraphics[width=0.8\linewidth]{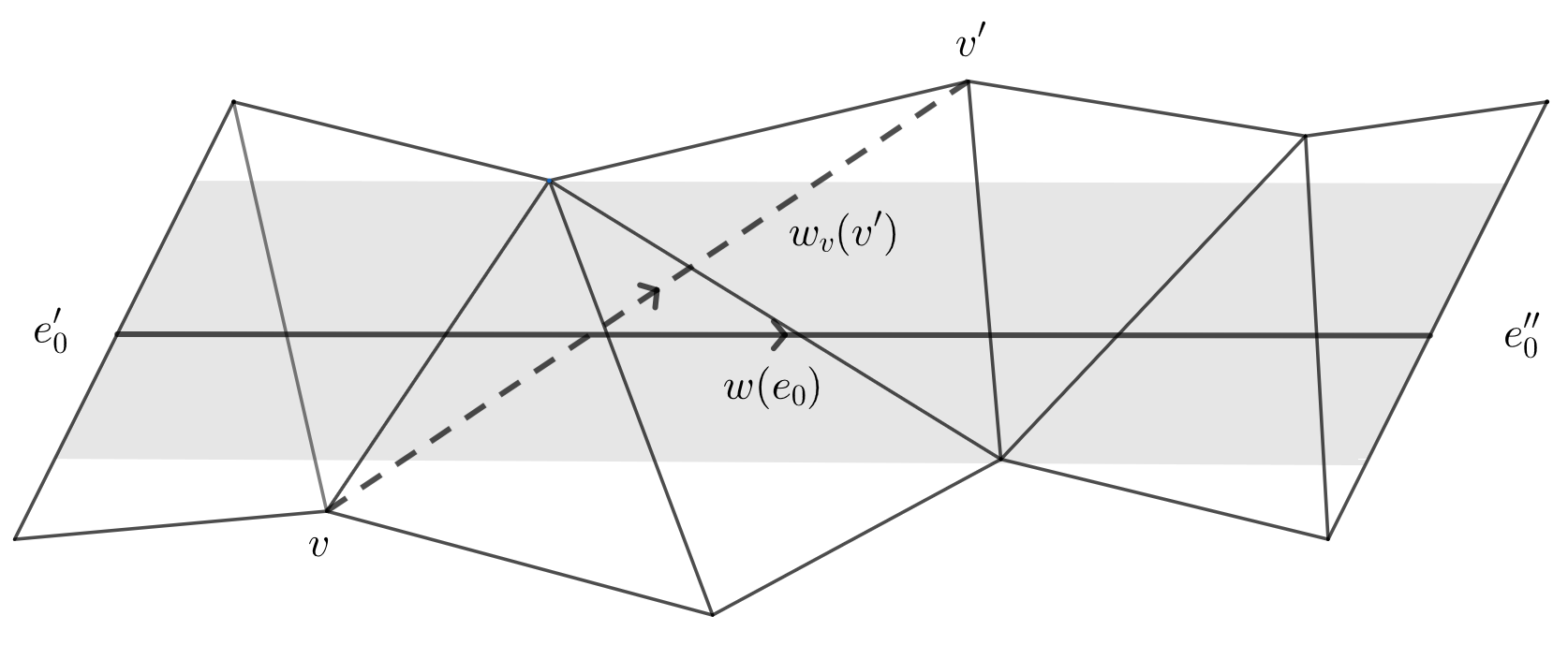}
    \caption{The unfolding $P(\gamma, X, T, e_0)$ and the complex vector $w_v(v')$ associated to the vertices $v$ and $v'$. The gray region indicates the corridor of the regular closed geodesic~$\gamma$.}\label{fig:unfoldclosedgeodesic}
    \end{figure}

    Note that the constraint~(3') is equivalent to the requirement that the corridor $L$ has non-empty interior in  $P(\gamma, X, T,e_0)$, and it is equivalent to require that $f(\gamma)$ is realized as a regular closed geodesic on $X$. Hence $\phi(D_{S,\underline{s}}^{Del}(\underline{k};T, \gamma)) = D_2$.
\end{proof}

Combining Lemma~\ref{lem:delaunaychart} with Lemma~\ref{lem:semialgdomain}, we conclude the proof of Theorem~\ref{thm:semialgdelaunayregion}.\newline

For later use, we prove the following property of normalized length functions.

\begin{lemma}\label{lem:semialgnormalizedlength}
    The normalized length function $\ell_{\gamma}(X)$ on $D_{S,\underline{s}}^{Del}(\underline{k};T, \gamma)$ is a non-constant semialgebraic function on the spanning tree coordinate chart in $D_{S,\underline{s}}(\underline{k};F)$ for any spanning tree $F$ in $T$.
\end{lemma}

\begin{proof}
    Recall that the normalized length function $\ell_{\gamma}(Z)$ is defined by
    $$
    \ell_{\gamma}(Z) = \frac{|\gamma|}{\sqrt{\mathrm{Area}(1,Z)}}
    $$
    where $|\gamma|$ is the metric length of $\gamma$ on the flat cone sphere $X$ corresponding to the coordinate $Z$. To express $|\gamma|$ in coordinate, consider the unfolding polygon $P(\gamma, X, T)$ (resp. $P(\gamma, X, T, e_0)$) in the plane. According to the proof of Lemma~\ref{lem:semialgdomain}, $\gamma$ is unfolded into a segment in $P(\gamma, X, T)$ (resp. $P(\gamma, X, T, e_0)$) that is represented by a complex vector $w_{v_1}(v)$ (resp. $w(e_0)$), and it is a complex linear combination of $(1,Z)$. Since $|\gamma| = |w_{v_1}(v)|$ (resp. $|w(e_0)|$) and the area function $\mathrm{Area}(z_0,Z)$ is a quadratic form in $(z_0,Z)$, we know that $\ell_{\gamma}(Z)$ is semialgebraic.

    Furthermore, consider the vector $(z_0, Z) \in \mathbb{C}^{n-2}$, where $z_0$ is not normalized to $1$. We extend $w_{v_1}(v)$ to the complex linear combination of $(z_0, Z)$ with the same coefficients. Then the expression 
    $$
    \frac{|w_{v_1}(v)|^2}{\mathrm{Area}(z_0, Z)}
    $$ 
    is defined for $(z_0, Z)$ in a domain of $ \mathbb{C}^{n-2}$ and is equal to $\ell_{\gamma}^2(z_0^{-1} Z)$. Since the numerator and denominator are quadratic forms with different signatures, this ratio cannot be a constant on any open domain.
\end{proof}

\subsection{Piecewise Analyticity}\label{sec:analy}
In this subsection, we focus on proving the following theorem.
\begin{theorem}\label{thm:analytic}
    Let $\underline{k}\in (0,1)^n$ be a curvature vector with a positive curvature gap. Then, there exists an infinite partition $0 = t_0 < t_1 < \ldots < t_m < \ldots$ with $\lim\limits_{m \to \infty} t_m = \infty$ such that the functions
    $$
    R \mapsto \mu^{sc}_{\underline{k}}(0, R) \quad \text{and} \quad R \mapsto \mu^{cg}_{\underline{k}}(0, R)
    $$
    are real analytic in $(t_i, t_{i+1})$ for each $i$.
\end{theorem}

According to the decomposition given by Theorem~\ref{thm:integrabilityandfinitedecomp}, the proof of the above theorem is reduced to analyzing the following function:
$$
R \mapsto \mu_{T, \gamma}(0,R),
$$
where $T$ is a triangulation of $(S, \underline{s})$ corresponding to a Delaunay region $D_{S,\underline{s}}^{Del}(\underline{k};T)$ in the Teichm\"uller space, and $\gamma$ is an arc or a loop on $(S, \underline{s})$.  Recall that $\mu_{T,\gamma}$ is defined as a Radon measure corresponding to the linear functional~\eqref{equ:lfpath}
$$g \mapsto \int_{D_{S,\underline{s}}^{Del}(\underline{k};T, \gamma)}g\circ\ell _{\gamma}(X,f)d\mu_{Thu}$$
on $C_c(\mathbb{R}_{>0})$, where $\ell _{\gamma}$ is the normalized length function defined on $D_{S,\underline{s}}^{Del}(\underline{k};T, \gamma)$.

The strategy is to employ the o-minimal tools introduced in Section~\ref{sec:ominimal}. Let $F$ be a spanning tree of $T$. The Delaunay region $D_{S,\underline{s}}^{Del}(\underline{k};T)$ is contained in the spanning tree coordinate chart $D_{S,\underline{s}}(\underline{k};F)$. Denote by $Z=(z_1,\ldots,z_{n-3})$ the coordinate in $D_{S,\underline{s}}(\underline{k};F)$. In the following analysis, we do not distinguish between the subset $D_{S,\underline{s}}^{Del}(\underline{k};T)$ and its corresponding domain in $\mathbb{C}^{n-3}$ under the spanning tree coordinate chart;  see~\eqref{equ:coordinatemap2}.

In the spanning tree coordinate chart, the volume form $d\mu_{Thu}$ is written as
\begin{equation}
    d\mu_{Thu} = \left(\frac{i}{2}\right)^{n-3}h(Z)dZd\overline{Z}
\end{equation}
where $dZd\overline{Z}:= dz_1d\overline{z}_1\ldots dz_{n-3}d\overline{z}_{n-3}$, and 
\begin{equation}\label{equ:expressionforh}
    h(Z) = \operatorname{det}\left(h_{Thu}\left(\frac{\partial}{\partial z_i},\frac{\partial}{\partial z_j}\right)\right)_{1\le i,j\le n-3}
\end{equation}
is a real-valued smooth function.

Define $D_{s,\underline{s}}^{Del}(T,\gamma,R)$ as the preimage $\ell_\gamma^{-1}(0,R)$ in $D_{S,\underline{s}}^{Del}(\underline{k};T, \gamma)$.
Define $G:\mathbb{R}_{>0}\times\mathbb{C}^{n-3}\to \mathbb{R}$ as
\begin{equation}
    G(R, Z)=\mathbbm{1}_{D_{s,\underline{s}}^{Del}(T,\gamma,R)}(Z)h(Z),
\end{equation}
where $\mathbbm{1}_{D_{s,\underline{s}}^{Del}(T,\gamma,R)}$ is the indicator function of the subset $D_{s,\underline{s}}^{Del}(T,\gamma,R)$ in $\mathbb{C}^{n-3}$.

The following lemma is formulated using the notation introduced in~\eqref{equ:defininte}.

\begin{lemma}\label{lem:measureandintegral}
    The measure $\mu_{T,\gamma}$ satisfies that $\mu_{T,\gamma}(0,R) = I_{\mathbb{R}_{>0}}(G)(R)$ for any $R>0$. 
\end{lemma}

\begin{proof}
    By definition, we compute that 
    \begin{align*}
        \mu_{T,\gamma}(0,R) 
        =&\int_{D_{S,\underline{s}}^{Del}(\underline{k};T, \gamma)}\mathbbm{1}_{(0,R)}\circ \ell_{\alpha}(X,f)d\mu_{Thu}\\
        =& \int_{D_{S,\underline{s}}^{Del}(\underline{k};T, \gamma)}\mathbbm{1}_{(0,R)}(\ell_{\alpha}(Z))\cdot h(Z) \cdot\left(\frac{i}{2}\right)^{n-3}dZd\overline{Z}\\
        =& \int_{\mathbb C^{n-3}}\mathbbm{1}_{D(T,\gamma,R)}(Z)h(Z)\cdot\left(\frac{i}{2}\right)^{n-3}dZd\overline{Z}\\
        =& \int_{\mathbb C^{n-3}} G(R, Z)\cdot\left(\frac{i}{2}\right)^{n-3}dZd\overline{Z}.
    \end{align*}
    Since $D_{S,\underline{s}}^{Del}(\underline{k};T)$ is of finite measure under $\mu_{Thu}$, we know that $G(R,Z)$ is integrable for each $R>0$. Note that $\left(\frac{i}{2}\right)^{n-3}dZd\overline{Z}$ is the Lebesgue measure on $\mathbb{C}^{n-3}$. It follows that 
    $$\mu_{T,\gamma}(0,R) = I_{\mathbb{R}_{>0}}(G)(R).$$  
\end{proof}

\begin{lemma}\label{lem:gisconstructible}
    The function $G: \mathbb{R}_{>0} \times \mathbb{C}^{n-3} \to \mathbb{R}$ is semialgebraic and, therefore, constructible on $\mathbb{R} \times \mathbb{C}^{n-3}$. Furthermore, the function $R\mapsto \mu_{T,\gamma}(0,R)$ is constructible on $\mathbb{R}_{>0}$.  
\end{lemma}

\begin{proof}
    Note that $G(R, Z)=\mathbbm{1}_{D_{s,\underline{s}}^{Del}(T,\gamma,R)}(Z)h(Z)$, where $h$ has the expression~\eqref{equ:expressionforh}. According to the formula of $h_{Thu}$ in the spanning tree coordinate chart in Lemma~\ref{lem:metricformula}, we know that $h$ is a rational function about $Z$, and, in particular, semialgebraic. 
    
    The rest of the proof is to show that the indicator function of $\mathbbm{1}_{D_{s,\underline{s}}^{Del}(T,\gamma,R)}(Z)$ on $\mathbb{C}^{n-3}$ is semialgebraic. Denote its graph in $\mathbb{R}_{>0}\times\mathbb{C}^{n-3}\times\mathbb{R}$ by $\Gamma$. Consider a subset $A$ of $\mathbb{R}_{>0}\times\mathbb{C}^{n-3}$ defined by
    $$
    A=\{(R,Z)\in \mathbb{R}_{>0}\times D_{S,\underline{s}}^{Del}(\underline{k};T, \gamma)\mid \ell_{\gamma}(Z)<R\}.$$
    Note that $(R,Z)\in \mathbb{R}_{>0}\times\mathbb{C}^{n-3}$ is contained in $A$ if and only if the indicator function $$\mathbbm{1}_{D_{s,\underline{s}}^{Del}(T,\gamma,R)}(Z) = 1.$$ 
    It follows that the graph $\Gamma$ is written as
    $$\Gamma=A\times\{1\} \sqcup \big(\mathbb{R}_{>0}\times\mathbb{C}^{n-3}\setminus A\big)\times \{0\}.$$
    Therefore, to demonstrate that the graph $\Gamma$ is semialgebraic, it suffices to show that the subset $A$ is semialgebraic.

    According to Lemma~\ref{lem:semialgdomain} and Lemma~\ref{lem:semialgnormalizedlength}, we know that $D_{S,\underline{s}}^{Del}(\underline{k};T, \gamma)$ is a semialgebraic subset in $\mathbb{C}^{n-3}$ and $\ell_{\gamma}(Z)$ is a semialgebraic function. It follows that $A$ is semialgebraic. Hence the graph $\Gamma$ is semialgebraic. By definition, the indicator function $\mathbbm{1}_{D(T,\gamma, R')}(Z)$ is semialgebraic.

    Since $G$ is a product of two semialgebraic functions, itself is also semialgebraic. By Lemma~\ref{lem:semalgeimplysubanaly}, the function $G$ is subanalytic, and, therefore, constructible.

    Furthermore, by Lemma~\ref{lem:measureandintegral} and Lemma~\ref{lem:constructible} imply that the function $\mu_{T,\gamma}(0,\cdot) = I_{\mathbb{R}_{>0}}(G)(\cdot)$ is constructible on $\mathbb{R}_{>0}$.
\end{proof}

We are now finally ready to complete the proof of Theorem~\ref{thm:analytic}.

\begin{proof}[Proof of Theorem~\ref{thm:analytic}]
    By Theorem~\ref{thm:integrabilityandfinitedecomp}, there are finitely many Delaunay regions 
    $$
    D_{S,\underline{s}}^{Del}(\underline{k};T_1), \ldots, D_{S,\underline{s}}^{Del}(T_a)
    $$ 
    such that, for any given $R > 0$, there exist finitely many arcs $\gamma_1, \ldots, \gamma_m$ on $(S, \underline{s})$ such that
    $$
        \mu^{sc}_{\underline{k}}\Big|_{(0, R)} = \sum_{i=1}^{a} \sum_{j=1}^{m} \mu_{T_i, \gamma_j}\Big|_{(0, R)}.
    $$
    Lemma~\ref{lem:gisconstructible} implies that each function $\mu_{T_i, \gamma_j}(0,\cdot)$ is constructible. By Lemma~\ref{lem:piecewiseanaly}, we know that $\mu_{T_i, \gamma_j}(0,\cdot)$ is piecewise analytic. It follows that the function $R'\mapsto \mu^{sc}_{\underline{k}}(0, R')$ is piecewise analytic on $(0,R)$. Since $R$ is an arbitrary positive number, we conclude that there is an infinite partition $0 = t_0 < t_1 < \ldots < t_m < \ldots$ with $\lim\limits_{m \to \infty} t_m = \infty$ such that the functions
    $$
    R \mapsto \mu^{sc}_{\underline{k}}(0, R) \quad \text{and} \quad R \mapsto \mu^{cg}_{\underline{k}}(0, R)
    $$
    are real analytic in $(t_i, t_{i+1})$ for each $i$.

    The proof for $\mu^{cg}_{\underline{k}}$ is analogous and is omitted here.
\end{proof}

\begin{corollary}\label{cor:abc}
    Let $\underline{k}\in(0,1)^n$ be a curvature vector with positive curvature gap. If $n\ge 4$, then the Siegel--Veech measures $\mu^{sc}_{\underline{k}}$ and $\mu^{cg}_{\underline{k}}$ are absolutely continuous with respect to the Lebesgue measure $\mu_{Leb}$ on $\mathbb{R}_{>0}$. 
\end{corollary}

\begin{proof}
    Theorem~\ref{thm:analytic} implies that $\mu^{sc}_{\underline{k}}(0,\cdot)$ and $\mu^{cg}_{\underline{k}}(0,\cdot)$ are piecewise real analytic. To conclude the absolute continuity, it suffices to show that the Siegel--Veech measures are non-atomic, that is, $\mu^{sc}_{\underline k}(\{R\}) = \mu^{cg}_{\underline k}(\{R\}) = 0$ for any $R\in \mathbb R_{>0}$. By the decomposition in Theorem~\ref{thm:integrabilityandfinitedecomp}, it suffices to prove that the measure $\mu_{T,\gamma}$ is non-atomic for any Delaunay region $D_{S,\underline{s}}^{Del}(\underline{k};T)$ and any arc or loop $\gamma$. By definition,
    \begin{align*}
        \mu_{T,\gamma}(\{R\}) &= \int_{D_{S,\underline{s}}^{Del}(\underline{k};T, \gamma)}
         \mathbbm{1}_{\{R\}}\big(\ell_{\gamma}(X,f)\big)d\mu_{Thu}\\
         &= \mu_{Thu}\left(\left\{(X, f)\in D_{S,\underline{s}}^{Del}(\underline{k};T, \gamma) \mid \ell_{\gamma}(X,f) = R \right\}\right).
    \end{align*}
    Note that Lemma~\ref{lem:semialgnormalizedlength} implies that the locus of the preimage $\ell^{-1}_{\gamma}(R)$ is a algebraic proper subset in the spanning tree coordinate chart. It follows that the measure under $\mu_{Thu}$ is zero, and $\mu_{T,\gamma}$ is non-atomic.
\end{proof}

\section{Moduli Space of Convex Infinite Flat Spheres}\label{bigsec:convexinfinite}

In this section, we introduce tools for studying the moduli space $\mathbb{P}\Omega(\underline{k})$ of the convex infinite flat spheres with fixed curvature vector $\underline{k}$. These tools include Delaunay triangulations and shortest spanning trees, which will later be used in the construction of coordinate charts and the computation of volume asymptotics.

\subsection{Delaunay Triangulations}
Recall that a non-negative infinite flat sphere $X$ has the curvatures of the conical singularities lie in the interval $[0,1)$, while the curvature of the poles is at least one. Since the total sum of the curvatures is equal to two, it follows that $X$ either has exactly one pole of curvature larger than $1$, or has two simple poles.

Let $\underline{k} = (k_1,\ldots,k_n)$ be a curvature vector with $k_i \ge 0$ for $1 \le i \le n-1$ and $k_n > 1$. We say that such a curvature vector is of \textbf{planar type}. A flat sphere is called \textbf{planar} if its curvature vector is of planar type. Note that a convex infinite flat sphere is always planar.

An example of a planar flat sphere is the Euclidean plane $\mathbb{C}$ equipped with labeled points $x_1, \ldots, x_n \in \mathbb{C}$. Such surfaces belong to the moduli space $\mathbb{P}\Omega(0,\ldots,0,2)$.\newline

\noindent\textit{{\textbf{Convention.}}} Given a curvature vector $\underline{k}$ of planar type, we assume that the the curvature $k_n$ is larger than one, so the singularity $x_n$ is the unique pole on $X\in \mathbb{P}\Omega(\underline{k})$.\newline

As in Section~\ref{sec:delaunay}, we define a \textbf{maximal} immersion of a disk $d: B(0,r) \to X$ and denote its continuous extension by $\overline{d}: \overline{B}(0,r) \to X$. 

Let $f$ be a geodesic triangle on $X$ whose edges are saddle connections. We say that $f$ is \textbf{inscribed in a disk} $\overline{d}: \overline{B}(0, r) \to X$ if $f$ is the image under $\overline{d}$ of a triangle inscribed in $\overline{B}(0, r)$. Such a triangular face $f$ is called \textbf{Delaunay}.

\begin{definition}
    Let $X$ be a planar flat sphere, and denote its core by $D$ (see Definition~\ref{def:convexhull}). A triangulation of $D$ is called \textbf{geometric} if each edge is a saddle connection.  A \textbf{geometric triangulation of $X$} is a geometric triangulation of $D$.

    A geometric triangulation $T$ of $X$ is called \textbf{Delaunay} if each face of $T$ is Delaunay.

    A planar flat sphere $X$ is called \textbf{Delaunay-generic} if there exists a Delaunay triangulation $T$ of $X$ such that no two triangular faces are inscribed in the same disk in $X$.
\end{definition}

\begin{remark}
The notion of a Delaunay triangulation on a planar flat sphere differs from the usual notion of triangulation on a surface. However, it coincides with the original definition given by Delaunay in~\cite{Delaunay}, which applies to a finite set of points in the Euclidean plane. In particular, this is the case $X \in \mathbb{P}\Omega(0,\ldots,0,2)$.
\end{remark}

For later use, we define Delaunay triangulations  for convex hulls.
\begin{definition}\label{def:delaunayforconvexhull}
    Let $X$ be a convex flat sphere, and let $D$ be a convex hull in $X$.
    Let $X_1$ be the infinitesimal sphere obtained by performing Thurston surgery along $D$ in $X$.
    A triangulation $T$ of $D$ is called \textbf{Delaunay} if $T$ is a Delaunay triangulation of $X_1$.

    We say that the convex hull $D$ is \textbf{Delaunay-generic} if the corresponding infinitesimal sphere is Delaunay-generic.
\end{definition}

Next, we define Delaunay regions in Teichm\"uller space.
\begin{definition}
Let $(S, \underline{s})$ be a sphere with labeled points $\underline{s} = (s_1, \ldots, s_n)$. A \textbf{triangulation $T$ of $(S, \underline{s})$ with respect to $\{s_1, \ldots, s_{n-1}\}$} is a cell decomposition of $S$ satisfying the following conditions:
\begin{itemize}
    \item The points $s_1, \ldots, s_{n-1}$ are the vertices of the cell decomposition, while the point $s_n$ lies in the interior of a face.
    \item The faces that do not contain $s_n$ are triangles.
\end{itemize}
We call the boundary of the face containing $s_n$ the \textbf{boundary of the triangulation $T$}.
\end{definition}

\begin{definition}
    Let $T$ be a triangulation of $(S, \underline{s})$ with respect to $\{s_1, \ldots, s_{n-1}\}$. We define the \textbf{Delaunay region} $D_{S,\underline{s}}^{Del}(\underline{k};T)$ as the subset of $\mathbb{P}\mathcal{T}_{S, \underline{s}}(\underline{k})$ consisting of $(X,f)$ such that $X$ is Delaunay-generic and $f(T)$ is the Delaunay triangulation of $X$.
\end{definition} 
By abuse of notation, we use $T$ to denote the Delaunay triangulation of $X$ for $(X, f) \in D_{S,\underline{s}}^{Del}(\underline{k};T)$.

Similarly to Lemma~\ref{lem:genericdelaunay}, the projection $\pi: \mathbb{P}\mathcal{T}_{S, \underline{s}}(\underline{k}) \to \mathbb{P}\Omega(\underline{k})$ restricts to a homeomorphism from $D_{S,\underline{s}}^{Del}(\underline{k};T)$ onto its image in $\mathbb{P}\Omega(\underline{k})$. We continue to use the same notation for this image and refer to it as a Delaunay region in $\mathbb{P}\Omega(\underline{k})$.

\subsection{Spanning Tree Coordinate Charts}\label{sec:spanningtreechartinfinitecase}

We introduce the following notions of spanning trees.
\begin{definition}\label{def:spanningtree}

We say that an embedded tree $F$ on $(S,\underline{s})$ is a \textbf{spanning tree with respect to $\{s_1, \ldots, s_{n-1}\}$} if the vertices of $F$ are $\{s_1, \ldots, s_{n-1}\}$.

If $X = (M,\underline{x},h)$ is a planar flat sphere, a geometric tree $F$ in $X$ is a \textbf{spanning tree of $X$} if it is a spanning tree of $(M,\underline{x})$ with respect to $\{x_1,\ldots,x_n\}$.

Moreover, let $X$ be a flat sphere and $D$ be a convex hull on a flat sphere $X$. We say that a geometric tree $F$ on $X$ is a \textbf{spanning tree of $D$} if $F$ is contained in $D$ and includes all the singularities in $D$. A geometric forest $F$ is called \textbf{a spanning forest of disjoint convex hulls $D_1, \ldots, D_s$} if $F$ has $s$ connected components, and each connected component of $F$ is a spanning tree of some $D_i$.
\end{definition}

Let $T$ be a triangulation of $(S, \underline{s})$ with respect to $\{s_1, \ldots, s_{n-1}\}$, and let $F$ be a spanning tree of $(S, \underline{s})$ with respect to $\{s_1, \ldots, s_{n-1}\}$ consisting of edges of $T$. Then, for $(X, f)\in D_{S,\underline{s}}^{Del}(\underline{k};T)$, $f(F)$ is a spanning tree of $X$. For simplicity, we will also use the notation $F$ to refer to the image $f(F)$ in $X$.

A coordinate chart on $D_{S,\underline{s}}^{Del}(\underline{k};T)$ associated with the spanning tree $F$ can be constructed as follows. 

For any $(X, f) \in D_{S,\underline{s}}^{Del}(\underline{k};T)$, let $D$ denote the core of $X$. By definition, $T$ is the Delaunay triangulation of $X$, and $F$ is a spanning tree of $X$. We fix a vertex on the boundary of $T$, denoted by $x$. Then $f(x)$ is a singularity in $\partial D$, which we also denote by $x$ for simplicity.

The components of $D \setminus F$ in $X$ are simply connected domains. Let $A_1, \ldots, A_m$ denote these components. Since $F$ is a spanning tree, each $A_i$ has exactly one edge on its boundary that is contained in $\partial D$ but not in $F$. We denote this edge by $e(A_i)$.

By choosing a developing map $Dev$, we proceed as follows:
\begin{enumerate}
    \item We orient the boundary $\partial D$ consistently with the orientation of $D$. Then, starting from the selected singularity $x$, we develop $\partial D$ into the plane.
    \item We further develop each component $A_i$ into the plane such that the edge $e(A_i)$ aligns with the corresponding edge in the development of $\partial D$ from step~(1).
\end{enumerate}

Note that in step~(2), an edge of $F$ may be developed into two vectors in the plane. To resolve this ambiguity, we choose the vector so that a developed component is always adjacent to the vector on its left.

From the above procedure, we denote the complex vectors associated with the edges of $F$ by
$$
(z_0, \ldots, z_{n-3}) \in \mathbb{C}^{n-3}.
$$
We refer to this as the \textbf{unprojectivized spanning tree parameter} of $F$ in $X$.

Note that the case where $D$ is degenerate requires a slight modification; see Figure~\ref{fig:degeneratedcore}. A degenerate core in a planar flat sphere consists of a single geodesic joining two singularities, possibly with singularities of curvature zero lying in its interior. In this case, each component $A_i$ is empty, and the spanning tree $F$ lies entirely in the boundary $\partial D$. Step~(1) then develops $F$ into the Euclidean plane, and we select a vector $z_i$ for each edge of $F$, for $0 \le i \le n - 3$.

    \begin{figure}[H]
	\centering
	\includegraphics[width=\linewidth]{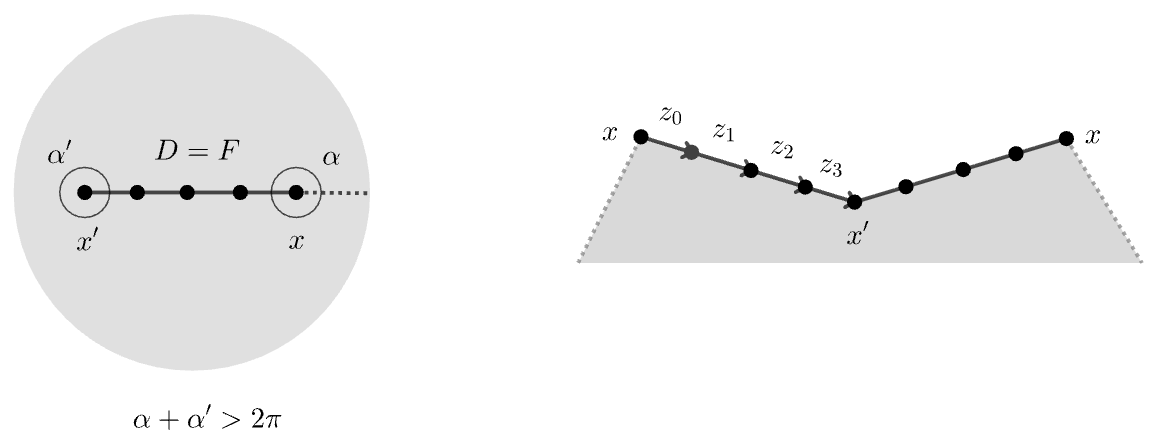}
    \caption{Left: a local picture of a neighborhood of a degenerate core. The core $D$ is a single geodesic, and the spanning tree $F$ coincides with $D$. The singularities in the interior have cone angle $2\pi$, while the sum of the cone angles at the endpoints is greater than $2\pi$. Right: a developing of the neighborhood by cutting along $\gamma$ and the dotted line. For each edge of $F$, we choose a developed image and denote by $z_i$ the corresponding vector.}\label{fig:degeneratedcore}
    \end{figure}

\begin{remark}
An unprojectivized spanning tree parameter depends on the choice of the point $x$ and the developing image of $\partial D$ in step~(1).
\end{remark}

\begin{lemma}\label{lem:spanningtreeinfinitecase}
    Let $\underline{k}$ be a curvature vector of planar type with a positive curvature gap. Let $T$ be a triangulation of $(S,\underline{s})$ with respect to $\{s_1,\ldots,s_{n-1}\}$ and let $F$ be a spanning tree contained in $T$. Then, the map
    $$
    \phi_F: D_{S,\underline{s}}^{Del}(\underline{k};T) \to \mathbb{C}P^{n-3}, \quad (X, f) \mapsto [z_0, \ldots, z_{n-3}]
    $$
    is injective, where $(z_0, \ldots, z_{n-3})$ is the unprojectivized spanning tree parameter of $F$.
\end{lemma}

\begin{proof}
For $(X,f)$ and $(X',f') \in D_{S,\underline{s}}^{Del}(\underline{k};T)$, denote their cores by $D$ and $D'$, respectively. Then $f(T)$ and $f'(T)$ are Delaunay triangulations of $D$ and $D'$, respectively. Assume that
$\phi_F(X,f) = \phi_F(X',f')$. Then the spanning tree coordinates of $F$ on $D$ and $D'$ are the same. In particular, there is an isometry from each component of $D\setminus F$ to the corresponding component of $D'\setminus F$, and these isometries glue to an isometry $h:D\to D'$. Moreover, $h\circ f|_T$ is isotopic to $f'|_T$ relative to the vertices of $T$.
 
We now show that $h:D\to D'$ extends to an isometry
$X\setminus D \to X'\setminus D'$. By applying Thurston surgeries along $D$ and $D'$, we obtain top surfaces $X^{(0)}$ and $(X')^{(0)}$. These surfaces are isometric, since they are both infinite cones with curvature $2-k_n$ at the apex. Note that $X\setminus D$ and $X'\setminus D'$ are identified with subsets of $X^{(0)}$ and $(X')^{(0)}$, respectively. The restriction
$h|_{\partial D}:\partial D\to \partial D'$
is an isometry between the boundaries of $X\setminus D$ and $X'\setminus D'$. Hence $h|_{\partial D}$ extends, by radial extension from the apex, to an isometry
$X^{(0)}\to (X')^{(0)}$, and in particular to an isometry
$X\setminus D \to X'\setminus D'$.
Combining this with the previous paragraph, we conclude that $h\circ f$ is isotopic to $f'$ relative to $\underline{s}$. Therefore, $(X,f)$ and $(X',f')$ define the same point.
\end{proof}

\subsection{Shortest Spanning Tree}
For later use, we study a special spanning tree in planar flat spheres.

Let $X$ be a flat sphere. We have previously considered only embedded trees in $X$. We define an \textbf{immersed tree} $F$ in $X$ an ordered pair $(V,E)$ satisfying that:
\begin{itemize}
    \item A subset $V$ of the conical singularities, whose elements are called \textbf{vertices}.
    \item A set $E$ of \textbf{edges}, where the edges are simple arcs joining vertices in $(S, \underline{s})$.
    \item $F$ is connected and has no cycle.
\end{itemize}
Note that the edges of $F$ may intersect in their interiors.

Furthermore, let $D$ be a convex hull in $X$. As in Definition~\ref{def:spanningtree}, an immersed tree $F$ is a \textbf{spanning tree} of $D$ if $F$ is contained in $D$ and the vertices are all the singularities in $D$.

We define $||F||_1$ as the sum of the metric lengths of the edges of $F$ in $X$, and $||F||_{\infty}$ as the maximal metric length of the edges of $F$. Let $X$ be a flat sphere, and let $D$ be a convex hull in $X$. A \textbf{shortest spanning tree} $F$ of $D$ is a spanning tree that minimizes $||F||_1$ among all spanning immersed trees of $D$.

\begin{lemma}
    Let $X$ be a convex flat sphere, and let $D$ be a convex hull in $X$. Then any shortest spanning tree $F$ of $D$ is a geometric embedded tree.
\end{lemma}

\begin{proof}
    Assume that two edges $e$ and $e'$ of $F$ intersect, and denote their intersection point by $y$. Let $A$ and $B$ be the endpoints of $e$, and let $C$ and $D$ be the endpoints of $e'$. We prove below that in this case, we can find a shorter spanning tree of $D$. There are two possible cases when removing $e$ and $e'$ from $F$.  
    
    If $B$ and $D$ belong to the same connected component of $F \setminus (e \cup e')$, we modify the tree in a neighborhood of the point $y$ as shown on the left of Figure~\ref{fig:resolve}.  
    If $A$ and $D$ belong to the same connected component of $F \setminus (e \cup e')$, we modify the tree in a neighborhood of the point $y$ as shown on the right of Figure~\ref{fig:resolve}.  
    
    Note that in either case, the resulting structure is an immersed tree with a shorter total length.

    \begin{figure}[!htbp]
	\centering
	\includegraphics[width=0.8\linewidth]{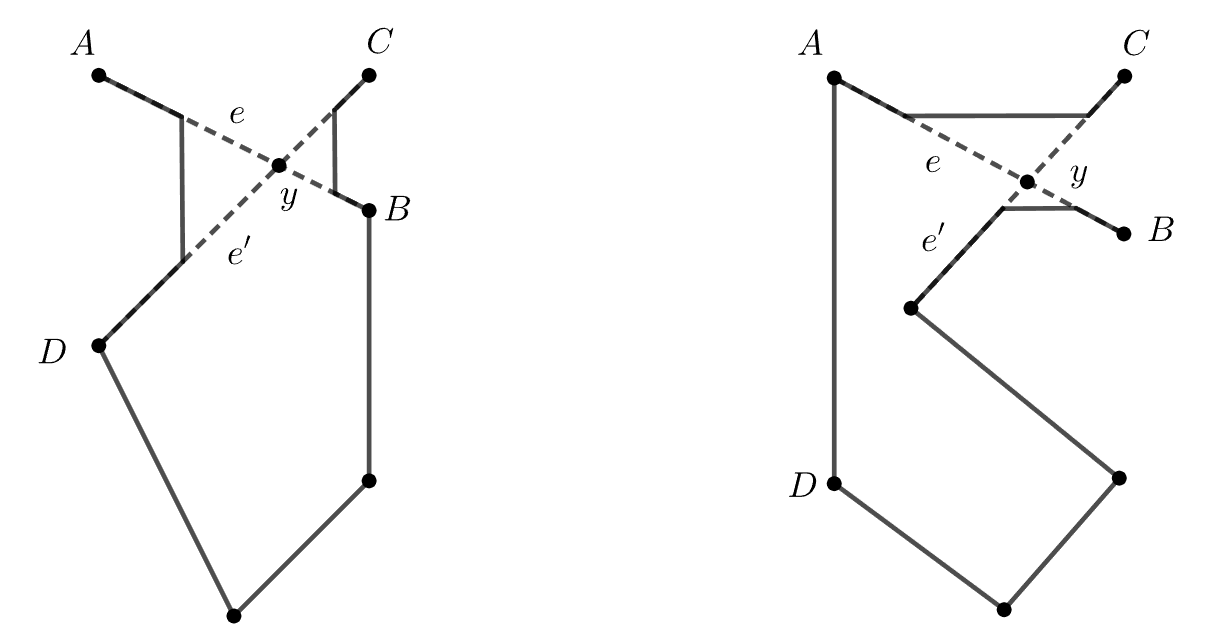}
    \caption{Shortening the sum of the lengths of $e'$ and $e_1$.}\label{fig:resolve}
    \end{figure}

    Furthermore, the edges of the shortest spanning tree $F$ must be geodesics that do not pass through any singularities. Hence, the edges are saddle connections. 
\end{proof}

The following result is known for a given set of points in the plane; see Section~9 of~\cite{deBerg2008} for a proof. The same method can be directly generalized to flat spheres, so we omit the proof here.

\begin{lemma}
    Let $X$ be a planar flat sphere. Then the edges of a Delaunay triangulation of $X$ contain a shortest spanning tree of $X$.
\end{lemma}

\subsection{Geometric Types}
For later use, we define the following notion of geometric type of shortest spanning forests. 
\begin{definition}
For $X, X'$ in the moduli space $\mathbb{P}\Omega(\underline{k})$ of planar flat spheres, let $D$ and $D'$ denote the cores of $X$ and $X'$, respectively. Let $T$ and $T'$ be Delaunay triangulations of $D$ and $D'$, and let $F$ and $F'$ be geometric spanning trees of $D$ and $D'$ with edges contained in the edges of $T$ and $T'$, respectively. 

We say that the pairs $(X, F)$ and $(X', F')$ are \textbf{geometrically equivalent} if there exists an orientation-preserving homeomorphism $f: X \to X'$ such that $f$ maps $T$ to $T'$ and $F$ to $F'$.

The \textbf{geometric type} of $F$ on $X$ is defined as the set of all pairs $(X', F')$ that are geometrically equivalent to $(X, F)$, denoted by $[X, F]$.
\end{definition}

\begin{definition}\label{def:domainofgeometrictype}
    Let $[X',F']$ be a geometric type of spanning trees on the spheres in $\mathbb{P}\Omega(\underline{k})$. We define a subset $D(\underline{k};[X',F'])\subseteq\mathbb{P}\Omega(\underline{k})$ satisfying that for $X\in D(\underline{k};[X',F'])$
\begin{enumerate}
    \item $X$ is Delaunay generic.
    \item The shortest spanning tree $F$ in $X$ is unique.
    \item The geometric type of $(X,F)$ is $[X',F']$.
\end{enumerate}
\end{definition}

\begin{remark}\label{rmk:chartonregionofgeometrictype}
By the definition of geometric types, there exists a Delaunay region $D_{S,\underline{s}}^{Del}(\underline{k};T)$ such that $D(\underline{k};[X',F']) \subseteq D_{S,\underline{s}}^{Del}(\underline{k};T)$, and the shortest spanning tree is a fixed subtree of $T$, which we also denote by $F'$. Therefore, the spanning tree coordinate chart associated with $F'$ is defined on $D(\underline{k};[X',F'])$.
\end{remark}

Since there are only finitely many geometric types of spanning trees for $X\in \mathbb{P}\Omega(\underline{k})$, we have the following result.

\begin{lemma}\label{lem:decompinfinitecase}
Let $\underline{k}$ be a curvature vector of planar type. Then there exist finitely many geometric types $[X^b, F^b]$ for $b = 1, \dots, q$ such that the complement of the union
$$
\bigsqcup_{b=1}^q D(\underline{k};[X^b, F^b])
$$
consists of spheres that are either not Delaunay-generic or have multiple shortest spanning trees.
\end{lemma}

\section{Fibered Coordinate Charts Near Boundary Strata}\label{bigsec:coor}

In this section, we study the structure near the boundary strata of the metric completion $\overline{\mathbb{P}\Omega}(\underline{k})$ of the moduli space of convex flat cone spheres.

\subsection{Thick and Thin Parts}
In this subsection, we recall the study of thick and thin parts as presented in~\cite{fu2024uniformlengthestimatestrajectories}, Section~5.

\begin{definition}[Thick parts]
Let $\underline{k}\in(0,1)^n$ be a curvature vector. Given a positive number $\lambda$, we define the \textbf{thick part} $\mathbb{P}\Omega(\underline{k})_{\lambda}$ as the subset of $\mathbb{P}\Omega(\underline{k})$ consisting of flat cone spheres where the normalized lengths of all saddle connections are bounded below by $\lambda$.
\end{definition}

\begin{remark}
    For a boundary stratum $B=B(P)$ of $\overline{\mathbb{P}\Omega}(\underline{k})$, Lemma~\ref{lem:thurstoncompletion} implies that $B$ is isometric to the moduli space $\mathbb{P}\Omega(\underline{k}(P))$. We denote by $B_{\lambda}$ the thick part of the corresponding moduli space.
\end{remark}

\begin{definition}
    Let $\underline{k}\in(0,1)^n$ be a curvature vector. Given an integer $1 \leq d \leq n-3$ and a positive real number $\varepsilon$, we define a subset $U_{d,\varepsilon}$ of $\mathbb{P}\Omega(\underline{k})$ as follows:
\begin{equation}\label{equ:ngbh}
    \begin{aligned}
    U_{d,\varepsilon} := &\left\{X \in \mathbb{P}\Omega(\underline{k}) :
    \mbox{there is a geometric forest on } X \mbox{ with } d \mbox{ edges} \right.\\
        &\left. \mbox{ such that } \ell_{e}(X) < \varepsilon \mbox{ for each edge } e \right\},
    \end{aligned}
\end{equation}
where $\ell_{e}(X)$ is the normalized length of the saddle connection $e$. We call a geometric forest in $X$ that satisfies the condition of $U_{d,\varepsilon}$ an \textbf{$\varepsilon$-geometric forest}.
\end{definition}

We will compute the volume of $U_{d,\varepsilon}$ in Section~\ref{bigsec:measurebounds}. To study this subset, we use the notion of thin parts following~\cite{fu2024uniformlengthestimatestrajectories}.

Let $B(P)$ be a boundary stratum corresponding to a $\underline{k}$-admissible partition $P$ of curvatures, and let $p_1, \ldots, p_s$ be the non-singleton elements in $P$. A geometric forest $F$ is \textbf{associated with the boundary stratum} $B(P)$ of $\overline{\mathbb{P}\Omega}(\underline{k})$ if $F$ consists of $s$ connected components $F_1, \ldots, F_s$, where the singularities in $F_i$ are those with curvatures in $p_i$ for each $i$. 

Note that if $F$ is a spanning forest of disjoint convex hulls $D_1, \ldots, D_s$, then $F$ is associated with $B(P)$ if and only if these convex hulls are associated with $B(P)$. Furthermore, recall that if $D_1,\ldots, D_s$ are associated with $B(P)$, then applying Thurston surgeries to $X$ along $D_1, \ldots, D_s$ results in a top flat cone sphere $X^{(0)}$ contained in $B(P)$.

\begin{definition}[Thin parts]\label{def:thinparts}
    Let $\underline{k}\in(0,1)^n$ be a curvature vector. Let $U_{d,\varepsilon}$ be the subset defined in~\eqref{equ:ngbh}. Let $B$ be a $d$-codimensional boundary stratum of $\overline{\mathbb{P}\Omega}(\underline{k})$. When $1\leq d < n-3$, we define the \textbf{thin part $U_{d,\varepsilon}(B_{\lambda})$ associated with $(B,\varepsilon,\lambda)$} as the subset of $U_{d,\varepsilon}$ consisting of flat spheres $X$ such that there is an $\varepsilon$-geometric forest $F$ in $X$ satisfying:
    \begin{itemize}
        \item The forest $F$ is associated with $B$.
        \item There exist convex hulls $D_1,\ldots,D_s$ in $X$ such that $F$ is a spanning forest of them.
        \item By applying Thurston surgeries along all $D_i$, the resulting top surface $X^{(0)}$ is contained in the thick part $B_{\lambda}$.
    \end{itemize}
    When $d = n-3$, we define the \textbf{thin part $U_{n-3,\varepsilon}(B)$ associated with $(B,\varepsilon)$} similarly, except that $B_{\lambda}$ is replaced by $B$, or equivalently, let $\lambda = 0$.
\end{definition}

\begin{lemma}[{\cite[Proposition~5.6]{fu2024uniformlengthestimatestrajectories}}]\label{lem:epsilonconvexhulls}
Let $\underline{k}\in(0,1)^n$ be a curvature vector with positive curvature gap $\delta(\underline{k})$, and let $B$ be a boundary stratum of $\overline{\mathbb{P}\Omega}(\underline{k})$. Assume that $\varepsilon$ and $\lambda$ are positive numbers satisfying the following conditions:
\begin{equation}\label{equ:restriction}
    0 < \varepsilon \leq \frac{\delta(\underline{k})^2}{4n^2} 
    \quad \text{and} \quad 
    \lambda = \bigg(1 + \frac{n}{\delta(\underline{k})}\bigg)\varepsilon.
\end{equation}
Then, for any $X \in U_{d,\varepsilon}(B_{\lambda})$, the convex hulls $D_1, \ldots, D_s$ satisfying Definition~\ref{def:thinparts} on $X$ are precisely the convex hulls in $X$ that contain a saddle connection with a normalized length less than $\varepsilon$.
\end{lemma}

\noindent\textit{\textbf{Convention}} Throughout the rest of the paper, we assume that the parameters $\varepsilon$ and $\lambda$ for the thin part $U_{d,\varepsilon}(B_{\lambda})$ satisfy the condition~\eqref{equ:restriction}.\newline

From the uniqueness established above, we introduce the following definition:
\begin{definition}[$(d,\varepsilon)$-convex hull]
Let $\underline{k}\in(0,1)^n$ be a curvature vector. Let $U_{d,\varepsilon}(B_{\lambda})$ denote a thin part of $\mathbb{P}\Omega(\underline{k})$, where $\varepsilon$ and $\lambda$ satisfy the condition~\eqref{equ:restriction}. The convex hulls $D_1, \ldots, D_s$ on $X \in U_{d,\varepsilon}(B_{\lambda})$ that satisfy Definition~\ref{def:thinparts} are referred to as the \textbf{$(d,\varepsilon)$-convex hulls} on $X$.
\end{definition}

In~\cite{fu2024uniformlengthestimatestrajectories}, it is proved that the thin parts have the following inductive relations:
\begin{lemma}[{\cite[Theorem~5.12]{fu2024uniformlengthestimatestrajectories}}]\label{lem:epsilondecomp}
    Let $\underline{k}\in(0,1)^n$ be a curvature vector with positive curvature gap $\delta(\underline k)$. Assume that $\varepsilon$ and $\lambda$ satisfy the condition~\eqref{equ:restriction}.  Then we have that: 
    \begin{itemize}
        \item When $1\le d < n-3$, the subset $U_{d,\varepsilon}$ decompose into a disjoint union of thin parts
        $$\underset{\mathrm{codim}(B)=d}{\bigsqcup} U_{d,\varepsilon}(B_\lambda)$$
        where the union is over the $d$-codimensional boundary strata of $\overline{\mathbb{P}\Omega}(\underline{k})$,
        and the complement which satisfies that
        $$U_{d,\varepsilon}\setminus \underset{\mathrm{codim}(B)=d}{\bigsqcup} U_{d,\varepsilon}(B_\lambda) 
        \subset
        U_{d + 1, \frac{6n}{\delta^2}\varepsilon}.$$
        \item When $d = n-3$, the subset $U_{n-3,\varepsilon}$ is a disjoint union
        $$U_{n-3,\varepsilon}=\underset{\dim B = 0}{\bigsqcup}U_{n-3,\varepsilon}(B)$$
        where the union is over the $0$-dimensional boundary strata of $\overline{\mathbb{P}\Omega}(\underline{k})$.
    \end{itemize}
\end{lemma}

We point out that in the definition of the thick part $\mathbb{P}\Omega(\underline{k})_{\lambda}$, the number $\lambda$ is a lower bound for all saddle connections, including those with identical endpoints. Thus, the complement of this thick part is not $U_{1,\lambda}$. However, we have the following lemma:

\begin{lemma}[{\cite[Lemma~4.16]{fu2024uniformlengthestimatestrajectories}}]\label{lem:closedtosimple}
    Let $X$ be a convex flat cone sphere with $n$ singularities and positive curvature gap $\delta$. Assume that $\gamma_1$ is a simple saddle connection with the same endpoint whose normalized length is $L$. Then there exists a simple saddle connection $\gamma$ with distinct endpoints such that its normalized length is at most $\frac{1}{\delta}L$. 
\end{lemma}

As a consequence of Lemma~\ref{lem:closedtosimple}, we obtain the following result on the complement of a thick part:

\begin{lemma}\label{lem:complementofthick}
    Let $\underline{k}\in(0,1)^n$ be a curvature vector with positive curvature gap $\delta(\underline{k})$. For any $\lambda$, the complement of the thick part $\mathbb{P}\Omega(\underline{k})_{\lambda}$ is contained in $U_{1,\lambda/\delta(\underline{k})}$.
\end{lemma}

\subsection{Projection and Real Rays on Fibers}
We recall the projections into the boundary strata as described in Section~3 of~\cite{thu} and present a rigorous formulation of the structure of real rays on fibers.

Let $\underline{k}\in(0,1)^n$ be a curvature vector with positive curvature gap $\delta(\underline{k})$, and let $B$ be a boundary stratum of $\overline{\mathbb{P}\Omega}(\underline{k})$. Let $\varepsilon$ and $\lambda$ be two constants satisfying the condition~\eqref{equ:restriction}.

We define the projection
\begin{equation}\label{equ:projection}
    \phi_{B}: U_{d,\varepsilon}(B_{\lambda}) \to B_{\lambda}, \quad X \mapsto X^{(0)},
\end{equation}
where $X^{(0)}$ is obtained by applying Thurston surgeries along the $(d,\varepsilon)$-convex hulls in $X$.

We define the structure of ``real rays" in each fiber of $\phi_B$. For this purpose, we introduce the concept of inverse Thurston surgeries.

Let $P$ be the $\underline{k}$-admissible partition associated with the boundary stratum $B$. Denote by $p_1, \ldots, p_s$ the non-singleton elements in $P$. For $X_0 \in B$, let $x_{p_i}$ be the singularity associated with the element $p_i$ see Definition~\ref{def:singularityassociatedtononsingletonelement}. We now define the following:
\begin{itemize}
    \item (Reference convex hulls) For each $p_i$, we take an infinitesimal flat sphere $X_i$ in $\mathbb{P}\Omega(\underline{k}(p_i))$, and denote its core by $D_i$. Each $X_i$ is called the \textbf{reference infinitesimal sphere associated with $p_i$}, and each $D_i$ is called the \textbf{reference convex hull associated with $p_i$}.
    
    \item (Reference points) For each reference convex hull $D_i$, we select a singularity on the boundary $\partial D_i$, denoted by $x(D_i)$. This point is called the \textbf{reference point associated with $p_i$}.
    
    \item (Reference segments) A simple geodesic $\gamma_i: [a, b] \to X_0$ is called a \textbf{reference segment associated with $p_i$} if $\gamma_i(a) = x_{p_i}$ and $y_i := \gamma_i(b)$ is a regular point.
\end{itemize}
We call the sequence:
$$
\mathcal{R}:= (X_i, x(D_i), \gamma_i)_{i=1}^s
$$
a \textbf{reference data of $X_0$ associated with $B$}.

Let $C_{D_i}$ be the replacement cone for the Thurston surgery along $D_i$ in $X_i$. Note that the curvature of the apex of $C_{D_i}$ matches the curvature of the singularity $x_{p_i}$. It follows that if $\gamma_i$ is sufficiently short, there exists a unique scalar $a_i \in \mathbb{R}_{>0}$ such that an isometric embedding
\begin{equation}\label{equ:embeddingofboundedcone}
    h_i: a_i \cdot C_{D_i} \to X_0
\end{equation}
can be constructed, where $a \cdot C_{D_i}$ denotes a bounded cone obtained by rescaling the metric of the replacement cone $C_{D_i}$ by $a$. This embedding maps the apex of $a \cdot C_{D_i}$ to $x_{p_i}$ and the reference point $x(D_i)$ to $y_i$. Note that such an isometry is unique.

\begin{definition}[Inverse Thurston surgery]
For $X_0 \in B_\lambda$, given a reference data $\mathcal{R}$ of $X_0$ associated with $B$, the \textbf{inverse Thurston surgery with respect to} $\mathcal{R}$ is the operation on $X_0$ as follows:
\begin{itemize}
    \item For each $i$, consider the embedding~\eqref{equ:embeddingofboundedcone} $h_i: a_i \cdot C_{D_i} \to X_0$. Remove the image of the rescaled cone $a_i \cdot C_{D_i}$ from $X_0$.
    \item Let $a_i \cdot D_i$ denote the surface obtained by rescaling the metric of $D_i$ by the factor $a_i$. Glue each rescaled piece $a_i \cdot D_i$ back to the surface along the corresponding boundary $a_i \cdot \partial D_i$.
\end{itemize}
Denote the resulting flat cone sphere by $X_0(\mathcal{R})$.
\end{definition}

\begin{remark}
To clarify that $X_0(\mathcal{R})$ is an element of $\mathbb{P}\Omega(\underline{k})$, we need to specify how its singularities are labeled. Note that the singularities of $X_0$ are labeled as introduced in Remark~\ref{rmk:canonicallabels}. The singularities of $X_0(\mathcal{R})$ that are not among the singularities $x_{p_i}$ inherit their labels from $X_0$. For each convex hull $D_i$ of $X_0(\mathcal{R})$, the singularities in $D_i$ inherit their labels from the core of $X_i$.
\end{remark}

\begin{remark}\label{rmk:inversethurstonsurgery}
For any $X \in U_{d,\varepsilon}(B_{\lambda})$, let $D_1, \ldots, D_s$ be the $(d,\varepsilon)$-convex hulls. Set $X_0 = X^{(0)}$, and let $X_i$ be the infinitesimal sphere obtained by performing the Thurston surgery along $D_i$ for each $i$. 

Choose a reference point $x(D_i) \in \partial D_i$ for each $i$. Note that the complement $X \setminus \bigsqcup_i D_i$ embeds canonically into $X_0$, so each reference point $x(D_i)$ is embedded in $X_0$ as well.

Let $\gamma_i(D_i)$ be the geodesic in $X_0$ from the singularity $x_{p_i}$ to the point $x(D_i)$, lying entirely within the replacement cone for each $i$. Consider the reference data
$$
\mathcal{R} = (X_i, x(D_i), \gamma_i(D_i))_{i=1}^s,
$$
then $X$ is isometric to the reconstructed sphere $X_0(\mathcal{R})$.
\end{remark}

Finally, we define the notion of real rays on the fibers of $\phi_B$ as follows. 
\begin{definition}[Real rays]\label{def:realrays}
    For $X_0$ in the boundary stratum $B$, let $\mathcal{R} = (X_i, x(D_i), \gamma_i)_{i=1}^s$ be a reference data associated with $B$, with $\gamma_i: [0, b_i] \to X_0$ sufficiently short. 
    For $t_i \in (0, 1]$, we define $t_i \cdot \gamma_i$ as the geodesic $t_i \cdot \gamma_i: [0, t_ib_i] \to X_0$, and define reference data:
    $$
    \mathcal{R}_{(t_1,\ldots,t_s)}:= (X_i, x(D_i), t_i\cdot\gamma_i).
    $$
    Then we call the map
    $$
    \mathbf{r}_{X_0,\mathcal{R}}: (t_1, \ldots, t_s) \mapsto X_0(\mathcal{R}_{(t_1,\ldots,t_s)}),
    $$
    for $(t_1,\ldots,t_s) \in (0, 1]^s$ \textbf{a product of real rays} on the fiber $\phi_B^{-1}(X_0)$.

    Furthermore, if we take the reference data $\mathcal{R}$ as defined in Remark~\ref{rmk:inversethurstonsurgery}), the corresponding product of real rays $\mathbf{r}_{X_0, \mathcal{R}}$ satisfies $\mathbf{r}_{X_0, \mathcal{R}}(1, \ldots, 1) = X$. We refer to this as the \textbf{product of real rays of $X$}.
\end{definition}

\begin{remark}\label{rmk:exactlyrescalingfactor}
    Let $\mathbf{r}_{X_0, \mathcal{R}}$ be the product of real rays of $X$, where $\mathcal{R} = (X_i, x(D_i), \gamma_i(D_i))_{i=1}^s$ is defined as in Remark~\ref{rmk:inversethurstonsurgery}. Then the scaling factor $a_i$ in the inverse Thurston surgery with respect to $\mathbf{r}_{X_0, \mathcal{R}}$ is equal to $t_i$.
\end{remark}

\subsection{Shortest Spanning Forest}

Let $X$ be a convex flat sphere, and let $D_1, \ldots, D_s$ be disjoint convex hulls in $X$. Let $F$ be a spanning forest of these convex hulls. We say that $F$ is a \textbf{shortest spanning forest} if the component in $D_i$ is the shortest spanning tree of $D_i$ (see Definition~\ref{def:spanningtree}). 

\begin{lemma}\label{lem:byshortestone}
Let $F$ be a shortest spanning forest of the convex hulls $D_1, \ldots, D_s$ in a convex flat cone sphere $X$. Then there exists a geometric spanning forest $F'$ of $D_1, \ldots, D_s$ such that $||F'||_\infty < \varepsilon$ if and only if $F$ satisfies $||F||_\infty < \varepsilon$.
\end{lemma}

\begin{proof}
The direction from $||F||_\infty < \varepsilon$ to the existence of $F'$ with $||F'||_\infty < \varepsilon$ is straightforward. For the reverse direction, it suffices to consider the case $s = 1$. Assume $||F||_\infty \geq \varepsilon$. Let $e$ be an edge of $F$ with metric length $|e| \geq \varepsilon$. Removing $e$ from $F$ disconnects the tree. Since $F'$ joins all  the singularities of $D_1$, there is an edge $e'$ of $F'$ such that $(F \setminus e)\cup e'$ is an immersed spanning tree. Since $|e'| < \varepsilon \leq |e|$, the length $||F_1||_1$ is strictly less than $||F||_1$. This contradicts the assumption that $F$ is a shortest spanning tree.
\end{proof}

\subsection{Fibered Coordinate Charts}\label{sec:fiberedcoordinatechart}
In this section, we define local coordinate charts on thick parts. 

Let $B$ be a boundary stratum of  $\overline{\mathbb{P}\Omega}(\underline{k})$. Let $P$ be the $\underline{k}$-admissible partition corresponding to $B$, and let $p_1, \ldots, p_s$ be the non-singleton elements. Recall that $B$ is isometric to the moduli space $\mathbb{P}\Omega(\underline{k}(P))$, where $\underline{k}(P)$ is defined in~\eqref{equ:curvaturenotations0}.

To introduce the domain on which the coordinate charts are built, we fix a choice of the following data:
\begin{itemize}
    \item A Delaunay region $D_{S,\underline{s}(P)}^{Del}(\underline{k}(P);T)$ in the boundary stratum~$B$, where $(S,\underline{s}(P))$ is the base sphere for the Teichmüller space $\mathbb{P}\mathcal{T}(\underline{k}(P))$, as defined in~\eqref{equ:boundarybasesphere}.
    \item A sequence $([X'_i, F'_i])_{i=1}^s$ of geometric types of spanning trees for the spheres in $\mathbb{P}\Omega(\underline{k}(p_i))$.
\end{itemize}

Next, we define the subset $N\big(T, ([X'_i, F'_i])_{i=1}^s\big) \subseteq U_{d,\underline{\varepsilon}}(B_{\lambda})$ satisfying that:
\begin{enumerate}
    \item $\phi_B\left(N\big(T, ([X'_i, F'_i])_{i=1}^s\big)\right) = B_{\lambda} \cap D_{S,\underline{s}(P)}^{Del}(\underline{k}(P);T)$.
    \item For $X \in N\big(T, ([X'_i, F'_i])_{i=1}^s\big)$, let $D_1, \ldots, D_s$ be the $(d,\varepsilon)$-convex hulls in $X$,  and $X_i$ be the infinitesimal flat sphere associated to $D_i$. Then for each $i$:
    \begin{enumerate}
        \item Each $D_i$ is Delaunay-generic.
        \item The shortest spanning tree of $D_i$ is unique, denoted by $F_i$.
        \item The pair $(X_i, F_i)$ has the geometric type $[X'_i, F'_i]$, or equivalently .
    \end{enumerate}
\end{enumerate}

\begin{remark}\label{rmk:projectiontoinfinitesimalspheres}
    According to the second condition of $N\big(T, ([X'_i, F'_i])_{i=1}^s\big)$, we know that for any $X\in N\big(T, ([X'_i, F'_i])_{i=1}^s\big)$,  the infinitesimal sphere $X_i$ satisfies that
    $$
    X_i\in D\big(\underline{k}(p_i); [X'_i,F'_i]\big)
    $$
    for each $i$, where $D\big(\underline{k}(p_i); [X'_i,F'_i]\big)$ is a subset of $\mathbb{P}\Omega(\underline{k}(p_i))$ defined in Definition~\ref{def:domainofgeometrictype}.
\end{remark}

For $X_0 \in B$, let $x_{p_i}$ be the singularity associated with the element $p_i$ for each $i$ (see Definition~\ref{def:singularityassociatedtononsingletonelement}). For the Delaunay region $D_{S,\underline{s}(P)}^{Del}(\underline{k}(P);T)$ in $B$, we select an edge of $T$ incident to $x_{p_i}$, denoted by $e'_i$. Note that $e'_i$ might be the same for different $x_{p_i}$. Consider the following subset $\mathcal{S} \subset U_{d, \varepsilon}(B_\lambda)$:
\begin{equation}\label{equ:seam}
    \begin{aligned}
        \mathcal{S} := \big\{ & X_0(\mathcal{R}) \in U_{d, \varepsilon}(B_\lambda) \mid X_0 \in D_{S,\underline{s}(P)}^{Del}(\underline{k}(P);T) \text{ and $\mathcal{R}$ is a reference data  associated with $B$}
        \\ & \text{ such that the reference segment } \gamma_i \text{ lies in the edge } e'_i \text{ for some } 1\le i \le s \big\}.
    \end{aligned}
\end{equation}
We call $\mathcal{S}$ a \textbf{cut locus in $U_{d, \varepsilon}(B_\lambda)$ associated with $D_{S,\underline{s}(P)}^{Del}(\underline{k}(P);T)$ and the edges $(e'_i)_{i=1}^s$}. Note that by definition, if $X$ is in $\mathcal{S}$, then the product of real rays of $X$ also lies in $\mathcal{S}$.\newline

\noindent\textbf{Goal}: We aim to construct a coordinate chart on $N\big(T, ([X'_i, F'_i])_{i=1}^s\big) \setminus \mathcal{S}$.\newline

Assume that $B$ is $d$-codimensional.
\begin{enumerate}
\item By selecting a spanning tree $F_0$, the region $D_{S,\underline{s}(P)}^{Del}(\underline{k}(P);T)$ is contained in the spanning tree coordinate chart on $D_{S,\underline{s}(P)}(F_0)$. For~$X \in N\big(T, ([X'_i, F'_i])_{i=1}^s\big) \setminus \mathcal{S}$, denote by
$$
(w_0,w_1,\ldots,w_{n-3-d}) \in \mathbb{C}^{n-2-d}
$$
an unprojectivized spanning tree parameter of $X$. By normalizing $w_0$ in $[w_0,w_1\ldots,w_{n-3}]$ to $1$, we obtain the spanning tree coordinate of $\phi_B(X)$:
$$
W := (w_1, \ldots, w_{n-3-d}) \in \mathbb{C}^{n-3-d}.
$$
\end{enumerate}

Next, we define the coordinates concerning the $(d,\varepsilon)$-convex hulls as follows. While the underlying idea is straightforward, the process of uniquely associating vectors to spanning trees of the convex hulls involves subtle details.

For any $X_0 \in D_{S,\underline{s}(P)}^{Del}(\underline{k}(P);T)$, we fix an unprojectivized parameter $(w_0,w_1,\ldots,w_{n-3})$ for $X_0$. Recall that $x_{p_i}$ is the singularity associated with the element $p_i$, and we select an edge $e'_i$ of $T$ incident to $x_{p_i}$. Consider a simple loop $l_i: S^1 \to X_0$ around the vertex $x_{p_i}$, oriented counterclockwise and transverse to the edges of $T$, and assume that $l_i(1)$ lies on $e'_i$. Starting from $l_i(1)$, we develop the triangles of $T$ traversed by $l_i$ until it completes a loop back to $l_i(1)$. This process produces a sector domain, denoted by $Sec(x_{p_i})$, whose two radial boundaries are copies of $e'_i$, labeled $r'_i$ and $r''_i$ in counterclockwise order. Note that the sector $Sec(x_{p_i})$ is independent of the specific choice of $l_i$ and is unique up to a Euclidean isometry of the plane. An example of the sector $Sec(x_{p_i})$ is shown in Figure~\ref{fig:developconvexboundary}.

The construction of the spanning tree coordinate $W$ determines a specific way to develop $X_0 \setminus F$ into the plane. We then fix the position of the sector $Sec(x_{p_i})$ in the plane so that the face in $Sec(x_{p_i})$ adjacent to the radial boundary $r'_i$ coincides with the corresponding face in the planar development of $X_0 \setminus F$.

For any $X \in N\big(T, ([X'_i, F'_i])_{i=1}^s\big)$, let $D_1, \ldots, D_s$ be the $(d, \varepsilon)$-convex hulls in $X$, and let $X_i$ be the infinitesimal sphere associated to $D_i$. Let $F$ be the unique shortest spanning forest of $D_1, \ldots, D_s$, and denote the component in $D_i$ by $F_i$. Importantly, we fix the same singularity in $\partial D_i$ as the reference point $x(D_i)$ for the spheres $X\in N\big(T, ([X'_i, F'_i])_{i=1}^s\big)$. 

\begin{enumerate}  
    \setcounter{enumi}{1}
    \item For $X \in N\big(T, ([X'_i, F'_i])_{i=1}^s\big) \setminus \mathcal{S}$, note that by the Thurston surgeries along $D_1,\ldots,D_s$, each reference point $x(D_i)$ is embedded into $X_0$. Then it corresponds to a unique point in the interior of the sector $Sec(x_{p_i})$ in the plane, which we also denote by $x(D_i)$.  We develop $\partial D_i$ into the plane starting from $x(D_i)$; see Figure~\ref{fig:developconvexboundary} for an example.

    \item For each $i$, we further develop the components $D_i\setminus F_i$ into the plane such that the boundary aligns with the developing of $\partial D_i$ in the step~(2), as in Section~\ref{sec:spanningtreechartinfinitecase}. Denote the unprojectivized spanning tree parameter associated to the edges of $F_i$ as $V^i := (v_{i1}, \ldots, v_{id_i})$. We define:
    \begin{equation}\label{equ:defofVparameter}
        V := (V^1, \ldots, V^s) \in \mathbb{C}^d,
    \end{equation}
    where $\sum_{i=1}^s d_i = d$.
\end{enumerate}

\begin{figure}[!htbp]
	\centering
	\includegraphics[width=\linewidth]{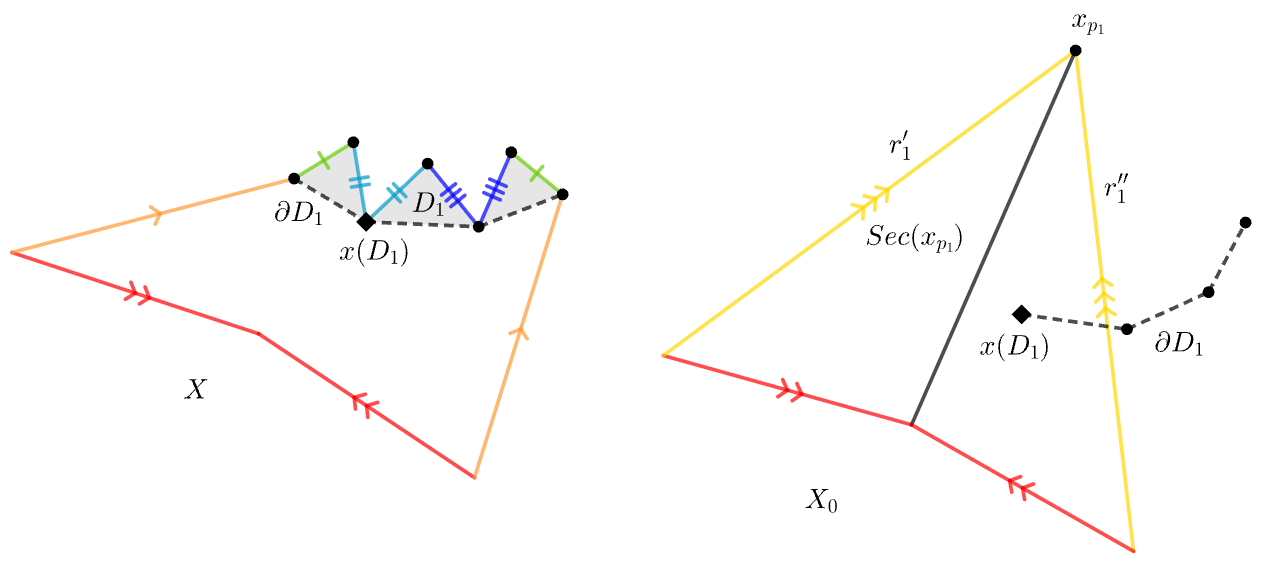}
\caption{Left: the flat cone sphere $X$ with a convex hull $D_1$, as in Figure~\ref{fig:1}. Let $X_0$ be the sphere obtained from $X$ by applying Thurston surgery along $D_1$. Right: the developing image $Sec(x_{p_1})$ of a neighborhood of $x_{p_1}$, together with the dotted boundary $\partial D_1$, developed from the reference point $x(D_1)$. }
\label{fig:developconvexboundary}
    \end{figure}

\begin{remark}
    In step~(2), we assert that there is a unique point in the interior of the sector $Sec(x_{p_i})$ corresponding to the reference point $x(D_i)$. This uniqueness is the reason why we need to remove the cut locus $\mathcal{S}$ from $N\big(T, ([X'_i, F'_i])_{i=1}^s\big)$. Without this restriction, if the reference point lies on the edge $e'_i$, it would correspond to two distinct points on the radial boundaries of the sector $Sec(x_{p_i})$.
\end{remark}

By the above, we call the vector
\begin{equation}\label{equ:unprojectivefiberedcoordinate}
    (w_0, V, w_1, \ldots, w_{n-3-d}) \in \mathbb{C}^{n-2}
\end{equation}
an \textbf{unprojectivized fibered parameter} of $X$. By normalizing the entry $w_0$ to $1$, the resulting affine coordinate is $(1, V, w_1, \ldots, w_{n-3-d})$.

We define the map:
\begin{equation}\label{equ:coordinatemapfibered}
    \Psi: N\big(T, ([X'_i, F'_i])_{i=1}^s\big) \setminus \mathcal{S} \to \mathbb{C}^{n-3}, \quad X \mapsto (V, W),
\end{equation}
where $(1, V, W)$ is the affine coordinate of $[w_0, V, w_1, \ldots, w_{n-3-d}] \in \mathbb{C}P^{n-3}$.

Recall that $P$ is the $\underline{k}$-admissible partition corresponding to $B$, and $p_1,\ldots,p_s$ are the non-singleton elements in $P$.

\begin{theorem}[Fibered coordinate chart]\label{thm:fiberedcoordinatechart}
    For any cut locus $\mathcal{S}$, the map $\Psi$ defined in~\eqref{equ:coordinatemapfibered} is an injective projective linear transformation in the local spanning tree coordinate charts. We refer to the coordinate chart on $N\big(T, ([X'_i, F'_i])_{i=1}^s\big) \setminus \mathcal{S}$ defined via the map $\Psi$ as a \textbf{fibered coordinate chart}.

    In this coordinate chart, the projection map $\phi_B$ from~\eqref{equ:projection} takes the form $\phi_B(V,W) = W$. In addition, the fiber of the projection $\phi_B$ at $W$ can be expressed as 
    \begin{equation}\label{equ:fiberexpression}
        \left\{V \in C_{([X'_i, F'_i])_{i=1}^s}(W) \mid \frac{|v_{ij}|}{\sqrt{\mathrm{Area}(1, V, W)}} < \varepsilon \text{ for each } (i, j)\right\}
    \end{equation}
    in $N\big(T, ([X'_i, F'_i])_{i=1}^s\big) \setminus \mathcal{S}$, where $C_{([X'_i, F'_i])_{i=1}^s}(W)$ is an infinite cone in $\mathbb{C}^d$ satisfying the following properties:
    \begin{enumerate}
    \item The infinite cone $C_{([X'_i, F'_i])_{i=1}^s}(W)$ is unique up to complex scaling on each component $V^i$ of $V$. More precisely, for any two choices of the parameter $W$, whether taken from different Delaunay regions $(\underline{k};T)$ or associated with different spanning trees, the resulting descriptions differ only by such a scaling. For simplicity, we also denote the cone by $C_{([X'_i, F'_i])_{i=1}^s}$.
    \item For each $i$, let $P_i$ be the partition of $\{1,\ldots,n\}$ in which $p_i$ is the only non-singleton element, and let $B_i$ be the boundary stratum corresponding to $P_i$. Let $C_{[X'_i, F'_i]}$ be the infinite cone in the expression of the fiber in fibered coordinate charts at $B_i$. Then, up to complex scaling on each $V^i$, the map
    $$V\mapsto (V^1,\ldots,V^s)$$
    is a Euclidean isometry from $C_{([X'_i, F'_i])_{i=1}^s}$ to the product $C_{[X'_1, F'_1]} \times \ldots \times C_{[X'_s, F'_s]}$.
    \end{enumerate}
\end{theorem}

\begin{proof}
We prove these statements separately.\newline

\noindent\underline{Injectivity}:
Let $\Psi(X) = (V,W)$ and $\Psi(X')=(V',W')$. Assume that $\Psi(X) = \Psi(X')$. Since $W=W'$, it follows that $\phi_B(X)$ is the same as $\phi_B(X')$, denoted by $X_0$. 

Let $(D_i)_{i=1}^s$ and $(D_i)_{i=1}^s$ be the $(d,\varepsilon)$-convex hulls in $X$ and $X'$ respectively. Let $X_i$ and $X'_i$ be the infinitesimal sphere corresponding to $D_i$ and $D_i'$. Note that the vector $V^i$ is an unprojectivized spanning tree parameter for $X_i$. Since $V = V'$, Lemma~\ref{lem:spanningtreeinfinitecase} implies that the $(\varepsilon, d)$-convex hulls in $X$ and $X'$ are metrically isometric. 

Furthermore, since $V=V'$, the reference points in $D_i$ corresponds to the same point in $Sec(x_i')$ for $X$ and $X'$. Let $\gamma_i$ be the geodesic segment joining the singularity $x_{p_i}$ to the reference point $x(D_i)$ in the replacement cone $C_{D_i}$.

Consider the reference data $\mathcal{R} = (X_i, x(D_i), \gamma_i)_{i=1}^s$. It follows that both $X$ and $X'$ are isometric to $X_0(\mathcal{R})$. Therefore, the map $\Psi$ is injective.\newline

\noindent\underline{Projective linear transform}:
To show that $\Psi$ is an projective linear transformation under spanning tree coordinate charts, it suffices to verify it is induced by a linear map from unprojectivized fibered parameters to unprojectivized spanning tree parameters. 

For $X \in N\big(T, ([X'_i, F'_i])_{i=1}^s\big) \setminus \mathcal{S}$, we first construct one spanning tree chart. Select a triangulation $T$ of $X$ such that the boundary $\partial D_i$ and the shortest spanning tree $F_i$ are contained in the edges of $T$ for each $i$. Pick a spanning tree of $T$, and let $e$ be an edge of the spanning tree. Denote by $z \in \mathbb{C}$ the corresponding coordinate of $e$ in the spanning tree coordinate chart.

If $e$ is contained in $D_i$ for some $i$, then the coordinate $z$ of $e$ can be expressed as a complex linear combination of the entries of $V \in \mathbb{C}^d$. If the endpoints of $e$ do not lie in the boundary of any $\partial D_i$, then $z$ is a complex linear combination of $(w_0, w_1, \ldots, w_{n-3-d})$.

Next, assume that the edge $e$ has one endpoint $x$ lying on a boundary $\partial D_i$. Denote the projection $\phi_B(X)$ by $X_0$. Join $x$ to the apex $x_{p_i}$ of the replacement cone $C_{D_i}$ in $X_0$ by a radial segment, denoted by $\gamma_x$. Note that the concatenation $e\gamma_x$ forms a path in $X_0$. The segment $\gamma_x$ in $C_{D_i}$ can be developed into a vector $v'$, which is a linear combination of $V$. The piecewise geodesic $e\gamma_x$ can then be developed into the plane such that the vector $a$, joining one endpoint to the other, is expressed as $v' + z$. On the other hand, since the path $e\gamma_x$ can be homotoped into the edges of $T$, the vector $a$ can also be written as a complex linear combination of $(w_0, w_1, \ldots, w_{n-3-d})$. It follows that $z$ is a complex linear combination of $(w_0, V, W)$.

The case where both endpoints of $e$ lie on the boundaries of convex hulls is analogous and thus omitted.

Therefore, the map $\Psi$ is projective linear transformation under the selected spanning tree coordinate chart.\newline

\noindent\underline{Fiber expression}:
According to the construction of the fibered coordinate chart, it is clear that the projection map $\phi_B$ from~\eqref{equ:projection} takes the form $\phi_B(V,W) = W$.

To prove the expression~\eqref{equ:fiberexpression}, for $(V,W)\in N\big(T, ([X'_i, F'_i])_{i=1}^s\big)\setminus\mathcal{S}$, denote by $C_{([X'_i, F'_i])_{i=1}^s}(W)$ the infinite cone in $\mathbb{C}^d$ that consists of the real rays $\mathbb{R}_{>0}V$ for the parameter $V$ in the coordinates $(V,W)$.

Note that the normalized length of the edges of $F$ is 
$$
\frac{|v_{ij}|}{\sqrt{\mathrm{Area}(1,V,W)}}
$$
for each $(i,j)$. Combining this with Lemma~\ref{lem:byshortestone}, we conclude that the fiber $\phi_B^{-1}(X_0)$ is injectively mapped into expression~\eqref{equ:fiberexpression} in the coordinate chart.  Surjectivity can be verified directly using inverse Thurston surgeries.\newline

\noindent\underline{Properties of the infinite cones}:
To prove~(1), recall that the parameter $V^i$ is constructed with respect to the developing of the boundary $\partial D_i$ within the sector $Sec(x_{p_i})$ in the plane. For any sphere $X_0\in B$, the neighborhood of the singularity $x_{p_i}$ are isometric. Since $Sec(x_{p_i})$ is a developing of the neighborhood of $x_{p_i}$ into the plane, we know that the angle of the sector is the same for different $X_0$.

For different $W, W'$, denote the complex vectors representing the radial boundary $r'_1$ of the sector $Sec(x_{p_i})$ by $w(r_i')$ and $w'(r_i')$, respectively.
Then, the infinite cone $C_{([X'_i, F'_i])_{i=1}^s}(W)$ maps bijectively onto $C_{([X'_i, F'_i])_{i=1}^s}(W')$ by the scaling:
$$
(V^1,\ldots,V^s)\mapsto \left(\frac{w(r_1')}{w'(r_1')}V^1,\ldots,\frac{w(r_s')}{w'(r_s')}V^s\right).
$$

The proof of~(2) is similar to that of~(1). Note that for spheres $X_0\in B$ and $X'_0\in B_i$, the neighborhoods of the singularities associated to $p_i$ in $P$ and $p_i$ in $P_i$ are isometric. It follows that the map
$$
V\mapsto (V^1,\ldots, V^s)
$$
induces an isometry from $C_{([X'_i, F'_i])_{i=1}^s}$ to the product $C_{[X'_1, F'_1]} \times \ldots \times C_{[X'_s, F'_s]}$, where each cone $C_{[X'_i, F'_i]}$ is considered up to complex scaling.
\end{proof}

\begin{remark}\label{rmk:specificprojectiontoinfinitesimalspheres}
    According to Remark~\ref{rmk:projectiontoinfinitesimalspheres}, we have the projection
    $$
    proj: C_{[X'_i, F'_i]} \to D\big(\underline{k}(p_i); [X'_i,F'_i]\big),\quad V^i\to [V^i].
    $$
    Here, $D\big(\underline{k}(p_i); [X'_i,F'_i]\big)$ is a subset of $\mathbb{P}\Omega(\underline{k}(p_i))$ (see Definition~\ref{def:domainofgeometrictype}), and $[V^i]$ denote the homogeneous coordinate of a sphere in $D\Big(\underline{k}(p_i); [X'_i,F'_i]\Big)$ under the spanning tree coordinate associated to $F'_i$ (see Remark~\ref{rmk:chartonregionofgeometrictype}).
\end{remark}

\begin{remark}[Local fibered coordinate chart]\label{rmk:localfiberedchart}
For any $X \in U_{d,\varepsilon}(B_{\lambda})$, including those lying in a cut locus, one can still mimic the chart construction to define a \textbf{local fibered coordinate chart} in a neighborhood of $X$, as follows.

Let $X_0$ be the projection $\phi_B(X)$, and consider a spanning tree coordinate chart of $X_0$. Let $W := (w_1, \ldots, w_{n-3})$ be the coordinates in this chart.

Let $D_1, \ldots, D_s$ be the $(d,\varepsilon)$-convex hulls in $X$, and let $X_i$ be the infinitesimal sphere associated with $D_i$. Let $T_i$ be a Delaunay triangulation of $D_i$, and let $F_i$ be a shortest spanning tree of $D_i$. Consider a sufficiently small neighborhood $U$ of $X$. We may assume that for each nearby surface in $U$, the $(d,\varepsilon)$-convex hulls $D_i$ have a geometric spanning tree $F_i$ such that $(X_i, F_i)$ has the same geometric type as that of $X$. Note, however, that $F_i$ may not necessarily be the shortest spanning tree for the nearby surface.

We develop $D_i$ into the plane following steps~(2) and~(3). Note that if the reference point $x(D_i)$ lies on the edge $e'_i$, we take a single copy on the radial boundary of the sector $Sec(x_{p_i})$ and develop $\partial D_i$ starting from a small neighborhood of this copy.

We then define $V^i$ similarly, parameterizing the edges of $F_i$, and set $V := (V^1, \ldots, V^s)$.

By shrinking the neighborhood $U$ if necessary, we may assume that the map
\begin{equation}\label{equ:localfiberedcoordinatechart}
    X\mapsto (V,W)
\end{equation}
defines a local coordinate chart on $U$.
\end{remark}

\subsection{Decomposition of Thin Parts via Coordinate Charts}\label{sec:goodcoordinate}
In this subsection, we aim to decompose a thin part into finitely many fibered coordinate charts, up to a measure-zero set under $\mu_{Thu}$.

In this subsection, all results are stated under the following assumptions. Let $\underline{k} \in (0,1)^n$ be a curvature vector with positive curvature gap, and let $B$ be a boundary stratum of $\overline{\mathbb{P}\Omega}(\underline{k})$ of codimension $d$. Let $\varepsilon$ and $\lambda$ be positive real numbers satisfying the condition~\eqref{equ:restriction}.

\begin{lemma}\label{lem:genericdelaunayinfinitecase}
Denote by $D_1, \ldots, D_s$ the $(d,\varepsilon)$-convex hulls on $X \in U_{d,\varepsilon}(B_\lambda)$. Then, for $\mu_{Thu}$-almost every $X$ in $U_{d,\varepsilon}(B_{\lambda})$, each convex hull $D_i$ is Delaunay-generic.
\end{lemma}

\begin{proof}
    For $X \in U_{d,\varepsilon}(B_\lambda)$, select a Delaunay triangulation $T_i$ of $D_i$ (which may not be unique). Then extend $T_1, \ldots, T_s$ to a geometric triangulation $T$ of $X$. Let $F$ be a spanning tree of $T$. By choosing a sufficiently small neighborhood $U$ of $X$, we may assume that there is a spanning tree coordinate chart on $U$ associated with $F$. 
    
    Let $(z_1, \ldots, z_{n-3})$ be the coordinates on $U$. The edges of $T_i$ can be expressed as complex linear combinations of $(1, z_1, \ldots, z_{n-3})$ under this chart. If $D_i$ in $X$ is not Delaunay-generic, then by the proof of Lemma~\ref{lem:genericdelaunay}, there exist algebraic relations among the vectors associated with the edges of $T_i$. It follows that such $X$ are contained in a measure zero subset under $\mu_{Thu}$.
\end{proof}

By a similar argument as above, one can establish the following statement regarding the uniqueness of shortest spanning trees, and we omit the proof.

\begin{lemma}\label{lem:uniqueshortestspanningforest}
    Denote by $D_1, \ldots, D_s$ the $(d,\varepsilon)$-convex hulls for $X \in U_{d,\varepsilon}(B_\lambda)$. Then, for $\mu_{Thu}$-almost every $X$ in $U_{d,\varepsilon}(B_{\lambda})$, the shortest spanning forest of $D_1, \ldots, D_s$ is unique.
\end{lemma}

Next, we show that the measure of a cut locus is always zero with respect to $\mu_{Thu}$.

\begin{lemma}\label{lem:seamisofmeasurezero}
 Any cut locus $\mathcal{S}$ in $U_{d,\varepsilon}(B_\lambda)$ has measure zero under $\mu_{Thu}$.
\end{lemma}

\begin{proof}
Let $\mathcal{S}$ be a cut locus associated with a Delaunay region $D_{S,\underline{s}(P)}^{Del}(\underline{k}(P);T)$ in $B$ and the selected edges $(e'_i)_{i=1}^s$ of $T$. By Definition~\eqref{equ:seam}, a flat cone sphere $X$ in $\mathcal{S}$ is obtained by inverse Thurston surgery of some $X_0\in D_{S,\underline{s}(P)}^{Del}(\underline{k}(P);T)$ with respect to some reference data:
\begin{equation}
    \mathcal{R} = (X_i, x(D_i),\gamma_i)_{i=1}^s
\end{equation}
where $X_i \in \mathbb{P}\Omega(\underline{k}(p_i))$, $D_i$ is the core of $X_i$, and $\gamma_i$ lies in the edge $e'_i$ for some $i$. Note that $D_1, \ldots, D_s$ are the $(d, \varepsilon)$-convex hulls of $X$.

For $X\in\mathcal{S}$, consider a local fibered coordinate chart $(V,W)$ (defined in Remark~\ref{rmk:localfiberedchart}) on a neighborhood $U$ of $X$, where $W$ parametrizes the Delaunay region $D_{S,\underline{s}(P)}^{Del}(\underline{k}(P);T)$. For each $i$, let $w(e'_i)$ be the vector associated with the selected edge $e'_i$ in the chart. Note that $w(e'_i)$ is a complex linear combination of $(1,w_1,\ldots,w_{n-3})$.

The reference segment $\gamma_i$ is contained in a replacement cone. Thus the reference segment is developed to a vector at the vertex of the sector $Sec(x_{p_i})$, denoted by $v(\gamma_i)$. Note that $v(\gamma_i)$ is a complex linear combination of $V^i$.

If $X$ lies in a cut locus, then the segment $\gamma_i$ is contained in the edge $e'_i$ for some $i$. Hence, in the local coordinate chart $U$, the cut locus $\mathcal{S}$ is contained in the algebraic equation
$$
\operatorname{Im} \frac{v(\gamma_i)}{w(e'_i)} = 0
$$
for some $i$. It follows that $\mathcal{S}$ is a measure-zero subset under $\mu_{Thu}$.
\end{proof}

For the boundary stratum $B$ of $\overline{\mathbb{P}\Omega}(\underline{k})$, assume that the Delaunay regions
$$D_{S,\underline{s}(P)}^{Del}(\underline{k}(P);T_a), \quad a=1,\ldots,m,$$
cover $B$ up to a measure-zero set (see Lemma~\ref{lem:genericdelaunay}).

Let $P$ be the $\underline{k}$-admissible partition associated with $B$. Let $p_1, \ldots, p_s$ be the non-singleton elements of $P$. Let 
\begin{equation}\label{equ:geometrictypes}
    [X_i^b, F_i^b],\quad b = 1,\ldots, q_i
\end{equation}
be the geometric types of spanning trees as in Lemma~\ref{lem:decompinfinitecase}.
Recall that for any Delaunay-generic convex infinite flat sphere $X_i \in \mathbb{P}\Omega(\underline{k}(p_i))$, if $X_i$ has a unique shortest spanning tree $F_i$, then the geometric type of $(X_i, F_i)$ is one of the types listed above.

\begin{proposition}\label{prop:goodcoordinate}
Let $\underline{k}$ be a curvature vector with positive curvature gap. Then the disjoint union
$$
\bigsqcup_{a=1}^m \bigsqcup_{\substack{1\le b_i \le q_i,\\ \text{ for } 1\le i\le s}}  
N\left(T_a, \left([X_i^{b_i}, F_i^{b_i}]\right)_{i=1}^s \right) \setminus \mathcal{S}_a
$$
has full measure in $U_{d,\varepsilon}(B_\lambda)$, where $\mathcal{S}_a$ denotes a cut locus associated with the Delaunay region $D_{S,\underline{s}(P)}^{Del}(\underline{k}(P); T_a)$ and the selected edges $(e_i^a)_{i=1}^s$.
\end{proposition}

\begin{proof}
For $X\in U_{d,\varepsilon}(B_{\lambda})$, consider a local fibered coordinate chart $(V,W)$, as defined in Remark~\ref{rmk:localfiberedchart}, on a neighborhood $U$ of $X$.

Note that under these local fibered charts, the projection $\phi_B$ is given by $(V,W) \mapsto W$. By the proof of Lemma~\ref{lem:genericdelaunay}, the complement of $\left(D_{S,\underline{s}(P)}^{Del}(\underline{k}(P);T_a)\right)_{a=1}^m$ is an analytic subset. The preimage of the complement under the projection $(V,W) \mapsto W$ remains analytic, so the preimage is of measure zero.

Therefore, up to a measure-zero set, we may assume that spheres in $U_{d,\varepsilon}(B_{\lambda})$ are mapped by $\phi_B$ into one of the Delaunay regions $\left(D_{S,\underline{s}(P)}^{Del}(\underline{k}(P);T_a)\right)_{a=1}^m$.

Let $D_1,\ldots,D_s$ be the $(d,\varepsilon)$-convex hulls in $X\in U_{d,\varepsilon}(B_\lambda)$. 
By Lemma~\ref{lem:genericdelaunayinfinitecase} and Lemma~\ref{lem:uniqueshortestspanningforest}, we conclude that by removing a subset of measure zero, the spheres $X$ in $U_{d,\varepsilon}(B_{\lambda})$ satisfy the following conditions:
\begin{itemize}
    \item Each $D_i$ is Delaunay-generic (in the sense of Definition~\ref{def:delaunayforconvexhull}).
    \item Each $D_i$ has a unique shortest spanning tree.
\end{itemize}
Combining this with Lemma~\ref{lem:seamisofmeasurezero}, we know that there exist geometric types $\Big([X_i^{b_i},F_i^{b_i}]\Big)_{i=1}^s$ for $1\le b_i \le q_i$ such that $X \in N\big(T_a, ([X_i^{b_i}, F_i^{b_i}])_{i=1}^s\big) \setminus \mathcal{S}_a.$
\end{proof}

\section{Asymptotic Volumes of Thin Parts of Moduli Spaces}\label{bigsec:measurebounds}

Recall from~\eqref{equ:ngbh} that the neighborhood $U_{d,\varepsilon} \subset \mathbb{P}\Omega(\underline{k})$ consists of flat spheres that admit a geometric forest with $d$ edges of normalized length less than $\varepsilon$.

The main goal of this section is to prove the following theorem:

\begin{theorem}\label{thm:main1}
Let $\underline{k}$ be a curvature vector with positive curvature gap. As $\varepsilon \to 0$, we have
\begin{equation}\label{equ:asymvolume}
    \mu_{Thu}(U_{d,\varepsilon}) =  \sum\limits_{\mathrm{codim}(B) = d} C(P)\cdot \mu_{Thu}(B)\cdot \varepsilon^{2d} + O(\varepsilon^{2d+2}),
\end{equation}
where the sum is taken over all boundary strata of codimension $d$. Here, $P$ is the $\underline{k}$-admissible partition associated with $B$, $C(P)$ is a positive constant depending only on $P$, and $\mu_{Thu}(B)$ denotes the total volume of the complex hyperbolic metric associated with $B$.
\end{theorem}

\noindent\textit{\textbf{Convention.}} In this section, unless otherwise stated, all results are considered under the following assumptions:
\begin{itemize}
    \item Let $\underline{k} \in (0,1)^n$ be a curvature vector with positive curvature gap.
    \item Let $B$ be a boundary stratum of codimension $d$ in $\overline{\mathbb{P}\Omega}(\underline{k})$.
    \item Let $\varepsilon$ and $\lambda$ be positive real numbers satisfying condition~\eqref{equ:restriction}.
    \item Let $P$ be the $\underline{k}$-admissible partition associated with $B$, and let $p_1,\ldots,p_s$ denote the non-singleton elements of $P$.
    \item Let $T$ be a triangulation of the base sphere $(S,\underline{s}(P))$ (see~\eqref{equ:boundarybasesphere}).
    \item Let $([X_i, F_i])_{i=1}^s$ be a sequence of geometric types of spanning trees on the spheres in $\mathbb{P}\Omega(\underline{k}(p_i))$.
    \item Let $\mathcal{S}$ be a cut locus in $U_{d,\varepsilon}(B_{\lambda})$.
\end{itemize}
All computations are carried out in the fibered coordinate chart
$$
N\big(T, ([X_i, F_i])_{i=1}^s\big) \setminus \mathcal{S}.
$$

\subsection{Area Function}
For $X \in U_{d,\varepsilon}(B_{\lambda})$, let $(D_i)_{i=1}^s$ be the $(d,\varepsilon)$-convex hulls, and let $(C_{D_i})_{i=1}^s$ be the associated replacement cones. Denote by $X_0$ the projection $\phi_B(X)$, and let $\mathrm{Area}(X)$ represent the area of $X$. By the definition of Thurston surgeries, we have
\begin{equation}\label{equ:areaformula}
    \mathrm{Area}(X) = \mathrm{Area}(X_0) - \sum_{i=1}^s \Big(\mathrm{Area}(C_{\partial D_i}) - \mathrm{Area}(D_i)\Big).
\end{equation}
For $X \in U_{d,\varepsilon}(B_{\lambda})$, we define
\begin{equation}\label{equ:defB}
    D(X) := \sum_{i=1}^s \Big(\mathrm{Area}(C_{\partial D_i}) - \mathrm{Area}(D_i)\Big).
\end{equation}

For the fibered coordinate chart $N\big(T_a, ([X_i, F_i])_{i=1}^s\big) \setminus \mathcal{S}$,
let $(w_0, V, W)$ be the unprojectivized fibered parameters associated with $X$ (see~\eqref{equ:unprojectivefiberedcoordinate}). Recall from Section~\ref{sec:complexhyper} that $\mathrm{Area}$ is a quadratic form about $(w_0, V, W)$. The equation~\eqref{equ:areaformula} can be rewritten in the fibered coordinate chart as
\begin{equation}\label{equ:defdv}
    \mathrm{Area}(w_0, V, W) = \mathrm{Area}(w_0, 0, W) - D_B(V).
\end{equation}

Denote by $H$ the associated Hermitian form of the quadratic form $\mathrm{Area}(w_0,V,W)$, and denote the hermitian form of $D_B(\cdot)$ by $D_B(\cdot,\cdot)$. The following lemma is obtained by a direct computation.

\begin{lemma}\label{lem:fibereareahermitian}
In the fibered coordinate chart $N\big(T_a, ([X_i, F_i])_{i=1}^s\big) \setminus \mathcal{S}$, the Hermitian form $H$ is given by
$$
H\big((w_0, V, W), (w_0', V', W')\big) = H\big((w_0, 0, W), (w_0', 0, W')\big) - D_B(V, V')
$$
for $(w_0, V, W), (w_0', V', W') \in \mathbb{C}^{n-3}$. In particular, $H$ satisfies the following properties:
    \begin{itemize}
        \item $H\big((0,V,0),(W_0',0,W')\big) = 0$.
        \item $H\big((W_0,V,W),(W_0',0,W')\big) = H\big((W_0,0,W),(W_0',0,W')\big)$.
        \item $H\big((W_0,V,W),(0,V',0)\big) = -D_B(V,V')$.
    \end{itemize}
\end{lemma}

\begin{lemma}\label{lem:bv}
In the fibered coordinate chart $N\big(T_a, ([X_i, F_i])_{i=1}^s\big) \setminus \mathcal{S}$, $D_B(V)$ is a positive definite quadratic form. In particular, we have that $\mathrm{Area}(X)\le \mathrm{Area}(X_0)$
\end{lemma}

\begin{proof}
    As introduced in Section~\ref{sec:complexhyper}, the function $\mathrm{Area}(w_0, V, W)$ is a quadratic form of signature $(1, n-3)$. Lemma~\ref{lem:thurstoncompletion} implies that $B$ is isometric to a moduli space of convex flat cone spheres. 

    Note that $\mathrm{Area}(w_0, 0, W)$ represents the area function expressed in the Delaunay region $D_{S,\underline{s}(P)}^{Del}(\underline{k}(P);T)$ in $B(P)$. It follows that the signature of $\mathrm{Area}(w_0, 0, W)$ is $(1, n-3-d)$. Furthermore, from the identity
    $$
    H\big((0, V, 0), (w_0', 0, W')\big) = 0,
    $$
    we conclude that $(w_0, 0, W)$ and $(0, V, 0)$ are orthogonal with respect to $H$. Therefore, $D_B(V)$ is positive definite.
\end{proof}

By Lemma~\ref{lem:metricformula}, we know that the Hermitian metric $h_{Thu}$ is given by
\begin{equation}\label{equ:metricfibered}
    h_{Thu} = \frac{|H((0,dV,dW),(1,V,W))|^2 - \mathrm{Area}(0,dV,dW)\mathrm{Area}(1,V,W)}
    {\mathrm{Area}(1,V,W)^2}
\end{equation}
in the fibered coordinate chart $N\big(T, ([X_i, F_i])_{i=1}^s\big) \setminus \mathcal{S}$. It follows that for any tangent vectors $(\alpha_1,\beta_1), (\alpha_2,\beta_2) \in \mathbb{C}^d \times \mathbb{C}^{n-3-d}$ at $(V,W)$, we have
\begin{equation}\label{equ:metricfiberedhermitian}
    \begin{aligned}
        h_{Thu}\big((\alpha_1,\beta_1),(\alpha_2,\beta_2)\big) 
        &= \frac{H\big((0,\alpha_1, \beta_1),(1,V,W)\big) H\big((1,V,W),(0,\alpha_2, \beta_2)\big)}
        {\mathrm{Area}(1,V,W)^2} \\
        &\quad - \frac{H\big((0,\alpha_1,\beta_1),(0,\alpha_2,\beta_2)\big)\mathrm{Area}(1,V,W)}
        {\mathrm{Area}(1,V,W)^2}.
    \end{aligned}
\end{equation}

\smallskip

\noindent\textit{\textbf{Notation}} To simplify later computations, we introduce the following notation:
$$
1_{v_i} := (0,\ldots,1,\ldots,0) \in \mathbb{C}^d,
$$
where $1$ is at the $i$-th entry. Similarly, we define
$$
1_{w_j} := (0,\ldots,1,\ldots,0) \in \mathbb{C}^{n-3-d},
$$
where $1$ is at the $j$-th entry. In the following computations, we may write the parameter $V$ in the fibered coordinate chart as
$$
V = (v_1, \ldots, v_d),
$$
where each $v_i \in \mathbb{C}$.\newline

Compared to~\eqref{equ:defofVparameter}, the notation for the parameter $V$ above is used when the computation does not depend on which spanning tree the entry of $V$ belongs to. Note that the tangent vector $\frac{\partial}{\partial v_i}$ (resp. $\frac{\partial}{\partial w_j}$) at a point $X$ in the fibered coordinate chart corresponds to $1_{v_i}$ (resp. $1_{w_j}$) at the coordinate $(V,W)$ of $X$.

By combining Lemma~\ref{lem:fibereareahermitian} with equation~\eqref{equ:metricfiberedhermitian}, we obtain the following results.

\begin{lemma}\label{lem:someformularsforthurstonmetrics}
In the fibered coordinate chart $N\big(T_a, ([X_i, F_i])_{i=1}^s\big) \setminus \mathcal{S}$,
we have that:
\begin{enumerate}
    \item 
    \begin{flalign}
        h_{Thu}\left(\frac{\partial}{\partial v_i}, \frac{\partial}{\partial v_j}\right) 
        &= \frac{D(1_{v_i}, V)D_B(V,1_{v_j}) + D(1_{v_i}, 1_{v_j})\operatorname{Area}(1,V,W)}
        {\mathrm{Area}(1,V,W)^2}. \notag
    \end{flalign}

    \item 
    \begin{flalign}
        h_{Thu}\left(\frac{\partial}{\partial v_i}, \frac{\partial}{\partial w_j}\right) 
        &= \frac{-D(1_{v_i}, V)H((1,0,W),(0,0,1_{w_j}))}
        {\mathrm{Area}(1,V,W)^2}. \notag
    \end{flalign}

    \item 
    \begin{flalign}
        h_{Thu}\left(\frac{\partial}{\partial w_i}, \frac{\partial}{\partial w_j}\right) 
        &= \frac{H\big((0,0, 1_{w_i}),(1,0,W)\big) H\big((1,0,W),(0,0, 1_{w_j})\big)}{\mathrm{Area}(1,V,W)^2} \notag \\
        &\quad - \frac{H\big((0,0,1_{w_i}),(0,0,1_{w_j})\big)\mathrm{Area}(1,V,W)}
        {\mathrm{Area}(1,V,W)^2}.\notag
    \end{flalign}
\end{enumerate}

\end{lemma}

\subsection{Uniform Approximation}

We introduce the notion of uniform approximation on thin parts.

\begin{definition}\label{def:usualbigOnotation}
    Let $\varepsilon, \lambda$ be positive numbers satisfying condition~\eqref{equ:restriction}. For each $\varepsilon$, let $f_\varepsilon(V,W)$ and $g_\varepsilon(V,W)$ be two functions defined on a fibered coordinate chart $N\big(T, ([X_i, F_i])_{i=1}^s\big) \setminus \mathcal{S}$. We say that the equation
    $$
    f_\varepsilon = O(g_\varepsilon)
    $$
    \textbf{holds uniformly} with respect to $(V,W)$ in the fibered coordinate chart if there exists a constant $L > 0$ such that the inequality
    $$
    f_\varepsilon(V,W) \leq L\cdot g_\varepsilon(V,W)
    $$
    holds for all $\varepsilon$ and all $(V,W)$ in the coordinate chart.
\end{definition}

We recall the following result from~\cite{fu2023boundssaddleconnectionsflat} for later use.

\begin{lemma}[{\cite[Lemma~4.3]{fu2023boundssaddleconnectionsflat}}]\label{lem:uniformboundonDelaunayedge}
In a flat sphere $X$ with $n$ conical singularities and a curvature gap $\delta > 0$, the normalized length $L$ of any edge in any Delaunay triangulation of $X$ satisfies 
$L < \sqrt{\frac{4}{\pi} + \frac{1}{2\pi\delta}}.$
\end{lemma}

To simplify later computations, we introduce the notation:
\begin{equation}\label{equ:uniformboundonDelaunayedge}
    c(\delta) := \sqrt{\frac{4}{\pi} + \frac{1}{2\pi\delta}}.
\end{equation}

\begin{lemma}\label{lem:basicinequalityforVW}
In the fibered coordinate chart $N\big(T_a, ([X_i, F_i])_{i=1}^s\big) \setminus \mathcal{S}$, the parameters $V$ and $W$ satisfy the following inequalities:
    \begin{enumerate}
        \item For each $1 \leq i \leq d$,
        $
        \frac{|v_i|}{\sqrt{\mathrm{Area}(1,0,W)}} \leq\frac{|v_i|}{\sqrt{\mathrm{Area}(1,V,W)}}\leq \varepsilon.
        $
        
        \item For each $1 \leq i \leq n-3-d$,
        $
       \frac{1}{\sqrt{\mathrm{Area}(1,0,W)}} \leq c(\delta)$ \text{and} $\frac{|w_i|}{\sqrt{\mathrm{Area}(1,0,W)}} \leq c(\delta).
        $
    \end{enumerate}
\end{lemma}

\begin{proof}
    The first assertion follows directly from equation~\eqref{equ:fiberexpression} in Theorem~\ref{thm:fiberedcoordinatechart} and Lemma~\ref{lem:bv}.

    For the second assertion, since $(1,W)$ parameterizes the edges of a spanning tree of $T$ in the Delaunay region $D_{S,\underline{s}(P)}^{Del}(\underline{k}(P);T)$ in the boundary stratum $B(P)$, Lemma~\ref{lem:uniformboundonDelaunayedge} implies that 
    $
    \frac{1}{\sqrt{\mathrm{Area}(1,V,W)}} \leq c(\delta)$ and $ \frac{w_i}{\sqrt{\mathrm{Area}(1,0,W)}} \leq c(\delta).
    $
\end{proof}

\begin{lemma}\label{lem:someidentitiesofuniformapproximation}
In the fibered coordinate chart $N\big(T_a, ([X_i, F_i])_{i=1}^s\big) \setminus \mathcal{S}$, the following asymptotic identities hold uniformly with respect to $(V,W)$ in the fibered coordinate chart:
\begin{enumerate}
    \item $\frac{D_B(V)}{\operatorname{Area}(1,V,W)}=O(\varepsilon^2)$.
    \item $\frac{D(1_{v_i}, V)}{\mathrm{Area}(1,V,W)} = O(\varepsilon)$ and $\frac{D(1_{v_i},1_{v_j})}{\operatorname{Area}(1,0,W)}=O(1)$.
    \item $\frac{H((1,0,W),(0,0,1_{w_j}))}{\mathrm{Area}(1,V,W)} = O(1)$ and $\frac{H((0,0,1_{w_i}),(0,0,1_{w_j}))}{\mathrm{Area}(1,V,W)} = O(1)$.
\end{enumerate}
\end{lemma}

\begin{proof}
Define
\begin{align*}
c_1 &:= \max_{1\le i,j\le d} |D(1_{v_i},1_{v_j})|\\    
c_2 &:= \max_{1\le i,j\le n-3-d} \left\{|H\big((1,0,0),(0,0,1_{w_j})\big)|, |H\big((0,0,1_{w_i}),(0,0,1_{w_j})\big)|\right\}.
\end{align*}

For (1), by Lemma~\ref{lem:basicinequalityforVW} (1), we have
\begin{align*}
    \frac{D_B(V)}{\operatorname{Area}(1,V,W)}&=D_B\left(\frac{V}{\sqrt{\operatorname{Area}(1,V,W)}}\right)\\
    &\leq \sum_{1\leq i, j\leq d} c_1 \cdot \left|\frac{v_i}{\sqrt{\operatorname{Area}(1,V,W)}}\right| \left|\frac{v_j}{\sqrt{\operatorname{Area}(1,V,W)}}\right|\\
    &\leq d^2 c_1 \varepsilon^2.
\end{align*}

For (2), using Lemma~\ref{lem:basicinequalityforVW} (1) and (2), we compute
$$
\left|\frac{D(1_{v_i}, V)}{\mathrm{Area}(1,V,W)}\right| = 
\left|D_B\left(\frac{1_{v_i}}{\sqrt{\operatorname{Area}(1,V,W)}},\frac{V}{\sqrt{\operatorname{Area}(1,V,W)}}\right)\right|
\leq c \cdot d \cdot c(\delta) \cdot \varepsilon.
$$
Similarly, we obtain
$$
\left|\frac{D(1_{v_i},1_{v_j})}{\operatorname{Area}(1,0,W)}\right|\leq c \cdot c(\delta)^2.
$$

For (3), first note that
\begin{align*}
    \frac{|w_i|}{\sqrt{\operatorname{Area}(1,V,W)}} &= \frac{|w_i|}{\sqrt{\operatorname{Area}(1,0,W) - \operatorname{Area}(1,0,W)}} \\
    &= \frac{|w_i|}{\sqrt{\operatorname{Area}(1,0,W) - D_B(V)}}\\
    &= \frac{|w_i|/\sqrt{\operatorname{Area}(1,0,W)}}{\sqrt{1 - O(\varepsilon^2)}}\\
    &= \frac{O(1)}{1-O(\varepsilon^2)} = O(1)
\end{align*}
where the third equality uses (1) of Lemma~\ref{lem:someidentitiesofuniformapproximation} and the fourth equality uses (2) of Lemma~\ref{lem:basicinequalityforVW}.
Similarly, one can show that
$$
    \frac{|w_i|}{\sqrt{\operatorname{Area}(1,V,W)}}=O(1).
$$
It follows that
\begin{align*}
    &\left|\frac{H((1,0,W),(0,0,1_{w_j}))}{\mathrm{Area}(1,V,W)}\right| \\
    &= \left|H\left( \left(\frac{1}{\sqrt{\operatorname{Area}(1,V,W)}},0,\frac{W}{\sqrt{\operatorname{Area}(1,V,W)}}\right),\left(0,0,\frac{W}{\sqrt{\operatorname{Area}(1,V,W)}}\right)\right)\right|\\
    &\leq \sum_{1\leq j\leq n-3-d} c_2 \left|\frac{1}{\sqrt{\operatorname{Area}(1,V,W)}}\right| \left|\frac{w_j}{\sqrt{\operatorname{Area}(1,V,W)}}\right| \\
    &\quad + \sum_{1\leq i,j\leq n-3-d} c_2 \left|\frac{w_i}{\sqrt{\operatorname{Area}(1,V,W)}}\right| \left|\frac{w_j}{\sqrt{\operatorname{Area}(1,V,W)}}\right|\\
    &\leq (n-3-d)\cdot c_2\cdot c(\delta)^2 + (n-3-d)^2\cdot c_2\cdot c(\delta)^2.
\end{align*}
Thus, we conclude that 
$$
\frac{H((1,0,W),(0,0,1_{w_j}))}{\mathrm{Area}(1,V,W)} = O(1).
$$
Similarly, we obtain 
$$
\frac{H((0,0,1_{w_i}),(0,0,1_{w_j}))}{\mathrm{Area}(1,V,W)} = O(1).
$$
\end{proof}

\subsection{Approximation Using Product Metric}
We construct a product metric on the fibered coordinate chart as follows.

In the fibered coordinate chart $N\big(T, ([X_i, F_i])_{i=1}^s\big) \setminus \mathcal{S}$, we define the following symmetric $2$-tensors:
\begin{equation}\label{equ:hVmetric}
    h_V = \frac{D(dV)}{\operatorname{Area}(1,0,W)}
\end{equation}
and
\begin{equation}\label{equ:hBmetric}
    h_B =\frac{|H((0,0,dW),(1,0,W))|^2-\mathrm{Area}(0,0,dW)\mathrm{Area}(1,0,W)}{\mathrm{Area}(1,0,W)^2}.
\end{equation}
Note that $h_V$ is a metric tensor when restricted to the leaves where $W$ is fixed (by Lemma~\ref{lem:bv}), while $h_B$ is a metric tensor when restricted to the leaves where $V$ is fixed.

\begin{proposition}\label{prop:metricapproximation}
In the fibered coordinate chart $N\big(T, ([X_i, F_i])_{i=1}^s\big) \setminus \mathcal{S}$, the Thurston metric satisfies
    $$
    h_{Thu} = h_{V} \oplus h_B + b,
    $$
    where $b$ is a symmetric $2$-tensor on the fibered coordinate chart satisfying the uniform estimates:
    $$
    b\left(\frac{\partial}{\partial v_i}, \frac{\partial}{\partial v_j}\right) = O(\varepsilon^2),
    \quad
    b\left(\frac{\partial}{\partial v_i}, \frac{\partial}{\partial w_j}\right) = O(\varepsilon),
    \quad
    b\left(\frac{\partial}{\partial w_i}, \frac{\partial}{\partial w_j}\right) = O(\varepsilon^2).
    $$
\end{proposition}

\begin{proof}
According to (1) and (2) in Lemma~\ref{lem:someidentitiesofuniformapproximation}, we compute:
\begin{align*}
    b\left(\frac{\partial}{\partial v_i}, \frac{\partial}{\partial v_j}\right)
    &= h_{Thu}\left(\frac{\partial}{\partial v_i}, \frac{\partial}{\partial v_j}\right) - h_V\left(\frac{\partial}{\partial v_i}, \frac{\partial}{\partial v_j}\right)\\
    &= \frac{D(1_{v_i}, V)D_B(V,1_{v_j})}{\mathrm{Area}(1,V,W)^2} + 
       D(1_{v_i},1_{v_j})\left[\frac{1}{\operatorname{Area}(1,V,W)} - \frac{1}{\operatorname{Area}(1,0,W)}\right]\\
    &= \frac{D(1_{v_i}, V)D_B(V,1_{v_j})}{\mathrm{Area}(1,V,W)^2} + 
       \frac{D(1_{v_i},1_{v_j})}{\operatorname{Area}(1,0,W)}\frac{D_B(V)}{\operatorname{Area}(1,V,W)}\\
    &= O(\varepsilon^2).
\end{align*}

By (2) and (3) in Lemma~\ref{lem:someidentitiesofuniformapproximation}, we compute:
\begin{align*}
    b\left(\frac{\partial}{\partial v_i}, \frac{\partial}{\partial w_j}\right)
    &= h_{Thu}\left(\frac{\partial}{\partial v_i}, \frac{\partial}{\partial w_j}\right) \\
    &= \frac{-D(1_{v_i}, V)H((1,0,W),(0,0,1_{w_j}))}{\mathrm{Area}(1,V,W)^2} \\
    &= O(\varepsilon).
\end{align*}

Finally, applying (3) of Lemma~\ref{lem:someformularsforthurstonmetrics}, we compute:
\begin{align*}
    &b\left(\frac{\partial}{\partial w_i}, \frac{\partial}{\partial w_j}\right)\\ 
    &= h_{Thu}\left(\frac{\partial}{\partial w_i}, \frac{\partial}{\partial w_j}\right) - h_B\left(\frac{\partial}{\partial w_i}, \frac{\partial}{\partial w_j}\right)\\
    &= \frac{H\big((0,0, 1_{w_i}),(1,0,W)\big) H\big((1,0,W),(0,0, 1_{w_j})\big)}{\mathrm{Area}(1,V,W)^2} \frac{\mathrm{Area}(1,0,W) + \mathrm{Area}(1,V,W)}{\mathrm{Area}(1,0,W)}\frac{D_B(V)}{\mathrm{Area}(1,0,W)} \\ 
    &\quad - \frac{H\big((0,0,1_{w_i}),(0,0,1_{w_j})\big)}{\mathrm{Area}(1,V,W)}\frac{D_B(V)}{\mathrm{Area}(1,0,W)}.
\end{align*}
Then, applying (1) and (3) in Lemma~\ref{lem:someidentitiesofuniformapproximation}, we conclude that
$$
b\left(\frac{\partial}{\partial w_i}, \frac{\partial}{\partial w_j}\right) = O(\varepsilon)
$$
holds uniformly.
\end{proof}

We use the above metric approximation to deduce a uniform approximation of the associated volume form. Recall that $d\mu_{Thu}$ denotes the volume form of the metric $h_{Thu}$. 
\begin{itemize}
    \item Denote by $\mu_V$ the measure induced by $h_V$ on any leaf where $W$ is fixed, and denote by $\mu_B$ the measure induced by $h_B$ on any leaf where $V$ is fixed.
    \item Denote by $d\mu_V$ the $2d$-form associated with the tensor $h_V$, and by $d\mu_B$ the $2(n-3-d)$-form associated with the tensor $h_B$.
\end{itemize}
Note that $d\mu_B$ is also the volume form of the complex hyperbolic metric $h_{Thu}$ on the moduli space corresponding to the boundary stratum $B$.\newline

Set
$$
dV d\overline{V} := dv_1 d\overline{v}_1 \cdots dv_d d\overline{v}_d, 
\quad
dW d\overline{W} := dw_1 d\overline{w}_1 \cdots dw_{n-3-d} d\overline{w}_{n-3-d}.
$$

\begin{corollary}\label{cor:volumeformapproximation}
      In the fibered coordinate chart $N\big(T, ([X_i, F_i])_{i=1}^s\big) \setminus \mathcal{S}$, we have
     \begin{equation}\label{equ:volumeformsasymptotics}
         d\mu_{Thu} = d\mu_V \wedge d\mu_B + f dV d\overline{V} dW d\overline{W},
     \end{equation}
     where $f$ is a smooth function of $(V,W)$ satisfying
     \begin{equation}
         f = O(\varepsilon^2),
     \end{equation}
     which holds uniformly with respect to $(V,W)$ in the fibered coordinate chart.
\end{corollary}

\begin{proof}
In the fibered coordinate chart, the volume form $d\mu_{Thu}$ is expressed as:
$$
d\mu_{Thu} = \left(\frac{i}{2}\right)^{n-3} h(V,W) dV d\overline{V} dW d\overline{W},
$$
where 
$$
 h(V,W) = \operatorname{det}\left(
    \begin{array}{ll}
       \left(h_{Thu}\left(\frac{\partial}{\partial v_i},\frac{\partial}{\partial v_j}\right)\right)_{1\leq i,j\leq d}  & \left(h_{Thu}\left(\frac{\partial}{\partial v_i},\frac{\partial}{\partial w_j}\right)\right)_{1\leq i \leq d,1\leq j\leq n-3-d}  \\
       \left(h_{Thu}\left(\frac{\partial}{\partial w_i},\frac{\partial}{\partial v_j}\right)\right)_{1\leq i \leq n-3-d,1\leq j\leq d}  & \left(h_{Thu}\left(\frac{\partial}{\partial w_i},\frac{\partial}{\partial w_j}\right)\right)_{1\leq i,j\leq n-3-d}
    \end{array}
    \right).
$$
Applying Proposition~\ref{prop:metricapproximation}, we compute:
\begin{align*}
    h(V,W) 
    &= \operatorname{det}\left(
    \begin{array}{ll}
       \left(h_V\left(\frac{\partial}{\partial v_i},\frac{\partial}{\partial v_j}\right) + O(\varepsilon^2)\right)_{1\leq i,j\leq d}  & \left(O(\varepsilon)\right)_{1\leq i \leq d,1\leq j\leq n-3-d}  \\
       \left(O(\varepsilon)\right)_{1\leq i \leq n-3-d,1\leq j\leq d}  & \left(h_{B}\left(\frac{\partial}{\partial w_i},\frac{\partial}{\partial w_j}\right) + O(\varepsilon^2)\right)_{1\leq i,j\leq n-3-d}
    \end{array}
    \right)\\
    &= \operatorname{det}\left(\left(h_V\left(\frac{\partial}{\partial v_i},\frac{\partial}{\partial v_j}\right)\right)_{1\leq i,j\leq d}\right) 
    \operatorname{det}\left(\left(h_B\left(\frac{\partial}{\partial w_i},\frac{\partial}{\partial w_j}\right)\right)_{1\leq i,j\leq n-3-d}\right) +O(\varepsilon^2).
\end{align*}
Let $f$ be the difference between $h(V,W)$ and the first term in the last equation. Then $f$ is a smooth function of $(V,W)$ satisfying $f = O(\varepsilon^2)$. 

Furthermore, using the expressions of $d\mu_V$ and $d\mu_B$ in the fibered coordinate chart, we conclude that
$$
d\mu_{Thu} = d\mu_V\wedge d\mu_B + f dV d\overline{V} dW d\overline{W}.
$$
\end{proof}

For later use in estimating the error term in the asymptotic formula~\eqref{equ:asymvolume}, we need to control the integral of the term
$$
fdV d\overline{V} dW d\overline{W}
$$
on the fibered coordinate chart. However, the Delaunay region $D_{S,\underline{s}(P)}^{Del}(\underline{k}(P);T)$ in the boundary stratum $B(P)$ may map to an unbounded subset in the spanning tree coordinate chart. In this case, the integral of $dW d\overline{W}$ over $D_{S,\underline{s}(P)}^{Del}(\underline{k}(P);T)$ is infinite. To address this, we establish the following lemma:

\begin{lemma}\label{lem:compactdelaunay}
Let $\underline{k}\in(0,1)^n$ be a curvature vector. Let $T$ be a triangulation of $(S,\underline{s})$, and let $F$ be a spanning tree contained in $T$. Then, the Delaunay region $D_{S,\underline{s}}^{Del}(\underline{k};T)$ can be decomposed into finitely many disjoint subsets
\begin{equation}\label{equ:compactdelaunay}
    D_{S,\underline{s}}^{Del}(\underline{k};T) = A_1 \sqcup \ldots \sqcup A_l
\end{equation}
such that each $A_i$ corresponds to a bounded subset of $\mathbb{C}^{n-3}$ under a spanning tree coordinate chart associated to $F$ where a certain edge of $F$ is normalized to $1$.
\end{lemma}

\begin{proof}
    Recall that the spanning tree coordinate chart associated with $F$ embeds $D_{S,\underline{s}}^{Del}(\underline{k};T)$ into $\mathbb{CP}^{n-3}$, with coordinates of the form
    $$
    [z_0, \ldots, z_{n-3}]
    $$
    where $z_0, \ldots, z_{n-3}$ parameterize the $(n-2)$ edges of $F$.

    We claim that there exists a finite partition of $\mathbb{CP}^{n-3}$,
    $$
    \mathbb{CP}^{n-3} = B_1 \sqcup \ldots \sqcup B_l,
    $$
    such that each $B_i$ corresponds to a bounded subset of $\mathbb{C}^{n-3}$ under a standard affine chart in $\mathbb{CP}^{n-3}$. This follows from the compactness of $\mathbb{CP}^{n-3}$. Indeed, every point in projective space is contained in some affine chart. We select a bounded neighborhood in the corresponding coordinates. Considering all points in $\mathbb{CP}^{n-3}$, we obtain an open cover. By the compactness, we find a finite cover:
    $$
    \mathbb{CP}^{n-3} = U_1 \cup \ldots \cup U_l,
    $$
    where each $U_i$ corresponds to a bounded subset under an affine chart. We then define
    $$
    B_1 = U_1, \quad B_i = U_i \setminus \bigcup_{j < i} U_j.
    $$
    This proves the claim.

    To complete the proof of the lemma, we define $A_i$ as the subset corresponding to $D_{S,\underline{s}}^{Del}(\underline{k};T) \cap B_i$ in $\mathbb{CP}^{n-3}$. Note that a standard affine chart in $\mathbb{CP}^{n-3}$ determines the edge of $F$ to be normalized. Therefore, each $A_i$ corresponds to a bounded subset under the spanning tree coordinate chart, with a certain edge of $F$ normalized to $1$.
\end{proof}

\subsection{Asymptotics Volumes of Fibers}

Recall that we define a $2$-tensor
$$
h_V = \frac{D_B(dV)}{\operatorname{Area}(1,0,W)}
$$
on a fibered coordinate chart; see~\eqref{equ:hVmetric}. Since $h_V$ defines a metric tensor when restricted to the leaves of the chart where $W$ is fixed, the associated $2d$-form $d\mu_V$ serves as a volume form on the fiber $\phi_B^{-1}(W)$ in $N\big(T, ([X_i, F_i])_{i=1}^s\big) \setminus \mathcal{S}$, where $\phi_B$ is the projection~\eqref{equ:projection}. We define the \textbf{volume of the fiber at $W$} as
$$
\mu_V\left(\phi_B^{-1}(W)\right) := \int_{\phi_B^{-1}(W)\cap N\big(T, ([X_i, F_i])_{i=1}^s\big) \setminus \mathcal{S}} d\mu_V.
$$

According to Theorem~\ref{thm:fiberedcoordinatechart}, the expression of the fiber $\phi_B^{-1}(W)$ in the fibered coordinate chart is given by
\begin{equation}\label{equ:volumefiberinfiberedcoordinatechart}
    A(V,W,\varepsilon) := \left\{V \in C_{([X_i, F_i])_{i=1}^s}(W) \mid \frac{|v_{i}|}{\sqrt{\operatorname{Area}(1, V, W)}} < \varepsilon \text{ for each } 1\le i\le d\right\}
\end{equation}
where $C_{([X_i, F_i])_{i=1}^s}(W)$ is an infinite cone in $\mathbb{C}^d$. Hence, the volume of the fiber is given by $\mu_V\left(A(V,W,\varepsilon)\right)$.

\begin{lemma}\label{lem:independentonWofthevolumeoffibers}
The volume of the fiber over $W$ is independent of the parameter $W \in D_{S,\underline{s}(P)}^{Del}(\underline{k}(P);T)$ in $B$, as well as of the choice of spanning tree coordinate charts on $D_{S,\underline{s}(P)}^{Del}(\underline{k}(P);T)$.    
\end{lemma}

\begin{proof}
By Theorem~\ref{thm:fiberedcoordinatechart}, the infinite cone $C_{([X_i, F_i])_{i=1}^s}(W)$ is unique up to complex scaling on each entry $V^i$ of $V$ for different $W$ and choices of spanning tree coordinate charts. Moreover, since the metric $h_V$ is defined as $h_V = \frac{D_B(dV)}{\operatorname{Area}(1,0,W)}$,
it follows that $h_V$ is invariant under Euclidean isometries, and so is the volume form $d\mu_V$. Consequently, the volume of the fiber remains independent of $W$.
\end{proof}

Since $D(\cdot)$ is a positive definite quadratic form, the metric $D_B(dV)$ induces the Lebesgue measure $\mu_{D_B}$ on $\mathbb{C}^d$. Denote the associated volume form on $\mathbb{C}^d$ by $d\mu_{D_B}$.

Consider the substitution $G:\mathbb{C}^d\to\mathbb{C}^d$ defined by
\begin{equation}\label{equ:substitution}
    G:V \to \frac{V}{\sqrt{\operatorname{Area}(1,0,W)}}.
\end{equation}
It follows that
$$
(G^{-1})^{\ast}\left(\frac{D_B(dV)}{\operatorname{Area}(1, 0, W)}\right) = D_B(dV) \quad \text{and} \quad (G^{-1})^{\ast}(d\mu_V) = d\mu_{D_B}.
$$

\begin{lemma}\label{lem:volumeoffiber}
The volume of the fiber at $W$ satisfies
\begin{equation}\label{equ:asymvolumeoffibers}
    \mu_V\left(\phi_B^{-1}(W)\right) = \varepsilon^{2d} \mu_{D_B}\left(C^{< 1}_{([X^b_i, F^b_i])_{i=1}^s}(W)\right) + O(\varepsilon^{4d})
\end{equation}
as $\varepsilon\to 0$, where $C^{< 1}_{([X_i, F_i])_{i=1}^s}(W)$ is the bounded cone defined by
$$
C^{< 1}_{([X_i, F_i])_{i=1}^s}(W) := \left\{V\in C_{([X_i, F_i])_{i=1}^s}(W) \mid |v_i| < 1  \right\}.
$$
\end{lemma}

\begin{proof}
By~\eqref{equ:volumefiberinfiberedcoordinatechart}, the volume of the fiber at $W$ is given by $\int_{A(V,W,\varepsilon)} d\mu_V$ in the fibered coordinate chart.

Define the subset
$$
A_0(V,W,\varepsilon):=\left\{V \in C_{([X_i, F_i])_{i=1}^s}(W) \mid \frac{|v_{i}|}{\sqrt{\operatorname{Area}(1, 0, W)}}<\varepsilon\text{ for each }i=1,\dots,d\right\}.
$$
We compare the measures of $A(V,W,\varepsilon)$ and $A_0(V,W,\varepsilon)$.

By Lemma~\ref{lem:bv}, we know that $\operatorname{Area}(1,V,W) \leq \operatorname{Area}(1,0,W)$, which implies that $A(V,W,\varepsilon) \subseteq A_0(V,W,\varepsilon)$.
It follows that
\begin{align*}
    0\leq \mu_V\big(A_0(V,W,\varepsilon)\big) - \mu_V\big(A(V,W,\varepsilon)\big) = \mu_V\big(A_0(V,W,\varepsilon)\setminus A(V,W,\varepsilon)\big).
\end{align*}
The subset $A_0(V,W,\varepsilon)\setminus A(V,W,\varepsilon)$ can be rewritten as
\begin{align*}
    & A_0(V,W,\varepsilon)\setminus A(V,W,\varepsilon) \\
    = & \left\{V \in C_{([X_i, F_i])_{i=1}^s}(W) \mid \varepsilon^2\operatorname{Area}(1, V, W) \leq |v_{i}|^2 < \varepsilon^2\operatorname{Area}(1, 0, W) \text{ for each } i=1,\dots,d \right\} \\
    = & \left\{V \in C_{([X_i, F_i])_{i=1}^s}(W) \mid \varepsilon^2\left(1 - \frac{D_B(V)}{\operatorname{Area}(1, 0, W)} \right) \leq \frac{|v_{i}|^2}{\operatorname{Area}(1, 0, W)} < \varepsilon^2 \text{ for each } i=1,\dots,d \right\} \\
    \subseteq & \left\{V \in C_{([X_i, F_i])_{i=1}^s}(W) \mid \varepsilon^2\left(1 - L\varepsilon^2 \right) \leq \frac{|v_{i}|^2}{\operatorname{Area}(1, 0, W)} < \varepsilon^2 \text{ for each } i=1,\dots,d \right\}
\end{align*}
for some constant $L>0$ independent of $(V,W)$. Here, the second equality follows from~\eqref{equ:defdv}, and the last inclusion uses Lemma~\ref{lem:someidentitiesofuniformapproximation}.

Applying the substitution~\eqref{equ:substitution}, we obtain
$$
\int_{A_0(V,W,\varepsilon)\setminus A(V,W,\varepsilon)} d\mu_V \leq \int_{\varepsilon^2\left(1 - L\varepsilon^2 \right) \leq |v_{i}|^2 < \varepsilon^2, 1\leq i\leq d} d\mu_{D_B} = O(\varepsilon^{4d}).
$$
Thus, we have that
$\mu_V(A(V,W,\varepsilon)) = \mu_V(A_0(V,W,\varepsilon)) + O(\varepsilon^{4d}).$

Next, applying the substitution~\eqref{equ:substitution} to $\int_{A_0(V,W,\varepsilon)} d\mu_V$, we obtain
\begin{align*}
    \int_{A_0(V,W,\varepsilon)} d\mu_V &= \mu_{D_B}\left(\left\{V\in C_{([X_i, F_i])_{i=1}^s}(W)\mid |v_i|<\varepsilon, \; i=1,\dots,d\right\}\right) \\
    &= \varepsilon^{2d} \mu_{D_B}\left(C^{< 1}_{([X_i, F_i])_{i=1}^s}(W)\right).
\end{align*}
This concludes the proof.
\end{proof}

\begin{remark}\label{rmk:independenceofvolumeoffibers}
Note that the leading coefficient
$$
\mu_{D_B}\left(C^{< 1}_{([X_i, F_i])_{i=1}^s}(W)\right)
$$
in the formula~\eqref{equ:asymvolumeoffibers} does not depend on $W$ or the choices of spanning tree charts. This follows from Theorem~\ref{thm:fiberedcoordinatechart}, which states that the infinite cone $C_{([X_i, F_i])_{i=1}^s}(W)$ is invariant under complex scaling on $V\in\mathbb{C}^d$, and from the fact that $\mu_{D_B}$ is invariant under Euclidean isometries. We may therefore simplify the notation as
$$
\mu_{D_B}\left(C^{< 1}_{([X_i, F_i])_{i=1}^s}\right).
$$    
\end{remark}

Let $P$ be the $\underline{k}$-admissible partition associated with $B$, and let $p_1,\ldots,p_s$ be the non-singleton elements in $P$. For each $i$, let $[X_i^{b(i)}, F_i^{b(i)}]$ for $1\le b(i)\le q_i$ be the geometric types of spanning trees as defined in Proposition~\ref{prop:goodcoordinate}, see also Lemma~\ref{lem:decompinfinitecase}. By definition, these geometric types are uniquely determined by the moduli space $\mathbb{P}\Omega(\underline{k}(p_i))$ and thus depend only on $p_i$.

By summing over all geometric types $[X_i^{b(i)}, F_i^{b(i)}]$ for $1\leq b \leq q_i$ and $1\leq i\leq d$, we define the following constant, which depends only on the partition $P$:
\begin{equation}\label{equ:svconst}
    C(P) := \sum_{j=1}^{s}\sum_{b(j)=1}^{q_i} \mu_{D_B}\left(C^{< 1}_{([X^{b(i)}_i, F^{b(i)}_i])_{i=1}^s}\right).
\end{equation}

\subsection{Proof of Theorem~\ref{thm:main1}}

We first compute the volume of the subset $U_{d,\varepsilon}(B_\lambda)$.

\begin{lemma}\label{lem:thickvolume}
    As $\varepsilon \to 0$, we have that
    $$
    \mu_{Thu}\big(U_{d,\varepsilon}(B_\lambda)\big) = \varepsilon^{2d}C(P)\mu_{Thu}(B_{\lambda}) + O(\varepsilon^{2d+2}),
    $$ 
    where $C(P)$ is defined as in~\eqref{equ:svconst}.
\end{lemma}

\begin{proof}
    By Proposition~\ref{prop:goodcoordinate}, we have
    $$
    \mu_{Thu}\big(U_{d,\varepsilon}(B_\lambda)\big) = \sum_{a=1}^m\sum_{j=1}^s\sum_{b(j)=1}^{q_j} 
    \mu_{Thu}\big(N\big(T_a, ([X_i^{b(i)}, F_i^{b(i)}])_{i=1}^s\big) \setminus \mathcal{S}_a\big),
    $$
    where the Delaunay regions $\big(D_{S,\underline{s}(P)}^{Del}(\underline{k})P);T_a)\big)_{a=1}^m$ in $B$ and the sequences of geometric types $([X_i^{b(i)},F_i^{b(i)}])_{i=1}^s$ are as defined in Proposition~\ref{prop:goodcoordinate}.
    
    By Corollary~\ref{cor:volumeformapproximation}, we obtain
    \begin{align*}
         &\mu_{Thu}\big(N\big(T_a, ([X_i^{b(i)}, F_i^{b(i)}])_{i=1}^s\big) \setminus \mathcal{S}_a\big)\\ 
         =& \int_{N\big(T_a, ([X_i^{b(i)}, F_i^{b(i)}])_{i=1}^s\big) \setminus \mathcal{S}_a}d\mu_{V}d\mu_B
        +
        \int_{N\big(T_a, ([X_i^{b(i)}, F_i^{b(i)}])_{i=1}^s\big) \setminus \mathcal{S}_a} f \, dV \, d\overline{V} \, dW \, d\overline{W},
    \end{align*}
    where $f$ is as in Corollary~\ref{cor:volumeformapproximation}.
    
    The subset $N\big(T_a, ([X_i^{b(i)}, F_i^{b(i)}])_{i=1}^s\big) \setminus \mathcal{S}_a$ is fibered over the Delaunay region $D_{S,\underline{s}(P)}^{Del}(\underline{k}(P);T_a)$ via the projection $\phi_B$. By Lemma~\ref{lem:volumeoffiber} and Remark~\ref{rmk:independenceofvolumeoffibers}, the volume of the fiber is independent of the point $W$ in $D_{S,\underline{s}(P)}^{Del}(\underline{k}(P);T_a)$. Thus, by Fubini's theorem,
    \begin{align*}
        \int_{N\big(T_a, ([X_i^{b(i)}, F_i^{b(i)}])_{i=1}^s\big) \setminus \mathcal{S}_a}d\mu_{V}d\mu_B = \mu_V\left(\phi_B^{-1}(W)\right) \mu_{B}\big(D_{S,\underline{s}(P)}^{Del}(\underline{k}(P);T_a)\big).
    \end{align*}
    By Lemma~\ref{equ:compactdelaunay}, we may assume that $D_{S,\underline{s}(P)}^{Del}(\underline{k}(P);T_a)$ is a bounded subset of $\mathbb{C}^{n-3}$ in the spanning tree coordinate chart. Furthermore, by Lemma~\ref{lem:basicinequalityforVW}, the parameter $V$ satisfies
    $$
    \frac{|v_i|}{\sqrt{\operatorname{Area}(1,0,W)}} \leq \varepsilon.
    $$
    Since $W$ is bounded, $V$ is also bounded. Applying Fubini’s theorem, we obtain
    \begin{align*}
        &\left|\int_{N\big(T_a, ([X_i^{b(i)}, F_i^{b(i)}])_{i=1}^s\big) \setminus \mathcal{S}_a} f \, dV \, d\overline{V} \, dW \, d\overline{W} \right|\\ 
        &\le L\varepsilon^2\int_{D_{S,\underline{s}(P)}^{Del}(\underline{k}(P);T_a)}\mu_{Leg}
        \left(\Big\{V\in \mathbb{C}^d\mid\frac{|v_i|}{\sqrt{\operatorname{Area}(1,0,W)}} \leq \varepsilon\Big\}\right)
        dW d\overline{W}\\
        &= L\varepsilon^{2+2d}\int_{D_{S,\underline{s}(P)}^{Del}(\underline{k}(P);T_a)}\mu_{Leg}
        \left(\Big\{V\in \mathbb{C}^d\mid\frac{|v_i|}{\sqrt{\operatorname{Area}(1,0,W)}} \leq 1\Big\}\right)
        dW d\overline{W}\\
        &= O(\varepsilon^{2d + 2}),
    \end{align*}
    where $L$ is a positive constant arising from the estimate
    $f = O(\varepsilon^2)$.

    Summing over all Delaunay regions in $B$ and all geometric types, we obtain
    $$
    \mu_{Thu}\big(U_{d,\varepsilon}(B_\lambda)\big) =  \varepsilon^{2d}C(P)\mu_{Thu}(B_{\lambda}) + O(\varepsilon^{2d+2}).
    $$
\end{proof}

We now prove Theorem~\ref{thm:main1}.

\begin{proof}[Proof of Theorem~\ref{thm:main1}]
    The proof is by induction on the number $n$ of singularities in the moduli space $\mathbb{P}\Omega(\underline{k})$.

    Assume that the asymptotic formula~\eqref{equ:asymvolume} holds for all moduli spaces with fewer than $n$ singularities.

    Let $\mathbb{P}\Omega(\underline{k})$ be a moduli space with $n$ singularities. We further prove the formula~\eqref{equ:asymvolume} by reverse induction on the parameter $d$ for the subsets $U_{d,\varepsilon}$.

    Assume that the formula~\eqref{equ:asymvolume} holds for $U_{m,\varepsilon}$ whenever $m > d$. By Lemma~\ref{lem:epsilondecomp}, we have
    \begin{equation}\label{equ:1}
        \mu_{Thu}(U_{d,\varepsilon}) = \sum\limits_{\mathrm{codim}(B) = d}\mu_{Thu}\big(U_{d,\varepsilon}(B_\lambda)\big) + \mu_{Thu}\Big( U_{d,\varepsilon}\setminus\underset{\mathrm{codim}(B)=d}{\bigsqcup}U_{d,\varepsilon}(B_\lambda)\Big),
    \end{equation}
    where the sum is taken over all boundary strata of codimension $d$.

    By Lemma~\ref{lem:epsilondecomp}, we have
    $$U_{d,\varepsilon}\setminus \underset{\mathrm{codim}(B)=d}{\bigsqcup} U_{d,\varepsilon}(B_\lambda)\subseteq U_{d + 1,\frac{6n}{\delta^2} d\varepsilon}.$$
    By the induction hypothesis on $d$, it follows that as $\varepsilon\to 0$, 
    \begin{equation}\label{equ:2}
            \mu_{Thu}\Big( U_{d,\varepsilon}\setminus\underset{\mathrm{codim}(B)=d}{\bigsqcup}U_{d,\varepsilon}(B_\lambda)\Big) = O(\varepsilon^{2d+2}).
    \end{equation}

    By Lemma~\ref{lem:thickvolume}, we obtain
    \begin{equation}\label{equ:3}
        \sum\limits_{\mathrm{codim}(B) = d}\mu_{Thu}\big(U_{d,\varepsilon}(B_\lambda)\big)  = \varepsilon^{2d}\sum\limits_{\mathrm{codim}(B) = d}C(P)\mu_{S}(B_{\lambda}) + O(\varepsilon^{2d + 2}),
    \end{equation}
    where $P$ is the $\underline{k}$-admissible partition associated with $B$.

    Furthermore, by Lemma~\ref{lem:complementofthick}, we know that $B\setminus B_\lambda$ is contained in $U_{1,\lambda/\delta}$, which is defined in the moduli space corresponding to the boundary stratum $B$. The constraint~\eqref{equ:restriction} gives
    $$
    \lambda = \left(1 + \frac{n}{\delta}\right)\varepsilon.
    $$
    Since the number of singularities in $B$ is less than $n$, the induction hypotheses on $n$ and $d$ imply that
    \begin{equation}\label{equ:4}
        \mu_{Thu}(B) = \mu_{Thu}(B_\lambda) + O(\varepsilon^2).
    \end{equation}

    Substituting~\eqref{equ:2}, \eqref{equ:3}, and~\eqref{equ:4} into~\eqref{equ:1}, we complete the proof of~\eqref{equ:asymvolume} for $U_{d,\varepsilon}$ in $\mathbb{P}\Omega(\underline{k})$.

    It remains to prove the base cases for the inductions on $n$ and $d$. 

    For the base case of the induction on $d$, that is, $d = n-3$, Lemma~\ref{lem:thickvolume} gives
    $$
    \mu_{Thu}(U_{n-3,\varepsilon}) =\sum\limits_{\dim B= 0}\mu_{Thu}(U_{n-3,\varepsilon}(B)) = \sum\limits_{\dim B= 0} \varepsilon^{2(n-3)}C(P) + O(\varepsilon^{2n-4})
    $$
    as $\varepsilon \to 0$. This completes the induction on $d$.   
    
    For the base case of the induction on $n$, namely $n = 3$, the moduli space consists of a single element $X_0$. In this case, the measure of $U_{d,\varepsilon}$ tends to zero as $\varepsilon \to 0$, and the formula holds trivially.

    Therefore, the proof of the theorem is complete.
\end{proof}

\subsection{Interpretation of Leading Coefficients}\label{sec:computationconstant}

\begin{lemma}\label{lem:nonsingletonproduct}
Let $\underline{k}$ be a curvature vector with positive curvature gap. Let $B(P)$ be a boundary stratum of $\overline{\mathbb{P}\Omega}(\underline{k})$. For each non-singleton element $p_i\in P$, let $P_i$ be the $\underline{k}$-admissible partition of $\{1,\ldots,n\}$ in which $p_i$ is the only non-singleton element. Then, we have
$$
C(P) = \prod_{i=1}^{s} C(P_i).
$$
\end{lemma}

\begin{proof}
    For each $i$, $\mathbb{P}\Omega(\underline{k}(p_i))$ is the moduli space of convex infinite flat spheres with conical curvatures given by $p_i$. Let
    $$
    [X_i^{b(i)}, F_i^{b(i)}],\quad 1\le b(i)\le q_i
    $$
    be the geometric types of spanning trees in the spheres in $\mathbb{P}\Omega(\underline{k}(p_i))$, as in Lemma~\ref{lem:decompinfinitecase}. Denote by $B_i$ the boundary stratum corresponding to $P_i$.

    By definition~\eqref{equ:svconst}, we have
    \begin{equation}\label{equ:cp1}
        C(P) = \sum_{j=1}^{s}\sum_{b(j)=1}^{q_i} \mu_{D_B}\left(C^{< 1}_{([X^{b(i)}_i, F^{b(i)}_i])_{i=1}^s}\right),
    \end{equation}
    and for each $1\le i\le s$ and each $1\le b(i)\le q_i$,
    $$
    C(P_i) = \sum_{b(j)=1}^{q_i} \mu_{D_{B_i}}\left(C^{< 1}_{[X^{b(i)}_i, F^{b(i)}_i]}\right).
    $$

    Theorem~\ref{thm:fiberedcoordinatechart} implies that
    $$
    C_{([X^{b(i)}_i, F^{b(i)}_i])_{i=1}^s}\simeq C_{[X^{b(1)}_1, F^{b(1)}_1]}\times\ldots\times C_{[X^{b(s)}_s, F^{b(s)}_s]}.
    $$
    Furthermore, for $V=(V^1,\ldots,V^s)\in C_{([X^{b(i)}_i, F^{b(i)}_i])_{i=1}^s}$, by definition~\eqref{equ:defB}, we have
    $$
    D_B(V) = D_{B_1}(V^1) + \ldots + D_{B_s}(V^s).
    $$
    It follows that
    \begin{align*}
        \mu_{D_B}\left(C^{< 1}_{([X^{b(i)}_i, F^{b(i)}_i])_{i=1}^s}\right) &= \int_{C^{< 1}_{([X^{b(i)}_i, F^{b(i)}_i])_{i=1}^s}} d\mu_B\\
        &= \prod_{i=1}^s \int_{C^{< 1}_{[X^{b(i)}_i, F^{b(i)}_i]}} d\mu_{B_i}.
    \end{align*}
    Substituting this into~\eqref{equ:cp1}, we obtain
    $$
    C(P) = \prod_{i=1}^{s} C(P_i).
    $$
\end{proof}

Let $P$ be an $\underline{k}$-admissible partition of $\{1,\ldots,n\}$ with only one non-singleton element $p$, and let $B$ be the corresponding boundary stratum.

Consider the geometric types $[X^b,F^b]$ for $1\le b \le q$ in the moduli space $\mathbb{P}\Omega(\underline{k}(p))$, as given in Lemma~\ref{lem:decompinfinitecase}. Recall that Lemma~\ref{lem:decompinfinitecase} states that the complement of the union
$$
\bigsqcup_{b=1}^q D\Big(\underline{k}(p);[X^b, F^b]\Big)
$$
in the moduli space $\mathbb{P}\Omega(\underline{k}(p))$ consists of spheres that are either not Delaunay-generic or have multiple shortest spanning trees. Here, the subset $D\Big(\underline{k}(p);[X^b, F^b]\Big)$ is defined in Definition~\ref{def:domainofgeometrictype}.

Remark~\ref{rmk:specificprojectiontoinfinitesimalspheres} states that there exists a projection
\begin{equation}\label{equ:localprojection}
    \operatorname{proj}: C_{[X^b,F^b]}\to D\Big(\underline{k}(p);[X^b, F^b]\Big)\subseteq \mathbb{P}\Omega(\underline{k}(p)),
\end{equation}
where a vector $V\in C_{[X^b,F^b]}$ represents an unprojectivized spanning tree parameter of a sphere $X = \operatorname{proj}(V)\in D\Big(\underline{k}(p);[X^b, F^b]\Big)$. Then, for any function $f$ on $\mathbb{P}\Omega(\underline{k}(p))$, we can consider the composition 
$$
f\circ \operatorname{proj}: C_{[X^b,F^b]}\to \mathbb{R}.
$$

\begin{definition}
    For the moduli space $\mathbb{P}\Omega(\underline{k}(p))$, we define the following linear functional 
    \begin{equation}
        f \mapsto \sum_{b=1}^q\int_{V\in C_{[X^b,F^b]}^{<1}} f\big(proj(V)\big) \, d\mu_{D_B}    
    \end{equation}
    on the space of compactly supported continuous functions $C_c\big(\mathbb{P}\Omega(\underline{k}(p))\big)$.    We define the Radon measure  $\mu_P$ on $\mathbb{P}\Omega(\underline{k}(p))$ to be the measure associated with this linear functional.
\end{definition}

Equivalently, the measure $\mu_P$ restricted to each subset $D\big(\underline{k}(p);[X^b, F^b]\big)\subseteq \mathbb{P}\Omega(\underline{k}(p))$ is given by the pushforward measure
$$
\mu_P = (proj^{<1})_{\ast}(\mu_{D_B}),
$$
where $proj^{<1}$ is the restriction of~\eqref{equ:localprojection}:
\begin{equation}
    proj^{<1}: C_{[X^b,F^b]}^{<1} \to D\Big(\underline{k}(p);[X^b, F^b]\Big).
\end{equation}
The measure of the complement (i.e., the subset consisting of spheres that are either not Delaunay-generic or have multiple shortest spanning trees) is zero under $\mu_P$.

\begin{remark}
    Note that the projection~\eqref{equ:localprojection} is the restriction of the canonical projection $\mathbb{C}^{n-2} \to \mathbb{C}P^{n-3}$ to $C_{[X^b,F^b]}$. However, the pushforward measure $(\operatorname{proj})_{\ast}(\mu_{D_B})$ on $D\Big(\underline{k}(p);[X^b, F^b]\Big)\subseteq \mathbb{P}\Omega(\underline{k}(p))$ is not the measure induced by the Fubini–Study metric. 

    This discrepancy arises because the norm on $V$ in $C_{[X^b,F^b]}^{<1}$ is the maximum norm, $\max_{1\leq i\leq n-3}|v_i|$, rather than the Euclidean norm defined by $D_B$.
\end{remark}

\begin{remark}
    By the definition of $\mu_P$, we have
    $$
    C(P) = \mu_P\big(\mathbb{P}\Omega(\underline{k}(p))\big).
    $$    
\end{remark}

The constant $C(P)$ can be computed explicitly in the following case.

\begin{lemma}\label{lem:cod1}
    Let $\underline{k}$ be a curvature vector with positive curvature gap. Let $P$ be an $\underline{k}$-admissible partition of $\{1,\ldots,n\}$ that contains only one non-singleton subset, $p = \{k_{i_1}, k_{i_2}\}$. Then, we have
    $$
    C(P) = \pi(1 - k_{i_1} - k_{i_2}) \frac{\sin{(k_{i_1}\pi)}\sin{(k_{i_2}\pi)}}{\sin{((k_{i_1}+k_{i_2})\pi)}}.
    $$
\end{lemma}

\begin{proof}
    The moduli space $\mathbb{P}\Omega(k_{i_1}, k_{i_2}, 2 - k_{i_1} -k_{i_2})$ consists of a single sphere $X$. Let $x_i$ be the conical singularities on $X$ with curvatures $k_i$, respectively. Assume that $k_{i_1}\le k_{i_2}$. 
    
    Let $\gamma$ be the saddle connection joining $x_{i_1}$ and $x_{i_2}$ in $X$. We first determine the possible geometric types of spanning trees in $X$. There are two cases:
    \begin{enumerate}
        \item If $k_{i_1}\le k_{i_2}\le \frac{1}{2}$, the core $D$ of $X$ degenerates and is precisely the image of the saddle connection $\gamma$. The shortest spanning tree is also $\gamma$, and there is only one geometric type, $[X,\gamma]$.
        \item If $k_{i_1}< \frac{1}{2} < k_{i_2}$, there exists a simple saddle connection $\gamma'$ with the same endpoint $x_{i_1}$. The core $D$ is the finite-area domain enclosed by $\gamma'$. Again, the saddle connection $\gamma$ is the shortest spanning tree, so there is also only one geometric type, $[X,\gamma]$.
    \end{enumerate}

    Next, we compute the infinite cone $C_{[X,\gamma]}$. Since $\mathbb{P}\Omega( k_{i_1}, k_{i_2}, 2 - k_{i_1} -k_{i_2})$ consists of a single sphere $X$, the vector $V\in C_{[X,\gamma]}\subseteq \mathbb{C}$ parameterizes the saddle connection $\gamma$ in $X$. To determine the cone, we need to find the range of the argument of $V$.
    
    Let $B$ be the boundary stratum corresponding to $P$. Let $x'$ be the singularity in the spheres $X_0\in B$ that is associated with the non-singleton element $p\in P$. The curvature of $x'$ is $k_{i_1}+k_{i_2}$. It follows that the sector $\operatorname{Sec}(x')$ used in the construction of the fibered coordinate chart has an angle of $2\pi(1-k_{i_1}-k_{i_2})$. Consequently, the argument $\arg V$ is in the interval $\big(\theta, \theta + 2\pi(1-k_{i_1}-k_{i_2})\big)$ for some $\theta\in\mathbb{R}$. Since the cone $C_{[X,\gamma]}$ is defined up to complex scaling, we can write
    $$
    C_{[X,\gamma]} = \Big\{ V\in\mathbb{C} \mid \arg V \in \big(0, 2\pi(1-k_{i_1}-k_{i_2})\big) \Big\}.
    $$

    We now compute the volume form $d\mu_{D_B}$. Recall that 
    $$
    D_B(X) = \operatorname{Area}(C_{D}) - \operatorname{Area}(D).
    $$
    Since $V$ parameterizes the saddle connection $\gamma$, in both cases, we obtain
    $$
    D_B(V) = \frac{\sin{(k_{i_1}\pi)}\sin{(k_{i_2}\pi)}}{\sin{(k_{i_1}+k_{i_2})\pi}}|V|^2.
    $$
    It follows that
    $$
    d\mu_{D_B} = \frac{\sin{(k_{i_1}\pi)}\sin{(k_{i_2}\pi)}}{\sin{(k_{i_1}+k_{i_2})\pi}}\frac{i}{2} dV d\overline{V}.
    $$

    \begin{figure}[!htbp]
	\centering
	\includegraphics[width=0.9\linewidth]{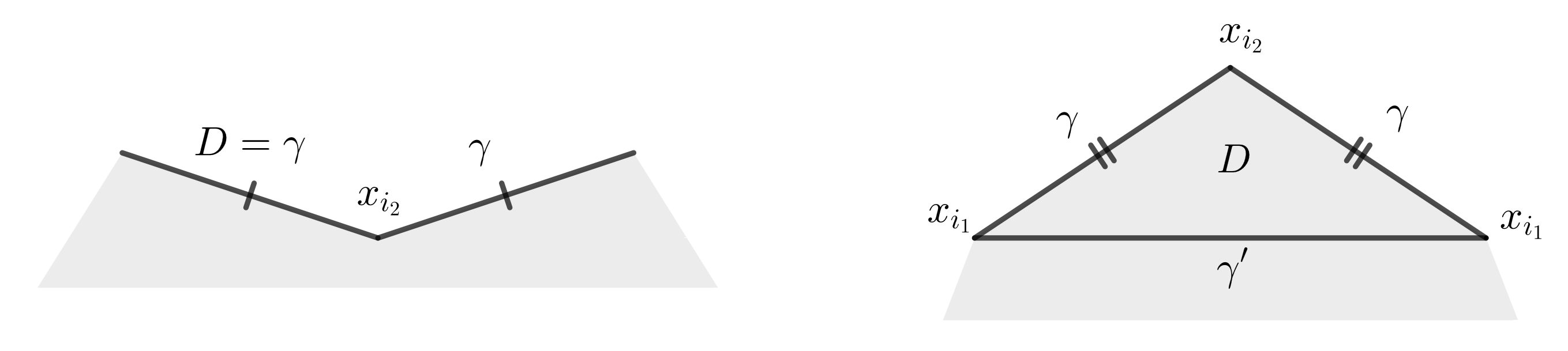}
    \caption{The developing of the convex hull $D$ in the case $k_{i_1}\le k_{i_2}\le \frac{1}{2}$ and the case $k_{i_1}< \frac{1}{2} < k_{i_2}$.}\label{fig:codim1}
    \end{figure}

    Finally, we compute the constant $C(P)$. We have
    $$
    C(P) = \int_{C_{[X,\gamma]}^{<1}}d\mu_{D_B} = \pi(1-k_{i_1}-k_{i_2})\frac{\sin{(k_{i_1}\pi)}\sin{(k_{i_2}\pi)}}{\sin{(k_{i_1}+k_{i_2})\pi}}.
    $$
\end{proof}

\section{Proof of Theorem~\ref{thm:main3}}\label{bigsec:proof}

\subsection{Length of Regular Closed Geodesics}
We recall the following lemma:

\begin{lemma}[{\cite[Proposition~6.3]{fu2024uniformlengthestimatestrajectories}}]\label{lem:lowerboundclosed}
    Let $X$ be an area-one convex flat cone sphere with $n$ singularities and a positive curvature gap $\delta$. Then, for any regular closed geodesic $\gamma$ on $X$, the length of $\gamma$ is bounded below by $\sqrt{\pi\delta}$.
\end{lemma}

The above lemma implies the following:

\begin{lemma}\label{lem:closedsupport}
    Let $\underline k$ be a curvature vector with positive curvature gap. Then, the support of the Siegel–Veech measure $\mu^{cg}_{\underline k}$ is contained in $[\sqrt{\pi\delta},\infty)$.
\end{lemma}

\begin{proof}
    Suppose $f\in C_c(\mathbb{R}_{>0})$ has support contained in $(0,\sqrt{\pi\delta})$. Then, by Lemma~\ref{lem:lowerboundclosed}, we have $    \hat{f}^{cg}(X) = 0$
    for any $X\in \mathbb{P}\Omega(\underline{k})$. It follows that the support of $\mu^{cg}_{\underline k}$ is contained in $[\sqrt{\pi\delta},\infty)$.
\end{proof}

\subsection{Proof of Theorem~\ref{thm:main3}}\label{sec:asymptotic0}

In this section, we prove the asymptotic formula~\eqref{equ:mainasym}.

We define $U_{\varepsilon}$ as the subset of $\mathbb{P}\Omega(\underline k)$ consisting of surfaces $X$ that admit at least one saddle connection with normalized length less than $\varepsilon$. Note that we do not require the saddle connections to have distinct endpoints. By Lemma~\ref{lem:closedtosimple}, we have the inclusion
\begin{equation}\label{equ:inclusion}
    U_\varepsilon \subseteq U_{1, \varepsilon/\delta}.
\end{equation}

\begin{proof}[Proof of Theorem~\ref{thm:main3}]
The first two assertions of the theorem follow from Theorem~\ref{thm:analytic} and Corollary~\ref{cor:abc}. The third assertion is established in Lemma~\ref{lem:closedsupport}.

Now, we proceed to prove the final assertion.

    Approximate the indicator function $\mathbbm{1}_{(0,\varepsilon)}$ on $\mathbb{R}_{>0}$ by a sequence of positive compactly supported smooth functions $(f_m)$ on $\mathbb{R}_{>0}$. Since $\hat{\mathbbm{1}}^{sc}_{(0,\varepsilon)} \in L^\infty(\mathbb{P}\Omega(\underline{k}), \mu_{Thu})$ (see Theorem~\ref{thm:integrabilityandfinitedecomp}), the dominated convergence theorem implies that
    $$\mu^{sc}_{\underline{k}}(0,\varepsilon)
    =\int \mathbbm{1}_{(0,\epsilon)}d\mu^{sc}_{\underline{k}} = \lim_{m\to\infty}\int f_md\mu^{sc}_{\underline{k}}=\lim_{m\to\infty} \int_{\mathbb{P}\Omega(\underline{k})}\hat{f}_m^{sc}d\mu_{Thu} = \int_{\mathbb{P}\Omega(\underline{k})}\hat{\mathbbm{1}}^{sc}_{(0,\varepsilon)}(X)d\mu_{Thu}.$$
    Combining~\eqref{equ:inclusion}, we have that
    $$
    \mu^{sc}_{\underline{k}}(0,\varepsilon)
    = \int_{U_{\varepsilon}}\hat{\mathbbm{1}}^{sc}_{(0,\varepsilon)}(X)d\mu_{Thu}
    =\int_{U_{1,\varepsilon'}}\hat{\mathbbm{1}}^{sc}_{(0,\varepsilon)}(X)d\mu_{Thu},
    $$
    as $\varepsilon\to 0$, where $\varepsilon' = \varepsilon/\delta$.
    
    By Theorem~\ref{lem:epsilondecomp} and Theorem~\ref{thm:main1}, we have that
    $$
    \mu_{Thu}\big(U_{1,\varepsilon'}\setminus\sqcup_{\mathrm{codim}(B) =1} U_{1,\varepsilon'}(B_\lambda)\big) = O(\varepsilon^4),
    $$
    where $\varepsilon'$ and $\lambda$ satisfy the condition~\eqref{equ:restriction}.
    Combining Theorem~\ref{thm:integrabilityandfinitedecomp}, we obtain
    \begin{equation}\label{equ:asyintermediate}
        \mu^{sc}_{\underline{k}}(0,\varepsilon) = \sum\limits_{\mathrm{codim}(B)=1}\int_{ U_{1,\varepsilon'}(B_\lambda)}\hat{\mathbbm{1}}^{sc}_{(0,\varepsilon)}(X)d\mu_{Thu} + O(\varepsilon^4).
    \end{equation}

    Let $B$ be a codimension-$1$ boundary stratum, and let $P$ be the corresponding $\underline{k}$-admissible partition. Note that $P$ has only one non-singleton element, and it is written as $\{i_1.i_2\}$. Assume that $k_{i_1}\le k_{i_2}$. Let $X$ be a surface in $U_{1,\varepsilon'}(B_\lambda)$ with area one. Then,
    $$
    \hat{\mathbbm{1}}^{sc}_{(0,\varepsilon)}(X) = N^{sc}(X,\varepsilon),
    $$
    where $N^{sc}(X,\varepsilon)$ denotes the number of saddle connections in $X$ with metric length less than $\varepsilon$. By Lemma~\ref{lem:epsilonconvexhulls}, any saddle connection in $X$ must be contained in the $(1,\varepsilon')$-convex hull $D$ of $X$. Let $\gamma$ be the shortest saddle connection between $x_{i_1}$ and $x_{i_2}$ in $X$. Then, we have $|\gamma|\leq \varepsilon'$.

    \textbf{Case~1}: If $k_{i_1} \leq k_{i_2} \leq \frac{1}{2}$, the convex hull $D$ degenerates and coincides with the image of $\gamma$. It follows that $\gamma$ is the only simple saddle connection in $D$. Thus, 
    \begin{equation}\label{equ:countingcase1}
            N^{sc}(X,\varepsilon) = 1,\quad\text{for }X\in U_{1,\varepsilon}(B_\lambda)\subseteq U_{1,\varepsilon'}(B_\lambda).
    \end{equation}
    Consequently, we have
    $$
    \int_{ U_{1,\varepsilon'}(B_\lambda)}\hat{\mathbbm{1}}_{(0,\varepsilon)}(X)d\mu_{Thu} = \mu_{Thu}(U_{1,\varepsilon}(B_\lambda)).
    $$ 
    Combining this with Theorem~\ref{thm:main1}, we obtain
    \begin{equation}
        \int_{ U_{1,\varepsilon'}(B_\lambda)}\hat{\mathbbm{1}}^{sc}_{(0,\varepsilon)}(X)d\mu_{Thu} = C(P)\mu_{Thu}(B)\varepsilon^2 + O(\varepsilon^4),
    \end{equation}
    as $\varepsilon\to0$, where
    the constant $C(P)$ is as in Theorem~\ref{thm:main1} and was computed in Lemma~\ref{lem:cod1}.
    
    \textbf{Case~2}: If $k_{i_1} < \frac{1}{2} < k_{i_2}$, the $\varepsilon'$-convex hull of $x_{i_1}$ and $x_{i_2}$ is enclosed by a simple saddle connection $\gamma'$ with the same endpoints as $x_{i_1}$; see Figure~\ref{fig:codim1}.

    \begin{figure}[!htbp]
	\centering
	\includegraphics[width=0.6\linewidth]{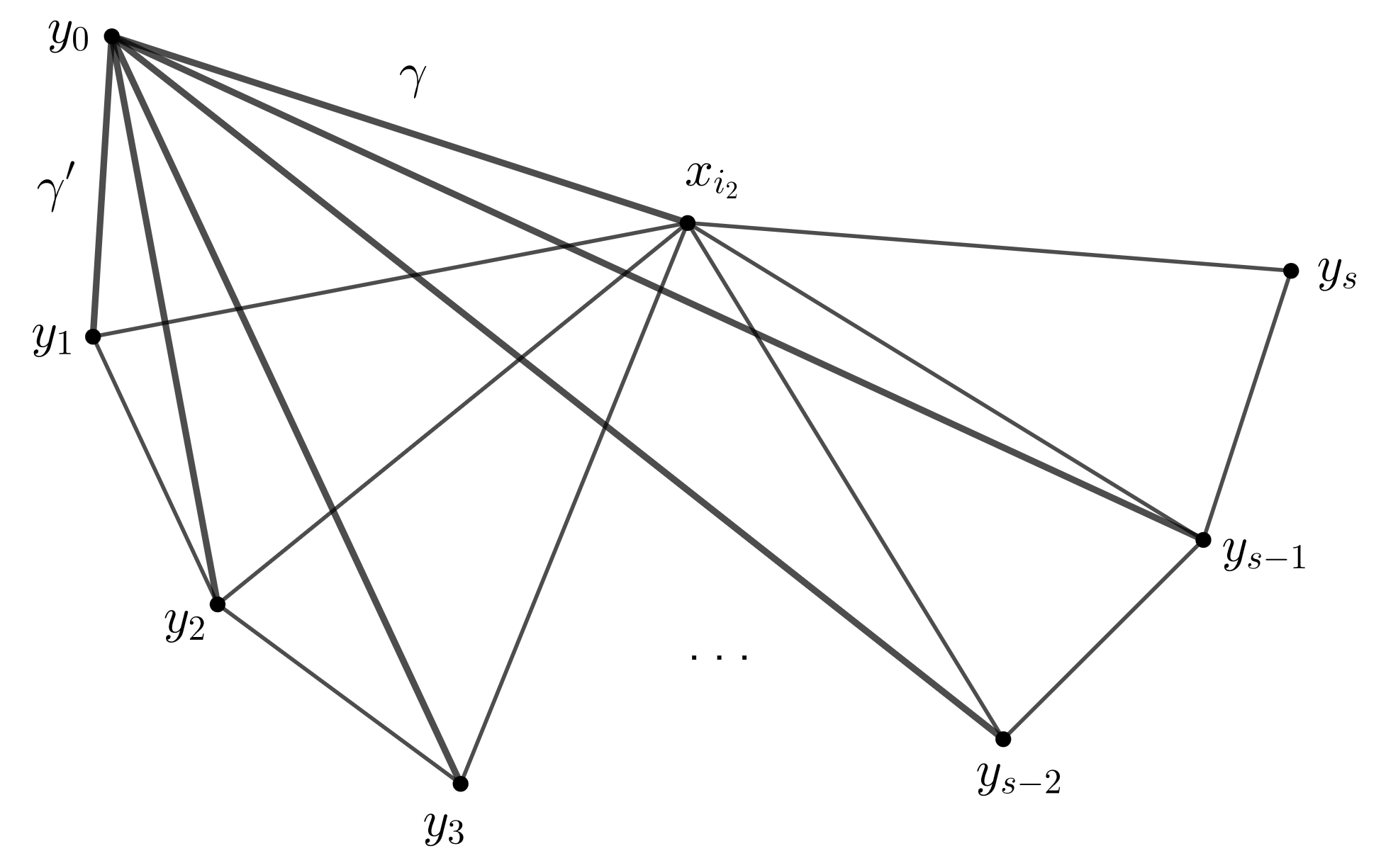}
    \caption{The developing of the convex hull $D$ and the saddle connections in $D$.}\label{fig:multishort}
    \end{figure}

    The convex hull $D$ can be unfolded into an isosceles triangle $y_0x_{i_2}y_1$ in the plane. Identifying the sides $x_{i_2}y_0$ and $x_{i_2}y_1$ glues the triangle back into $D$. In this setting, the segment $y_0y_1$ corresponds to $\gamma'$, while the sides $x_{i_2}y_0$ and $x_{i_2}y_1$ correspond to two copies of $\gamma$. The additional triangles $y_{i-1}x_{i_2}y_{i}$ are obtained via rotations about $x_{i_2}$. 

    Define
    $s:=\Big\lfloor \frac{1}{2(1-k_{i_2})} \Big\rfloor.$
    Note that we can join the vertices $y_0$ to $y_j$ by a segment contained in the union of the triangles $y_{i-1}x_{i_2}y_{i}$ for $j=1,\ldots,s$. Moreover, the length of each segment $y_0y_j$ is given by
    \begin{equation}\label{equ:multipleshortlength}
        |y_0y_j| = 2\sin (jk_{i_2}\pi)|x_{i_2}y_0|.
    \end{equation}

    \underline{Case~2.1}: If there exists $m\in\mathbb{N}$ such that $k_{i_2} = 1-\frac{1}{2m}$, then $s = m$, and the segment $y_0y_m$ passes through $x_{i_2}$. Thus, under the developing, the saddle connections in $D$ correspond to $x_{i_2}y_0$ and the segments $y_0y_j$ for $1\le j\le m-1$. By~\eqref{equ:multipleshortlength}, if $X\in U_{1,\varepsilon/2\sin (jk_{i_2}\pi)}(B_{\lambda})$, then
    the length of the saddle connection corresponding to $y_0y_j$ satisfies that 
    $|y_0y_j| = 2\sin (jk_{i_2}\pi)|\gamma|<\varepsilon.$
    It follows that $y_0y_j$ contributes to a saddle connection in $X$ of length less than $\varepsilon$.
    Therefore, 
    \begin{align*}
        \int_{U_{1,\varepsilon'}(B_\lambda)}\hat{\mathbbm{1}}_{(0,\varepsilon)}(X)d\mu_{Thu} &= \mu_{Thu}(U_{1,\varepsilon}(B_\lambda)) + \sum\limits_{j=1}^{m - 1}\mu_{Thu}(U_{1,\varepsilon/2\sin (jk_{i_2}\pi)}(B_\lambda))\\
        &= \left[1+\sum\limits_{j=1}^{m - 1}\Big(\frac{1}{2\sin jk_{i_2}\pi}\Big)^2\right]C(P)\mu_{Thu}(B)\varepsilon^2 + O(\varepsilon^4).
    \end{align*}

    \underline{Case~2.2}: If 
    $k_{i_2} \notin \left\{ 1 - \frac{1}{2m} \right\}_{m \in \mathbb{N}},$
    then the saddle connections in $D$ are induced by the segment $x_{i_2}y_0$ and the segments $y_0y_j$ for $1\le j\le s$. A similar computation yields
    \begin{align*}
        \int_{U_{1,\varepsilon'}(B_\lambda)}\hat{\mathbbm{1}}_{(0,\varepsilon)}(X)d\mu_{Thu} &= \left[1+\sum\limits_{j=1}^{s}\Big(\frac{1}{2\sin jk_{i_2}\pi}\Big)^2\right]C(P)\mu_{Thu}(B)\varepsilon^2 + O(\varepsilon^4).
    \end{align*}

    Combining~\eqref{equ:asyintermediate}, we complete the proof.
\end{proof}

\bibliographystyle{amsalpha}
\bibliography{ref}

\end{document}